\documentclass[12pt, a4paper]{article}
\usepackage[margin=1in]{geometry}
\usepackage{graphicx}

\usepackage{graphicx}
\usepackage{amsthm, amsfonts,amssymb, amsmath}
\usepackage{lmodern}
\usepackage{mathtools}
\numberwithin{equation}{section}
\usepackage{xcolor}
\newtheorem{thm}{Theorem}[section]
\newtheorem{lem}[thm]{Lemma}

\newtheorem{cor}[thm]{Corollary}

\newtheorem{mydef}{Definition}[section]

\newtheorem{rem}{Remark}[section]
\theoremstyle{remark}
\newcommand\bb[1]{\mathbf{#1}}
\newcommand\ddfrac[2]{\frac{\displaystyle #1}{\displaystyle #2}}
\usepackage[T1]{fontenc}
\usepackage[utf8]{inputenc}
\newcommand\bint[1]{\displaystyle\int #1}

\usepackage[font=footnotesize,labelfont=bf]{caption}

\begin{document}

\author{Sr\dj{}an Trifunovi\'c  \thanks{School of Mathematical Sciences,
Shanghai Jiao Tong University, Shanghai 200240, China,
 email: sergej1922@gmail.com, tarathis@sjtu.edu.cn} \and Ya-Guang Wang\thanks{School of Mathematical Sciences, Center of Applied Mathematics,
  MOE-LSC and SHL-MAC, Shanghai Jiao Tong University, Shanghai 200240,  China,
 email: ygwang@sjtu.edu.cn} }

\title{On the interaction problem between a compressible viscous fluid and a nonlinear thermoelastic plate}

\maketitle
\abstract{In this paper we study the interaction problem between a nonlinear thermoelastic plate and a compressible viscous fluid with the adiabatic constant $\gamma>12/7$. The existence of a weak solution for this problem is obtained by constructing a time-continuous operator splitting scheme that decouples the fluid and the structure. The fluid sub-problem is given on a fixed reference domain in the arbitrary Lagrangian-Eulerian (ALE) formulation, and the continuity equation is damped on this domain as well. This allows the majority of the analysis to be performed on the fixed reference domain, while the convergence of the approximate pressure is obtained on the physical domain.}
\\ \\
{\footnotesize \textbf{Keywords and phrases:} {fluid-structure interaction, compressible viscous fluid, nonlinear thermoelastic plate, three space variables, weak solution}
\\ \\
{\footnotesize \textbf{AMS Mathematical Subject classification (2020):} {35Q30, 35M13, 35D30, 74F05, 74F10}

\tableofcontents
\section{Introduction}
The area of fluid-structure interaction (FSI) spans over mathematics, physics, engineering, biomedicine etc. It considers interaction problems between various types of fluids and rigid bodies or elastic bodies/shells/plates. The mathematical theory of FSI has developed quite significantly over the recent years. Here we only mention the results closely related to the model we will study, in particular, on the interaction problems between fluids and elastic structures (plates or shells) located at the (part or whole) boundary of the fluid domain. 

Desjardins et al. \cite{time} obtained a first weak solution existence result for the interaction problem between an incompressible viscous fluid and a viscoelastic structure. Then, Grandmont improved this result by obtaining the weak solution when the structure is purely elastic \cite{grandmont3}. R\r{u}\v{z}i\v{c}ka and Lengeler \cite{ruzicka} obtained a weak solution for an incompressible viscous fluid and elastic shell interaction model, where the shell is a regular manifold that deforms in its normal direction. Muha and \v{C}ani\'c developed a time discretization via operator splitting decoupling numerical scheme for constructing the weak solutions for the incompressible viscous fluid and elastic plate/shell interaction model under various cases in \cite{BorSun,BorSunNonLinear,BorSunNavierSlip}. Later, the authors considered in \cite{trwa} a general semilinear plate model that generalizes the Kirchhoff, von Karman and Berger plates\footnote{In the present paper, we also consider this model (see the assumptions (A1) and (A2) in Section $\ref{qq22sec1.2}$).} and constructed a hybrid splitting scheme (stationary for the fluid and time-continuous in a finite base for the plate) in order to deal with the general form of the nonlinearity in the plate equation. Recently, we extended this result in \cite{trwa2} to the problem where a thermoelastic semilinear/quasilinear\footnote{The quasilinear plate model corresponds to a case where the nonlinearity is cubic. The same model is also considered in this paper (see remark $\ref{qq22quasil}$).} plate interacted with an incompressible viscous fluid. Muha and Schwarzacher \cite{muhasch} proved the existence of a weak solution for the interaction problem of a nonlinear (quasilinear) Koiter shell and incompressible viscous fluid. Here, the convergence of the approximate nonlinear elastic force in the structure equation was obtained by proving the (better than energy) regularity $L_t^2 H_x^{2+s}$, $s<1/2$, of the structure displacement, by utilizing the dissipation effects of the fluid onto the structure.   In \cite{3dmesh}, contrary to other literature, a shell with 3D displacement was considered in interaction with viscous incompressible fluid and a mesh of elastic curved rods modeling stents, thus constituting a 1D-2D-3D nonlinearly coupled fluid-structure interaction problem. It is important to note that the behavior of such a shell cannot be controlled properly by using only the energy estimates. Thus, the authors construct a weak solution based on the time semi-discretization and operator splitting approach for this problem, under certain assumptions for the approximate shell displacement which ensure that it is regular enough and doesn't self-intersect on some time interval. In \cite{contact}, global weak solutions for 2D interaction problem between an incompressible viscous fluid and an elastic beam with possible contact were constructed, as a limit of a sequence of strong global solutions constructed in \cite{grandmont1} as viscoelasticity coefficient goes to zero. However, the contact mechanism was not prescribed. Recently, Schwarzacher and She \cite{schshe} proposed a monolithic numerical scheme for the interaction problem of a compressible viscous fluid and an elastic plate and studied its stability and consistency. We also mention a weak-strong uniqueness result obtained in \cite{WSinc} for the interaction problem of an incompressible viscous fluid and an elastic structure. 

In the context of strong solutions for the problem of incompressible viscous fluid and a viscoelastic structures, the first such result was due to Beir\~{a}o da Veiga in \cite{strongzero}, where a local possibly non-unique strong solution was obtained in 2D case for small initial data. Later, Lequerre extended this result to a global strong solution for small initial data in \cite{strong} in 2D. In \cite{grandmont1}, Grandmont et al. obtained a global solution for a 2D model with viscoelastic structure by proving that no collision between the beam and the bottom of the cavity occurs. In \cite{grandmont2}, contrary to the above mentioned work, a local 2D strong solution was constructed for the problem with a purely elastic structure. Mitra considered a 2D model where a viscoelastic beam interacts with a viscous compressible fluid and obtained a regular solution in \cite{sourav}. We also mention very recent results for the interaction problem between the full Navier-Stokes-Fourier system and a viscoelastic plate in 3D (\cite{NSFFSI}), and the interaction problem between a compressible viscous fluid and a wave equation in 3D (\cite{comstrong}).

Finally, we state the work by Chueshov \cite{igor,igor2} where the stability for the interaction problem of a semilinear plate model and a linearized compressible/inviscid, respectively, fluid were considered, and the work by Avalos et al. \cite{linint1,linint2} where the stability of the linear interaction problem between an elastic plate and a linearized (around arbitrary stationary state) compressible fluid was studied.

In this paper, we aim to study the existence of a weak solution for the interaction problem between a compressible viscous flow and a thermoelastic plate, by constructing a novel decoupling scheme (first such in the compressible case) that splits the fluid and the structure. This scheme was inspired by schemes in \cite{BorSun,BorSunNonLinear,BorSunNavierSlip,trwa,trwa2}, which are used to construct weak solutions for the incompressible case. However, unlike in the incompressible case, here the both approximate sub-problems, corresponding to the fluid and the structure, are constructed to be continuous in time. In this way, the nature of the fluid and the structure sub-problems is preserved almost completely compared to the corresponding fluid and structure systems. The approximate fluid sub-problem is formulated on a fixed reference domain by means of arbitrary Lagrangian-Eulerian (ALE) mappings. We construct a special artificial density damping for the continuity equation which allows us to perform the majority of the analysis on this fixed reference domain. However, the convergence of the approximate pressure (which is the most difficult part of the convergence) is proved on the physical domain as it relies on the inverse divergence operator.

This paper is organized as follows. In section 2, we introduce our model, define the notion of weak solution both on physical and fixed reference domains, and state the main result. In section 3,  we introduce the approximate problems, and obtain the uniform energy estimates. In section 4,  we study the operator splitting time step and finite Galerkin bases limits, and prove the convergence of the approximate solutions in suitable spaces. In section 5, we study the vanishing artificial density viscosity limit, and finally in section 6, the vanishing pressure limit, fixed reference domain limit and the structure regularization limit are studied to obtain the existence of weak solutions to this interaction problem.

\section{Preliminaries and the main result}
In this section, we will first describe the model and derive the energy identity for the classical solutions, if they exist. After that, we derive the problem in the weak form, give the definition of weak solutions and state the main result. At the end of the section, we introduce the equivalent formulation of the same weak solutions on the fixed reference domain.
\subsection{The model description}
Here we study the compressible, viscous fluid interacting
with a thermoelastic plate. The vertical plate displacement is described by a scalar function $w: \Gamma\to \mathbb{R}$, where $\Gamma \subset \mathbb{R}$ is a connected bounded domain with a Lipschitz boundary. The temperature of the plate is denoted by $\theta: \Gamma\to \mathbb{R}$. The fluid fills the domain defined as (see Figure $\ref{qq22figure1}$)
\begin{eqnarray*}
\Omega^w(t) := \{ (X,z) : X \in \Gamma, -1< z < w(t,X)\}.
\end{eqnarray*}
Denote the graph of $w$ by $\Gamma^w (t) = \{ (X,z) : X \in \Gamma, z=w(t, X) \}$ and the side wall of the domain by $W = \{(X,z): X \in \partial \Gamma, -1<z<0\}$, where the plate boundary is assumed to be fixed at $z=0$ for all $x \in \partial \Gamma$. The entire rigid part of the boundary $\partial \Omega^w(t)$ will be denoted as $\Sigma = (\Gamma \times \{-1\}) \cup W $.\\

\begin{figure}[!h]
\begin{center}
\includegraphics[scale=0.7]{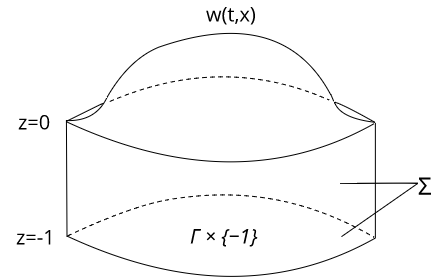}
\caption{The domain $\Omega^w(t)$ determined by the displacement $w(t,X)$ and the rigid part of the boundary $\Sigma$.}
\label{qq22figure1}
\end{center}

\end{figure}

The problem we will study reads:
\vskip 0.2in

\noindent
\framebox{
\parbox{6.0in}{\hskip 0.15in 

\noindent 
Find $( \rho, \bb{u},w, \theta)$ such that the following holds:\\

\noindent
\textbf{The thermoelastic structure equations} in $(0,T)\times \Gamma$:
\begin{equation}\label{qq22structureeqs}
\begin{array}{rcl}
\partial_t^2 w+\Delta^2 w +\Delta \theta +\mathcal{F}(w)&=&-S^w \bb{f}_{fl} \cdot \mathbf{e_3}\\
\theta_t - \Delta \theta - \Delta w_t &=& 0
\end{array}
\end{equation}
\textbf{The compressible viscous fluid equations} in $(0,T)\times \Omega^w(t)$:
\begin{equation}\label{qq22fluideqs}
\begin{array}{rcl}
\partial_t (\rho\mathbf{u}) + \nabla \cdot (\rho\mathbf{u}\otimes \mathbf{u})& =& -\nabla p(\rho) +\mu \Delta\bb{u}+(\mu+\lambda) \nabla (\nabla \cdot \bb{u})I\\
\partial_t \rho + \nabla \cdot (\rho \mathbf{u}) &=& 0 
\end{array}
\end{equation}
\textbf{The fluid-structure coupling (kinematic and dynamic, resp.)} on $(0,T) \times \Gamma$:
\begin{eqnarray}%%%%%%%%%%%%%
\partial_t w(t,X) \mathbf{e_3}&=& \mathbf{u}(t,X,w(t,X)),\label{qq22coupling}\\
\bb{f}_{fl}(t,X):&=&\big[\big(-p(\rho)I+\mu\nabla \bb{u}+(\mu+\lambda)(\nabla\cdot \bb{u})I\big)\nu^w\big](t,X,w(t,X)).\quad \label{qq22coupling2}
\end{eqnarray}%%----------------------------%%
\textbf{The boundary conditions:}
\begin{eqnarray}\label{qq22boundaryconditions}
\begin{aligned}
w(t, x)=\partial_\nu w(t, x)&=0, \text{ on } (0,T)\times \partial\Gamma,\\
\theta &=0, \text{ on } (0,T)\times \partial\Gamma, \\
\mathbf{u}&=0, ~~\text{on}~~ (0,T) \times \Sigma;
\end{aligned}
\end{eqnarray}
\vskip 0.1in
\noindent
\textbf{The initial data:}
\begin{eqnarray}\label{qq22initialdata}
\rho(0,\cdot) = \rho_0, ~(\rho\mathbf{u})(0, \cdot)=(\rho\bb{u})_0,~ w(0,\cdot) = w_0,~\partial_t w(0,\cdot) = v_0, ~ \theta(0,\cdot) = \theta_0.
\end{eqnarray}
}}
\vskip 0.1in
Here, $\mathcal{F}$ is a nonlinear function corresponding the nonlinear elastic force in various plate models (see assumptions (A1) and (A2) below), $S^w(t,X)$ is the Jacobian of the transformation from the Eulerian to the Lagrangian coordinates of the plate
\begin{eqnarray*}
S^w(t,X)=\sqrt{1+\partial_x w(t,X)^2+\partial_y
w(t,X)^2},
\end{eqnarray*}
$\nu^w$ is the unit normal vector on $\Gamma^w$, $p$ is the pressure given by the $\gamma$-law $p(\rho) = \rho^\gamma$ with $\gamma>12/7$, $\mu>0$ and\footnote{Here, we choose $\lambda + \frac{2}{3}\mu$ to be strictly positive as in \cite{compressible}, following the reasoning given in \cite[Remark 1.3]{compressible}.} $\lambda + \frac{2}{3}\mu> 0$, $\bb{e}_3 = (0,0,1)$ and $\nu$ is the normal vector on $\partial\Gamma$. The initial data given in $\eqref{qq22initialdata}$ is assumed to satisfy the following compatibility conditions:\

\begin{eqnarray}\label{qq22compatibilityconditions} %%%%%%%%%%%%%
\begin{aligned}
\rho_0>0,& \text{ in } \{(X,z)\in \Omega^{w_0}: (\rho\bb{u})_0(X,z) >0\},  \\
\frac{(\rho \bb{u})_0^2}{\rho_0} &\in L^1(\Omega^{w_0}),\\
\partial_\nu w_0=w_0 =0, &\text{ on } \partial \Gamma, \\
w_0>-1, &\text{ on } \Gamma.
\end{aligned}
\end{eqnarray}%%----------------------------%%

\subsection{Formulation of the weak solution and the main result}\label{qq22sec1.2}
Denote by
\begin{eqnarray*}%%%%%%%%%%%%%
&\Omega_\Gamma^w(t):=\Omega^w(t)\cup\Gamma^w(t), \quad Q_T^w:=[0,T]\times \Omega^w(t), \quad Q_{T,\Gamma}^w:=[0,T]\times \Omega_\Gamma^w(t),&
\end{eqnarray*}%%----------------------------%%
and
\begin{eqnarray*}%%%%%%%%%%%%%
& \Gamma_T^w: = [0,T]\times \Gamma^w(t),\quad \Gamma_T:= [0,T]\times \Gamma.&
\end{eqnarray*}%%----------------------------%%
We start with introducing the following two assumptions on the nonlinear elastic force $\mathcal{F}(w)$ which appears in the structure equation $\eqref{qq22structureeqs}_1$:

\begin{enumerate}
\item[(A1)] The mapping $\mathcal{F}$ is locally Lipschitz from $H_0^{2-\epsilon}(\Gamma)$ into $H^{-2}(\Gamma)$ for some $\epsilon >0$, i.e.
\begin{eqnarray*}
||\mathcal{F}(w_1)-\mathcal{F}(w_2)||_{H^{-2}(\Gamma)} \leq
C_R || w_1 -w_2 ||_{H^{2-\epsilon}(\Gamma)},
\end{eqnarray*}
for a constant $C_R>0$, for any $||w_i||_{H^{2-\epsilon}(\Gamma)} \leq R$ ($i=1,2$). 
\item[(A2)] $\mathcal{F}(w)$ has a potential
in $H_0^2(\Gamma)$, i.e. there exists a Fr\'{e}chet
differentiable functional $\Pi(w)$ on $H_0^2(\Gamma)$ such that
$\Pi'(w)=\mathcal{F}(w)$ in $H^{-2}(\Gamma)$, and there are $0<\kappa <1/2$ and $C^* \geq 0$, such that the following inequality holds,
\begin{eqnarray*}
\kappa || \Delta
w||_{L^2(\Gamma)}^2+\Pi(w)+ C^*\geq 0, ~~ \text{for all } w \in H_0^2(\Gamma).
\end{eqnarray*}
Moreover, the potential $\Pi(w)$ is bounded on any bounded set of $H_0^2(\Gamma)$.
\end{enumerate}
These assumptions are satisfied by the Kirchhoff, von Karman and Berger plates. There is a vast literature dedicated to these plate models (see \cite{berger,karmanplates,platesproofs} and the references therein). We also mention a semilinear Koiter shell model studied in \cite{BorSunNonLinear} which also satisfies these assumptions\\

We proceed to derive the weak formulation of the problem $\eqref{qq22structureeqs}$-$\eqref{qq22compatibilityconditions}$ for smooth solutions. First, by multiplying the continuity equation $\eqref{qq22fluideqs}_2$ by a function $\varphi \in C^\infty([0,T]\times \overline{\Omega^w(t)})$ and integrating over $Q_T^w$, we obtain
\begin{eqnarray}\label{qq22lokvanj}%%%%%%%%%%%%%
0=\int_{Q_T^w} \Big[\partial_t \rho \varphi + \nabla\cdot (\rho \bb{u}) \varphi \Big]&=&\int_{Q_T^w}\Big[ \frac{d}{dt}(\rho \varphi)-\rho \partial_t \varphi-\rho \bb{u} \cdot \nabla \varphi\Big]+\int_{\Gamma_T^w} \rho \bb{u} \cdot\nu^w \varphi, \quad \quad
\end{eqnarray}%%----------------------------%%
by integration by parts. Now from the Raynolds transport theorem, it follows
\begin{eqnarray}\label{qq22RTT}
\frac{d}{dt} \int_{\Omega^w(t)} \rho \varphi = \int_{\Omega^w(t)} \partial_t (\rho \varphi)+ \int_{ \Gamma^w(t)} \rho\varphi \partial_t w^w \bb{e}_3 \cdot \nu^w, 
\end{eqnarray}
where $w^w(t,X,z) := w(t,X)$, so by using the coupling condition $\eqref{qq22coupling}$, from $\eqref{qq22lokvanj}$ we have
\begin{eqnarray*}%%%%%%%%%%%%%
\int_{Q_T^w}\Big[ \rho \partial_t \varphi + \rho \bb{u} \cdot\nabla \varphi \Big] = \int_0^T \frac{d}{dt} \int_{\Omega^w(t)} \rho \varphi .
\end{eqnarray*}%%----------------------------%%
Next, from the Raynolds transport theorem, for any $\bb{q}\in C_0^\infty (Q_{T,\Gamma})$, it follows
\begin{eqnarray}\label{qq22RTT2}%%%%%%%%%%%%%
\frac{d}{dt} \int_{\Omega^w(t)} \rho \bb{u} \cdot\bb{q} = \int_{\Omega^w(t)} \partial_t (\rho \bb{u}\cdot\bb{q})+ \int_{ \Gamma^w(t)} \big(\rho \bb{u}\cdot \bb{q}\big) \big(\partial_t w^w \bb{q} \bb{e}_3 \nu^w\big),
\end{eqnarray}%%----------------------------%%
so by multiplying the momentum equation $\eqref{qq22fluideqs}_2$ by $\bb{q}\in C_0^\infty (Q_{T,\Gamma})$ and integrating over $Q_T$
\begin{eqnarray}%%%%%%%%%%%%%
0&=& \int_{Q_T}\big[ \partial_t (\rho\mathbf{u}) + \nabla \cdot (\rho\mathbf{u}\otimes \mathbf{u})+\nabla p(\rho) -\mu \Delta\bb{u}-(\mu+\lambda) \nabla \text{div } \bb{u}\big]\bb{q}\nonumber\\
&=&\int_0^T \frac{d}{dt} \int_{\Omega^w(t)}\rho\mathbf{u}\cdot\bb{q} - \int_{Q_T} \rho \bb{u} \cdot \partial_t \bb{q} \underbrace{- \int_{\Gamma_T^w} (\rho\bb{u}\cdot\bb{q}) (\partial_t w^w \bb{e}_3 \cdot \nu^w) + \int_{Q_T} (\rho\bb{u} \cdot \bb{q}) (\bb{u}\cdot \nu^w)}_{=0 \text{ by } \eqref{qq22coupling}} \nonumber\\
&&- \int_{Q_T^w}\big[ (\rho \bb{u} \otimes \bb{u}): \nabla \bb{q} - p(\rho) (\nabla \cdot \bb{q}) + 
\mu \nabla \bb{u}: \nabla \bb{q}
+ (\mu+\lambda)(\nabla \cdot \bb{u}) (\nabla \cdot \bb{q})\big]\nonumber \\
&&-\int_{\Gamma_T^w} \big[ -p(\rho)I+ \mu \nabla \bb{u} + (\mu + \lambda)(\nabla \cdot \bb{u})I \big] \nu^w\cdot \bb{q} .~~~~~~~~~~ \label{qq22idkidk1}
\end{eqnarray}%%----------------------------%%
Next, we multiply the equation $\eqref{qq22structureeqs}_1$ by $\psi \in C_0^\infty(\Gamma_T)$ and integrate over $\Gamma_T$ to obtain
\begin{eqnarray}%%%%%%%%%%%%%
0 &=&\int_{\Gamma_T} \Big[ \partial_t^2 w+\Delta^2 w +\mathcal{F}(w)+\Delta \theta + S^w \bb{f}_{fl}\cdot \bb{e}_3\Big]\psi \nonumber\\
&=& \int_0^T \frac{d}{dt} \int_\Gamma \partial_t w \psi + \int_{\Gamma_T} \Big[ - \partial_t w \partial_t \psi + \Delta w \Delta \psi + \mathcal{F}(w) \psi - \nabla \theta \cdot\nabla \psi \Big]\nonumber\\
&&+\int_{\Gamma_T^w} \big[ -p(\rho)I+ \mu \nabla \bb{u} + (\mu + \lambda)(\nabla \cdot \bb{u})I \big] \nu^w \cdot (\psi\mathbf{e}_3).~~~~~~~~ \label{qq22idkidk2}
\end{eqnarray}%%----------------------------%%
from $\eqref{qq22coupling2}$. By summing up $\eqref{qq22idkidk1}$ and $\eqref{qq22idkidk2}$, using the boundary condition $\eqref{qq22coupling}$ and choosing $\bb{q} \in C_0^\infty(Q_{T,\Gamma})$ and $\psi \in C_0^\infty(\Gamma_T)$ such that $\bb{q}(t,X,w(t,X))= \psi(t,X)\bb{e}_3$ for all $(t,X)\in \Gamma_T$, we have
\begin{eqnarray*}%%%%%%%%%%%%%
&&\int_{Q_T^w}\Big[ -\rho \bb{u}\cdot \partial_t \bb{q} - (\rho \bb{u} \otimes \bb{u}):\nabla \bb{q} - p(\rho) (\nabla \cdot \bb{q})+ \mu \nabla \bb{u} : \nabla \bb{q}+ (\mu+\lambda) (\nabla \cdot \bb{u})( \nabla \cdot \bb{q})\Big]\\
&&+\int_{\Gamma_T} \big[ - \partial_t w \partial_t \psi + \Delta w \Delta \psi + \mathcal{F}(w) \psi - \nabla \theta \cdot \nabla \psi \big]\\
&&=-\int_0^T \frac{d}{dt} \int_{\Omega^w(t)}\rho\mathbf{u}\cdot\bb{q} - \int_0^T \frac{d}{dt} \int_\Gamma \partial_t w \psi. 
\end{eqnarray*}%%----------------------------%%
To introduce a reasonable solution space for the weak solution of $\eqref{qq22structureeqs}-\eqref{qq22initialdata}$, let us derive the energy for smooth solutions of the problem of $\eqref{qq22structureeqs}-\eqref{qq22initialdata}$ in the following way. Multiplying $\eqref{qq22structureeqs}_1$ and $\eqref{qq22structureeqs}_2$ by $\partial_t w$ and $\theta$, respectively, integrating over $\Gamma$, and multiplying the equation $\eqref{qq22fluideqs}_1$ and $\eqref{qq22fluideqs}_2$ by $\bb{u}$ and $\frac{\rho^{\gamma-1}}{\gamma-1}$, respectively, and integrating over $\Omega^w(t)$, then summing up these four identities, integrating over $(0,T)$, and using $\eqref{qq22RTT}$, $\eqref{qq22RTT2}$, the boundary conditions given in $\eqref{qq22boundaryconditions}$ and the identity $\frac{d}{dt} \Pi(w)= (\mathcal{F}(w),\partial_t w)$, we obtain:
\begin{eqnarray}\label{qq22eneq}
E(t) + \Pi(w(t))+ D(t)= E(0) + \Pi(w(0)),
\end{eqnarray}
where
\begin{eqnarray}\label{qq22energies}
\begin{aligned}
&E(t) := F(t)+D(t), ~~D(t) := DF(t) + DS(t), \\
&F(t):= \frac{1}{2}||(\rho|\bb{u}|^2) (t)||_{L^1(\Omega^w(t))}+ \frac{1}{\gamma-1}||\rho||_{L^\gamma(\Omega^w(t))}^\gamma, \\
&S(t):= \frac{1}{2} ||w_t(t)||_{ L^2(\Gamma))}^2 + \frac{1}{2}|| \Delta w(t)||_{L^2(\Gamma)}^2 +\frac{1}{2} ||\theta(t)||_{L^2(\Gamma)}^2, \\
&DF(t) := \int_0^t\Big[ \mu|| \nabla\bb{u}(\tau)||_{L^2(\Omega^w(\tau))}^2 +(\mu+\lambda)|| \nabla \cdot \bb{u}(\tau)||_{L^2(\Omega^w(\tau))}^2 \Big] d\tau , \\
&DS(t) := \int_0^t ||\nabla \theta(\tau)||_{L^2(\Gamma)}^2 d\tau .
\end{aligned} 
\end{eqnarray}
Noticing that there is a constant $c(\lambda,\mu)>0$ such that
\begin{eqnarray}\label{munu}%%%%%%%%%%%%%
c(\lambda,\mu)||\nabla \bb{u}(t)||_{L^2(\Omega^w(t))}^2 \leq  \mu|| \nabla\bb{u}(t)||_{L^2(\Omega^w(\tau))}^2 +(\mu+\lambda)|| \nabla \cdot \bb{u}(t)||_{L^2(\Omega^w(\tau))}^2,
\end{eqnarray}%%----------------------------%%
for any $t\in [0,T]$, from $\eqref{qq22eneq}$ and (A2), it inspires to define the spaces of weak solutions of $\eqref{qq22structureeqs}-\eqref{qq22initialdata}$ as follows: the structure temperature space
\begin{eqnarray*}%%%%%%%%%%%%%
\mathcal{W}_H(0,T):= L^\infty(0,T; L^2(\Gamma))\cap L^2(0,T; H_0^1(\Gamma)),
\end{eqnarray*}%%----------------------------%%
the space for the fluid density
\begin{eqnarray*}%%%%%%%%%%%%%
\mathcal{W}_D(0,T):= C_w (0,T; L^\gamma(\Omega^w(t))),
\end{eqnarray*}%%----------------------------%%
the structure displacement space
\begin{eqnarray*}%%%%%%%%%%%%%
\mathcal{W}_S(0,T):=W^{1,\infty}(0,T;L^2(\Gamma)) \cap L^\infty(0,T;H_0^2(\Gamma)), 
\end{eqnarray*}%%----------------------------%%
the fluid velocity space
\begin{eqnarray*}%%%%%%%%%%%%%
\mathcal{W}_F(0,T):= L^\infty(0,T; L^2(\Omega^w(t)))\cap L^2(0,T; H^1(\Omega^w(t)),
\end{eqnarray*}%%----------------------------%%
the coupled fluid-structure solution space
\begin{eqnarray*}%%%%%%%%%%%%%
\mathcal{W}_{F S}(0,T)= \{ (\bb{u}, w) \in \mathcal{W}_F(0,T) \times \mathcal{W}_S (0,T): \gamma_{|\Gamma^w(t)}\bb{u} = \partial_t w \bb{e}_3 \text{ for a.e. }t\in(0,T) \}.
\end{eqnarray*}%%----------------------------%%
Here, for a given $w\in \mathcal{W}_S(0,T)$, $\gamma_{|\Gamma^w(t)}$ is the ``Lagrangian'' trace operator on $\Gamma^{w}(t)$ defined as 
\begin{eqnarray*}
(\gamma_{|\Gamma^w(t)}f)(X) := f(X,w(t,X)), \quad \text{for } X\in \Gamma, t\in[0,T],
\end{eqnarray*}
for any $f\in C^1(\Omega^w(t))$, and then continuously extended to a linear operator from $H^1(\Omega^w(t))$ to $H^{s}(\Omega)$ for any $s<1/2$ (see \cite{Boris}). Now, we can define the weak formulation of the problem $\eqref{qq22structureeqs}-\eqref{qq22initialdata}$ as follows:
\begin{mydef}\label{qq22weaksolmoving}(\textbf{Weak solution on the physical domain})
Under the assumptions (A1) and (A2) of $\mathcal{F}$, we say that $ (\rho ,\bb{u}, w, \theta) \in \mathcal{W}_D(0,T) \times \mathcal{W}_{FS}(0,T) \times \mathcal{W}_H(0,T)$ is a weak solution of the problem $\eqref{qq22structureeqs}$-$\eqref{qq22initialdata}$, if the initial data $\rho_0, (\rho \bb{u})_0, w_0, v_0,\theta_0\in L^\gamma(\Omega^{w_0})\times L^{\frac{2\gamma}{\gamma+1}}(\Omega^{w_0})\times H_0^2(\Gamma)\times [L^2(\Gamma)]^2$ satisfy the compatibility conditions given in $\eqref{qq22compatibilityconditions}$ and:
\begin{enumerate}
\item The heat equation
\begin{align}\label{qq22thermalweak}%%%%%%%%%%%%%
\int_{\Gamma_T}\theta \partial_t \widetilde{\psi} -\int_{\Gamma_T} \nabla \theta \cdot\nabla \widetilde{\psi}+\int_{\Gamma_T}\nabla w \cdot \nabla \partial_t \widetilde{\psi} =\int_0^T \frac{d}{dt}\int_{\Gamma}\theta\widetilde{\psi} +\int_0^T \frac{d}{dt} \int_{\Gamma}\nabla w \cdot\nabla \widetilde{\psi},\nonumber\\ 
\end{align}%%----------------------------%%
holds for all $ \widetilde{\psi} \in C_0^\infty(\Gamma_T)$.
\item The continuity equation
\begin{eqnarray}\label{qq22weaksolcont}%%%%%%%%%%%%%
 \int_{Q_T^w} \rho \partial_t \varphi + \int_{Q_T^w} \rho \bb{u}\cdot \nabla \varphi =\int_0^T \frac{d}{dt} \int_{\Omega^w(t)} \rho \varphi , 
\end{eqnarray}%%----------------------------%%
holds for all $\varphi \in C^\infty([0,T]\times \overline{\Omega^w(t)})$.
\item The coupled momentum equation\footnote{For simplicity, from here onwards, we will write $\int_{\Gamma} \mathcal{F}(w) \psi$ instead of $\langle \mathcal{F}(w),\psi \rangle_{H^{-2}(\Gamma),H_0^2(\Gamma)}$.}
\begin{eqnarray}%%%%%%%%%%%%%
&&\int_{Q_T^w} \rho \bb{u} \cdot\partial_t \bb{q} + \int_{Q_T^w}(\rho \bb{u} \otimes \bb{u}):\nabla\bb{q}+\int_{Q_T^w} \rho^\gamma (\nabla \cdot \bb{q}) - \mu\int_{Q_T^w} \nabla \bb{u}: \nabla \bb{q}\nonumber\\
&&- \int_{Q_T^w}(\mu+\lambda) (\nabla \cdot \bb{u})( \nabla \cdot \bb{q})  +\int_{\Gamma_T} \partial_t w \partial_t \psi - \int_{\Gamma_T}\Delta w \Delta \psi - \int_{\Gamma_T}\mathcal{F}(w) \psi \nonumber\\
&&+\int_{\Gamma_T} \nabla \theta \cdot \nabla \psi = \int_0^T \frac{d}{dt} \int_{\Omega^w(t)}\rho\mathbf{u}\cdot\bb{q} + \int_0^T \frac{d}{dt} \int_\Gamma \partial_t w \psi \quad \quad \quad \quad  \label{qq22weaksolmom}
\end{eqnarray}%%----------------------------%%
holds for all $\bb{q} \in C_0^\infty(Q_{T,\Gamma}^w)$ and $\psi\in C_0^\infty(\Gamma_T)$ such that $\bb{q}_{|\Gamma^w(t)}= \psi \bb{e}_3$.
\end{enumerate}
\end{mydef}
\noindent
The main result of this paper can be stated as follows:
\begin{thm}(\textbf{Main result})\label{qq22mainth}
Let $\gamma>12/7$ and the initial data $\big(\rho_0,(\rho \bb{u})_0, w_0, v_0,\theta_0 \big) \\ \in L^\gamma(\Omega^{w_0})\times L^{\frac{2\gamma}{\gamma+1}}(\Omega^{w_0})\times H_0^2(\Gamma)\times [L^2(\Gamma)]^2$ satisfy the compatibility conditions given in $\eqref{qq22compatibilityconditions}$. Then, there exists a solution in the sense of Definition $\ref{qq22weaksolmoving}$ that satisfies the following energy inequality for all $t\in [0,T]$
\begin{eqnarray}\label{qq22energymain}%%%%%%%%%%%%%
&&\frac{1}{2} \bint_{\Omega^w(t)} (\rho|\bb{u}|^2) (t)+\frac{1}{\gamma-1} \bint_{\Omega^w(t)} \rho^\gamma(t) + \int_0^t\bint_{\Omega^w(t)} \Big[ \mu |\nabla\bb{u}|^2+(\mu+\lambda)(\nabla \cdot \bb{u})^2 \Big] \nonumber \\
&&+ \frac{1}{2}\int_\Gamma |\partial_t w(t)|^2+ \frac{1}{2} \int_\Gamma |\Delta w(t)|^2 +\frac{1}{2}\int_\Gamma |\theta(t)|^2+\int_0^t \int_\Gamma |\nabla \theta|^2 \leq C (E_0,C^*,\kappa), \quad\quad
\end{eqnarray}%%----------------------------%%
where $C^*$ and $\kappa$ are given in the assumption $(A2)$ and $E_0: = F(0)+S(0)$ is the initial energy, with $F(t),S(t)$ being given in $\eqref{qq22energies}$. Moreover, if the free boundary $\{ z = {w}(t,X)\}$ touches the bottom $\{ z= - 1\}$ at time $T^*$, then this solution can be defined on the time interval $(0,T)$, for any $T<T^*$. If no collision occurs, this solution can be defined on the time interval $(0,\infty)$.
\end{thm}

\begin{rem}
(1) In the initial energy $E_0$, the initial kinetic energy of the fluid is understood as $\int_{\Omega^{w_0}}\frac{(\rho \bb{u})_0^2}{\rho_0}$.\\
(2) The weak solution we shall construct to prove this theorem also satisfies the renormalized continuity equation defined in Theorem $\ref{qq22RCEdelta}$.\\
(3) In standard theory for weak solutions for compressible viscous fluids, $\gamma>3/2$ suffices (see \cite{novotnystraskraba}). In this paper however, a stronger assumption $\gamma>12/7$ is required to obtain Lemma $\ref{qq22lemkappa}$, where we exclude the concentration of the mass of the approximate pressure near the boundary. This lemma, combined with Lemma $\ref{qq22lemmaQproof2}$ in which additional interior integrability of approximate density is shown, ensures the weak $L^1$-convergence of the approximate pressure. This is an alternative to the standard proof that relies on the usage of the Bogovskii operator, which fails in our framework because the elastic structure isn't regular enough to ensure the Lipschitz regularity of the fluid domain. This idea was developed by Kuku\v{c}ka \cite{kukucka} in the context of compressible viscous fluids in irregular domains, and later adapted to the context of fluid-structure interaction by Breit and Schwarzacher in \cite{compressible}.
\end{rem}

\begin{rem}\label{qq22quasil}
The same weak solution existence result holds for a special quasilinear thermoelastic plate equation case with the nonlinear elastic force in $\eqref{qq22structureeqs}_1$ being $\mathcal{F} = \Delta (\Delta w)^3$. This will be proved in Appendix A. In this case, the potential of $\mathcal{F}$ is $\Pi(w) = \frac{1}{4}||\Delta w||_{L^4(\Gamma)}^4$, so by the lower semicontinuity of norms, the potential is preserved in the energy inequality which then takes the form:
\begin{eqnarray*}%%%%%%%%%%%%%
E(t) + D(t)+\Pi(w(t)) \leq C(E_0)+\Pi(w(0)).
\end{eqnarray*}%%----------------------------%%
where $E(t)$ and $D(t)$ are given in $\eqref{qq22energies}$. Such a thermoelastic plate model was first studied in \cite{thefirstpaper} (see also \cite{lasiecka} and the references therein). This high order nonlinearity arises from a thermoelastic plate model where a nonlinear coupling is considered between the elastic, magnetic and thermoelastic fields.
\end{rem}

\subsection{The equivalent (ALE) formulation of the weak solution on the fixed reference domain} 

We first define the fixed reference domain
\begin{eqnarray*}
\Omega = \{ (X,z) : X \in \Gamma, -1< z < 0 \}
\end{eqnarray*}
and
\begin{eqnarray*}%%%%%%%%%%%%%
Q_T: = [0,T]\times\Omega, \quad Q_{T,\Gamma}:= [0,T]\times \big(\Omega\cup(\Gamma\times\{0\}) \big).
\end{eqnarray*}%%----------------------------%%
To formulate the problem on the fixed reference domain $\Omega$ (as it was done in the context of incompressible fluids in \cite{BorSun,BorSunNonLinear,BorSunNavierSlip,trwa,trwa2}), we introduce a family of the following arbitrary Lagrangian-Eulerian (ALE) transformations:
\begin{eqnarray*}
A_w(t) : &&\Omega \to \Omega^w(t),\\
&&(X,z) \mapsto (X,(z+1)w(t,X)+z).
\end{eqnarray*}
This mapping is a bijection and its Jacobian, defined by
\begin{eqnarray*}
J(t,X,z):=\text{det} \nabla A_w(t,X) = 1+w(t,X),
\end{eqnarray*}
is well-defined as long as $w(t,X)>-1$ for any $X \in \Gamma$. Define the ALE velocity as
\begin{eqnarray*}
&\mathbf{w}: = \ddfrac{d}{dt}A_w = (z+1)\partial_t w \bb{e}_3.& 
\end{eqnarray*}
To express the derivatives with respect to the coordinates on $\Omega^w(t)$ by those in the coordinates on the fixed domain $\Omega$, we first calculate 
\begin{eqnarray*}
&(\nabla A_w)^{-1} = [\bb{e}_1, \bb{e}_2, \overline{A}_w ]^T, ~~~~
\overline{A}_w = \ddfrac{1}{w+1}[-(z+1)\partial_x w, -(z+1)\partial_y w, 1]^T,
\end{eqnarray*}
and for an arbitrary (vector or scalar) function $\bb{f}$ defined on $\Omega^w (t)$, we introduce
\begin{enumerate}
\item The \textit{pullback} by $A_w$: ~~$\bb{f}^w(t,X,z):=\bb{f}(t, X, A_w(t,X))$, for $(X,z) \in \Omega$ ;
\item The \textit{push forward of the gradient} by $A_w$:~~ $\nabla^w \mathbf{f}^w:=(\nabla \mathbf{f})^w = \nabla \mathbf{f}^w (\nabla A_w^{-1})\circ A_w$;
\item The \textit{transformed divergence} of $\bb{f}$: $\nabla^w \cdot \bb{f} := \text{Tr}(\nabla^w \mathbf{f}^w)$.
\end{enumerate}
We want to define the weak solution on the fixed domain $\Omega$, by composing the functions $\rho$ and $\bb{u}$ with the mapping $A_w$. First, from the energy inequality $\eqref{qq22eneq}$, we only have the bound for $w$ in $H_0^2(\Gamma)$, and since $H_0^2(\Gamma)$ is embedded into the H\"{o}lder space $C^{0,\alpha}$ for $\alpha<1$, we cannot expect that $\Omega^w(t)$ has a Lipschitz boundary. This means that transformation $A_w$ is not necessarily Lipschitz, so the transformed velocity $\bb{u}^w$ may not be in $L^2(0,T; H^1(\Omega))$, but rather in the transformed velocity space defined as
\begin{eqnarray*}
\mathcal{W}_F^w := \{\bb{u}^w: \bb{u} \in \mathcal{W}_F \}
\end{eqnarray*}
for which we know that $\mathcal{W}_F^w \subset L^\infty(0,T; L^2(\Omega))\cap L^2(0,T; H^s(\Omega)) \cap L^2(0,T; W^{1,p}(\Omega))$, for any $s<1$ and $p<2$.

We define the coupled fluid-structure space for the fixed reference domain as
\begin{eqnarray*}%%%%%%%%%%%%%
\mathcal{W}_{FS}^w(0,T) := \{(\bb{U},w) \in \mathcal{W}_F^w(0,T) \times \mathcal{W}_S(0,T):  \gamma_{|\Gamma\times \{0\}} \bb{U} = w\bb{e}_3 \},
\end{eqnarray*}%%----------------------------%%
and the space for the density on the fixed reference domain
\begin{eqnarray*}%%%%%%%%%%%%%
\mathcal{W}_D^w(0,T):= C_w(0,T ; L^\gamma(\Omega)).
\end{eqnarray*}%%----------------------------%%
\begin{rem}(\textbf{A convention on the notation}) For $\rho$ and $\bb{u}$, the fluid density and velocity defined on the physical domain $\Omega^w(t)$, denote the corresponding pull-back density and velocity on the fixed reference domain by ${r}:= \rho \circ A_w$ and $\bb{U}:= \bb{u} \circ A_w$, respectively.
The gradient on both physical and fixed domains will be denoted by $``\nabla"$ without any confusion, since it will be clear either from the function that is applied onto, or from the domain of integration. 
\end{rem}
Now, to define the weak solution in the sense of Definition $\ref{qq22weaksolmoving}$ on the fixed reference domain $\Omega$, we express the functions $\rho,\bb{u}, \bb{q},\varphi$ defined on $\Omega^w(t)$ by the corresponding pull-backs by $A_w$ to obtain:
\begin{lem}\label{qq22weaksolution}(\textbf{Weak solution on the fixed reference domain}) The functions $(\rho,\bb{u},w,\theta)$ are weak solutions in the sense of Definition $\ref{qq22weaksolmoving}$ that satisfy the energy inequality $\eqref{qq22energymain}$ if and only if the following hold:
\begin{enumerate} 
\item[(1)] The initial data $\big(r_0, (r\bb{U})_0, w_0, v_0,\theta_0 \big)\in L^\gamma(\Omega)\times L^{\frac{2\gamma}{\gamma+1}}(\Omega)\times H_0^2(\Gamma)\times [L^2(\Gamma)]^2$ and the following compatibility conditions hold
\begin{eqnarray*} %%%%%%%%%%%%%
\begin{aligned}
r_0>0,& \text{ in } \{  (X,z)\in \Omega:  (r\bb{U})_0(X,z) >0\},\\
\frac{(r \bb{U})_0^2}{r_0}& \in L^1(\Omega),\\
\partial_\nu w_0(X)=w_0(X) =0, &\text{ on } \partial \Gamma, \\
w_0(X)>-1, &\text{ on } \Gamma.\\
\end{aligned}
\end{eqnarray*}%%----------------------------%%

\item[(2)] The heat equation holds in the sense of
\begin{eqnarray*}
\int_{\Gamma_T}\theta \partial_t \widetilde{\psi} -\int_{\Gamma_T} \nabla \theta \cdot \nabla \widetilde{\psi}+\int_{\Gamma_T}\nabla w \cdot \nabla \partial_t \widetilde{\psi}= \int_0^T \frac{d}{dt}\int_{\Gamma}\theta\widetilde{\psi} +\int_0^T \frac{d}{dt} \int_{\Gamma}\nabla w \cdot \nabla \widetilde{\psi},
\end{eqnarray*}
for all $\widetilde{\psi} \in C^\infty(\Gamma_T)$.
\item[(3)] The continuity equation holds in the sense of
\begin{eqnarray}\label{qq22fixedcontdef}
\int_{Q_T }Jr \partial_t\varphi +\int_{Q_T } J(r \bb{U}- r \bb{w})\cdot \nabla^w \varphi =\int_0^T \frac{d}{dt} \int_{\Omega}J r \varphi,
\end{eqnarray}
for all $\phi \in C^\infty([0,T] \times \overline{\Omega})$.
\item[(4)] The coupled momentum equation holds in the sense of
\begin{eqnarray}
&&\bint_{Q_T} J\Big((r\bb{U} - r\bb{w})\cdot
\nabla^w) \mathbf{q}\cdot \bb{U}+ Jr\bb{U} \cdot
\partial_t \mathbf{q}\Big)
+\bint_{Q_T}Jr^\gamma (\nabla^w \cdot \bb{U}) \nonumber\\[2mm]
&&-\mu\bint_{Q_T} J\nabla^w \bb{U}:\nabla^w\mathbf{q}-(\mu+\lambda)\bint_{Q_T} J(\nabla^w \cdot \bb{U})(\nabla^w \cdot \bb{q})  \nonumber\\[2mm]
&&+ \bint_{\Gamma_T} \partial_t w \partial_t \psi -
\bint_{\Gamma_T}\mathcal{F} (w) \psi- \bint_{\Gamma_T} \Delta w \Delta \psi + \bint_{\Gamma_T}\nabla \theta \cdot \nabla \psi \nonumber\\[2mm]
&&= \bint_0^T\frac{d}{dt}\int_{\Omega} Jr \mathbf{U} \cdot \bb{q} +\bint_0^T\frac{d}{dt} \int_\Gamma v \psi, \label{qq22fixedmomdef}
\end{eqnarray}
for every $\mathbf{q} \in C_0^\infty(Q_{T,\Gamma})$ and $\psi \in C_0^\infty(\Gamma_T)$ such that $\bb{q}_{|\Gamma}= \psi \bb{e}_3$.
\item[(5)] Functions $(r,\bb{U},w,\theta) \in \mathcal{W}_D^w(0,T)\times \mathcal{W}_{FS}^w(0,T) \times \mathcal{W}_H(0,T)$ and the following energy inequality holds
\begin{eqnarray*}%%%%%%%%%%%%%
E^w(t) + D^w(t) \leq C(E_0,C^*,\kappa),
\end{eqnarray*}%%----------------------------%%
where
\begin{eqnarray}\label{qq22energies2}
\begin{aligned}
&E^w(t) := F^w(t)+S(t), ~~D^w(t) := DF^w(t) + DS(t), \\
&F^w(t):= \frac{1}{2}||(Jr|\bb{U}|^2) (t)||_{L^1(\Omega)}+ \frac{1}{\gamma-1}||Jr^\gamma||_{L^1(\Omega)}, \\
&DF^w(t) := \int_0^t \Big[ \mu|| \nabla^w\bb{U}(\tau)||_{L^2(\Omega^w(\tau))}^2 +(\mu+\lambda)|| \nabla^w \cdot \bb{U}(\tau)||_{L^2(\Omega^w(\tau))}^2 \Big]d \tau.
\end{aligned} 
\end{eqnarray}
\end{enumerate}
\end{lem}
\begin{proof}
First, the assertions in (1), (2) and (5) are straightforward. Now, by using the fact that 
\begin{eqnarray*}%%%%%%%%%%%%%
\ddfrac{d}{dt} (\bb{q} \circ A_w)= (\partial_t \bb{q})\circ A_w + \bb{w} \cdot \nabla^w (\bb{q} \circ A_w),
\end{eqnarray*}%%-------
in the equations $\eqref{qq22weaksolcont}$ and $\eqref{qq22weaksolmom}$, one obtains that the equations $\eqref{qq22fixedcontdef}$ and $\eqref{qq22fixedmomdef}$ hold for the corresponding pull-backs of the smooth test functions, i.e. for all $\varphi \circ A_w$ and $(\bb{q} \circ A_w,\psi)$, respectively, such that $\varphi \in C^\infty ([0,T]\times \overline{\Omega^w(t)})$, $\bb{q} \in C_0^\infty(Q_{T,\Gamma}^w)$ and $\psi\in C_0^\infty(\Gamma_T)$ with $\bb{q}_{|\Gamma^w(t)}= \psi \bb{e}_3$. It remains to prove that the equations $\eqref{qq22fixedcontdef}$ and $\eqref{qq22fixedmomdef}$ hold for all $\phi \in C^\infty([0,T] \times \overline{\Omega})$ and $\mathbf{q} \in C_0^\infty(Q_{T,\Gamma})$ and $\psi \in C_0^\infty(\Gamma_T)$ such that $\bb{q}_{|\Gamma}= \psi \bb{e}_3$. This will follow by the density argument if we prove that the convective terms that include $\bb{w}$ (which are the only new terms that appear in this formulation) are integrable. This is indeed true because $\gamma>12/7$ and because of the trace regularity $\partial_t w \in L^2(0,T; H^s(\Gamma))$ for $0<s<1/2$ which then implies by the imbedding $(0,0,(z+1)\partial_t w) = \bb{w} \in L^2(0,T; L^p(\Gamma))$ for $p<4$.
\end{proof}

\begin{rem}
The condition $\gamma>12/7$ is crucial for this formulation to make sense. The domain transformation $A_w$ is chosen to be invertible since it Jacobian only depends on $w$ and not on its higher derivatives. However, the transformation itself has the same regularity as $w$ so the domain transformation velocity $\bb{w}$ is of the same regularity as $\partial_t w$. If the structure was viscoelastic, i.e. adding the term $-\partial_t \Delta w$ in the plate equation $\eqref{qq22structureeqs}_1$, then $\bb{w}$ would automatically be more regular and belong to the space $L^2(0,T; H^1(\Gamma))$), so $\gamma>3/2$ would suffice. 
\end{rem}

This lemma allows us to study the problem on the fixed reference domain $\Omega$ with smooth test functions, which will be very useful in the upcoming analysis. However, not all the analysis will be done on $\Omega$. In particular, the convergence of the approximate pressure constructed in the following sections will rely on the usage of the inverse divergence operator which doesn't make sense on the fixed reference domain, as an inverse transformed divergence operator depends on the displacement. This is mainly because, when we transform the problem onto the fixed reference domain, some of the natural properties are lost. For example, the problem isn't in the conservative form and the transformed divergence doesn't satisfy the divergence theorem. Thus, it will be convenient to jump from the fixed reference domain formulation to the physical domain formulation. Both formulations are useful for different parts of the analysis and their interplay is one of the important approaches in this paper, as this seems to be an effective way to study this problem.

\section{Approximate problems}
First, to get the existence of weak solutions in the sense of Definition $\ref{qq22weaksolmoving}$, we introduce the approximate problems. Then, we will solve them and obtain the uniform energy estimates of the approximate solutions.
\subsection{Formulation of approximate problems}
We will construct a 4-level approximation scheme on a fixed reference domain:
\begin{enumerate}
\item \textit{Artificial pressure, fixed reference domain regularization and structure regularization for a fixed $\delta>0$} ($\delta$-level):\\
The pressure $r^\gamma$ is replaced by $r^\gamma+\delta r^a$, for a large $a>0$, the fixed reference domain $\Omega$ to a more regular domain $\Omega_\delta$ and a regularizing term $\delta \nabla^3 w: \nabla^3\psi$ is added to the plate equation.
\item \textit{Artificial density damping on the fixed reference domain for a fixed $\varepsilon>0$} ($\varepsilon$-level):\\ 
The term $\varepsilon \big( \Delta r + \ddfrac{1}{J}\nabla J\cdot \nabla r\big)$ is added to the continuity equation defined on the fixed reference domain;
\item \textit{Finite Galerkin bases for a fixed $k \in \mathbb{N}$} ($k$-level):\\
The fluid velocity $\bb{U}$, the structure displacement $w$ and the structure temperature $\theta$ are projected onto the generated finite bases.
\item \textit{The operator splitting} ($\Delta t$-level):\\
For a fixed $T>0$ and $N\geq 1$, letting $\Delta t=\frac{T}{N} $, we split the time interval $[0,T]$ into $N$ equal sub-intervals and on each sub-interval we use the Lie operator splitting, and decouple the problem into two parts - the fluid and structure sub-problems.
\end{enumerate}
\begin{rem}
(1) \textbf{The $\Delta t$ level}. In the approximate problem, due to the operator splitting, the trace of the fluid velocity at the structure, denoted by $v$, and the structure velocity $\partial_t w$ are not necessarily equal, but their difference in $L^2(\Gamma_T)$ norm is smaller than $O(\sqrt{\Delta t})$. Also, $\Delta t$ is chosen to be smaller or equal to the maximal interval of the solution that we will obtain for the fluid sub-problem by the fixed-point argument, which is then prolonged $N-1$ times to be defined on $[0,T]$ by using the uniform estimates. 

The first ``time semi-discretization via operator splitting'' scheme in the context of the incompressible viscous fluids interacting with an elastic shells/plates was constructed by Muha and {\v C}ani{\'c} in \cite{BorSunNavierSlip, BorSun,BorSunNonLinear} where the corresponding fluid and structure sub-problems were both stationary. Then in \cite{trwa}, we studied the interaction between in which the nonlinear plate with the nonlinear elastic force $\mathcal{F}(w)$ satisfying the assumptions (A1) and (A2) given in section $\ref{qq22sec1.2}$ interacts with a viscous incompressible fluid and constructed a hybrid approximation scheme where the fluid sub-problem is stationary and the structure sub-problem is continuous in time and in a finite basis. We later extended this result in \cite{trwa2} by studying the interaction between an incompressible viscous fluid and a nonlinear thermoelastic plate, where we also included an additional quasilinear plate model with cubic nonlinear elastic force. Now, in this paper, we choose the fluid sub-problem to be continuous in time as well. This way, the nature of both sub-problems is preserved almost completely compared to the original problem. In particular, the energy inequality of the approximate solutions is very similar to the energy inequality $\eqref{qq22energymain}$ and the fluid sub-problem is solved in almost the same way as in the standard theory for compressible viscous fluids by means of the Schauder fixed-point theorem.\\

(2) \textbf{The $k$ level}. Spanning these functions is quite standard both for compressible fluids and elastic plates. Here we fix the same number of basis functions for all three functions $\bb{U}$, $w$ and $\theta$.\\

(3) \textbf{The $\varepsilon$ level}. It is standard in the study of weak solution theory for compressible viscous fluids to damp the continuity equation. Usually, it is done by adding the term $\varepsilon \Delta \rho$ to the continuity equation on the physical domain $\Omega^w(t)$. However, here we instead use the damping $\varepsilon(\Delta r+ \frac{1}{J} \nabla J \cdot \nabla r)$ for the following reasons. First, if one would use the standard damping, then the continuity equation on the fixed reference domain would have the push-forward of the Laplacian. This would result in a second order parabolic equation where both the coefficients of the second order derivatives of the transformed density and the boundary condition would depend on time, since the normal vector which is used for the Neumann boundary condition for the density would depend on the displacement. The second reason is that we would only obtain approximate weak solutions of this continuity equation where we could also have vacuum. The damping we construct allows us to solve the approximate continuity equation on the fixed reference domain as a linear second-order parabolic equation with Neumann boundary condition, its solutions are regular, they satisfy the maximal regularity estimates and the approximate density is bounded from below and above by positive constants, as in the standard weak solution theory for the compressible fluids. The later property ensures that the coupled momentum equation is non-degenerate. \\

(4) \textbf{The $\delta$ level}. The artificial pressure is used to ensure that the approximate density is integrable enough. This is essential in proving certain convergences throughout the sections $\ref{qq22section4},\ref{qq22section5}$ and $\ref{qq22section6}$. The reason we need a regular fixed reference domain $\Omega_\delta$ is to be able to solve the damped continuity equation and to obtain maximal regularity estimates. The regularizing term $\delta \nabla^3 w: \nabla^3\psi$ for the plate equation is added to keep the domain transformation mapping $A_w$ more regular. This additional regularity is used in bounding of the term $I_2$ in Lemma $\ref{qq22lemmaQproof2}$ and in Appendix B. Moreover, it also ensures that the functional spaces for the fluid density and velocity on the fixed and physical domains are the same. This will simplify the analysis and the notation.
\end{rem}

\subsubsection{The structure sub-problem (SSP)}
First, we want to span the plate temperature and displacement in finite bases. Let $\{ s_i \}_{i\in \mathbb{N}}$ and $\{ h_i \}_{i\in \mathbb{N}}$ be the sets of eigenfunctions  generated by the biharmonic eigenvalue problem with the clamped boundary condition,  and the harmonic eigenvalue problem with the Dirichlet boundary condition, respectively, with the corresponding eigenvalues $\{ \xi_i^s \}_{i\in \mathbb{N}}$ and $\{ \xi_i^h\}_{i\in \mathbb{N}}$. Denote by 
\begin{eqnarray*}%%%%%%%%%%%%%
\mathcal{P}_{str}^k: = \text{span}\{s_i\}_{ 1\leq i \leq k}, \quad
\mathcal{P}_{heat}^k = \text{span}\{h_i\}_{ 1\leq i \leq k},
\end{eqnarray*}%%----------------------------%%
and the corresponding projections $P_{str}^k: L^2(\Gamma) \to \mathcal{P}_{str}^k$ and $P_{heat}^k: L^2(\Gamma) \to \mathcal{P}_{heat}^k$. \\

The approximate initial data are chosen as $w^0(0,X) = w_{0,k}(X)\in P_{str}^k(w_0)$ such that $\min\limits_{X\in \Gamma}w_0(X) \leq w_{0,k}(X)\leq \max\limits_{X\in \Gamma}w_0(X)$ such that $w_{0,k} \to w_0$ in $H_0^2(\Gamma)$ as $k\to+\infty$ and 
\begin{eqnarray*}
\quad \partial_t w^0(0,X) = v_{0,k}(X):=P_{str}^k(v_0), \quad \theta_{\Delta t,k}^0(0,X) = \theta_{0,k}(X):= P_{heat}^k(\theta_0).
\end{eqnarray*}
We are ready to define:\\

\noindent
\underline{The structure sub-problem (SSP):} \\
By induction on $n$ for any $n\geq 0$, assuming that the approximate solution $v^n \in C^1([(n-1) \Delta t, n\Delta t]; \mathcal{P}_{str}^k)$ of (FSP) (which will be introduced in the next section) and $w^{n} \in C^2([(n-1) \Delta t, n\Delta t]; \mathcal{P}_{str}^k), \theta^{n} \in C^1([(n-1) \Delta t, n\Delta t]; \mathcal{P}_{heat}^k)$ are given already, determine $w^{n+1} \in C^2([n \Delta t, (n+1)\Delta t]; \mathcal{P}_{str}^k)$ and $ \theta^{n+1} \in C^1([n \Delta t, (n+1)\Delta t]; \mathcal{P}_{heat}^k)$ by solving the following problem:
\begin{eqnarray}\label{qq22SSP}
\begin{cases} 
\bint_{\Gamma} \partial_t \theta^{n+1} \widetilde{\psi} + \bint_{\Gamma}\nabla \theta^{n+1}\cdot \nabla \widetilde{\psi} +\bint_{\Gamma}\nabla \partial_t w^{n+1}\cdot \nabla \widetilde{\psi} = 0,\\[3.5mm]
\ddfrac{1}{2}\bint_{\Gamma}\partial_t^2 w^{n+1}(t) \psi+\ddfrac{1}{2}\bint_{\Gamma}\ddfrac{\partial_t w^{n+1}(t) - T_{\Delta t} v^{n+1}(t)}{\Delta t}\psi +\bint_{\Gamma}\Delta w^{n+1}(t) \Delta \psi\\[2.5mm]
\quad\quad  -\bint_{\Gamma}\nabla \theta^{n+1}(t) \cdot \nabla \psi + \bint_{\Gamma} \mathcal{F}(w^{n+1}(t))\psi+\delta \bint_{\Gamma}\nabla^3 w^{n+1}: \nabla^3 \psi = 0, \\[3.5mm] 
w^{n+1}(n \Delta t, X) =w^{n}(n \Delta t, X), ~~ \partial_t w^{n+1}(n \Delta t,X) = \partial_t w^{n}(n \Delta t,X) ,\\
\theta^{n+1}(n \Delta t,X) = \theta^{n}(n \Delta t,X) ,
\end{cases}
\end{eqnarray}
for all $t\in (n\Delta t, (n+1)\Delta t]$, $\psi \in \mathcal{P}_{str}^k$, $ \widetilde{\psi} \in\mathcal{P}_{heat}^k $, with $T_{\Delta t} f(t):=f(t-\Delta t)$ being the translation in time operator, while when $0\leq t\leq \Delta t$, we choose $T_{\Delta t} f(t)=f(0)$.

\subsubsection{The fluid sub-problem (FSP)}
Since we will need a domain smoother than $\Omega = \Gamma \times (-1,0)$ in order to solve the approximate continuity equation given below and to obtain certain maximal regularity estimates later on, we introduce the following extended domain:
\begin{mydef}
For a given $\delta \in (0,1)$ we define an open connected set $\Omega_\delta \supset \Omega$ (see Figure $\ref{qq22sections}$), such that it satisfies the following properties:
\begin{enumerate}
\item The boundary $\partial \Omega_\delta$ is of $C^{2,\alpha}$ regularity, for some $0<\alpha <1$, and uniformly Lipschitz with respect to $\delta$;
\item $\Omega_\delta \subset \mathbb{R}^2 \times (-1,0)$ and $\Gamma \times \{-1,0\} \subset \partial\Omega_\delta$;
\item $\forall x \in \partial \Omega_\delta$, $\text{dist}(x, \Omega) < \delta$;
\item $\Omega_{\delta'}\subset \Omega_{\delta''}$, for $\delta'<\delta''$.
\end{enumerate}
\end{mydef}
\begin{figure}[h]
\begin{center}
\includegraphics[scale = 0.3]{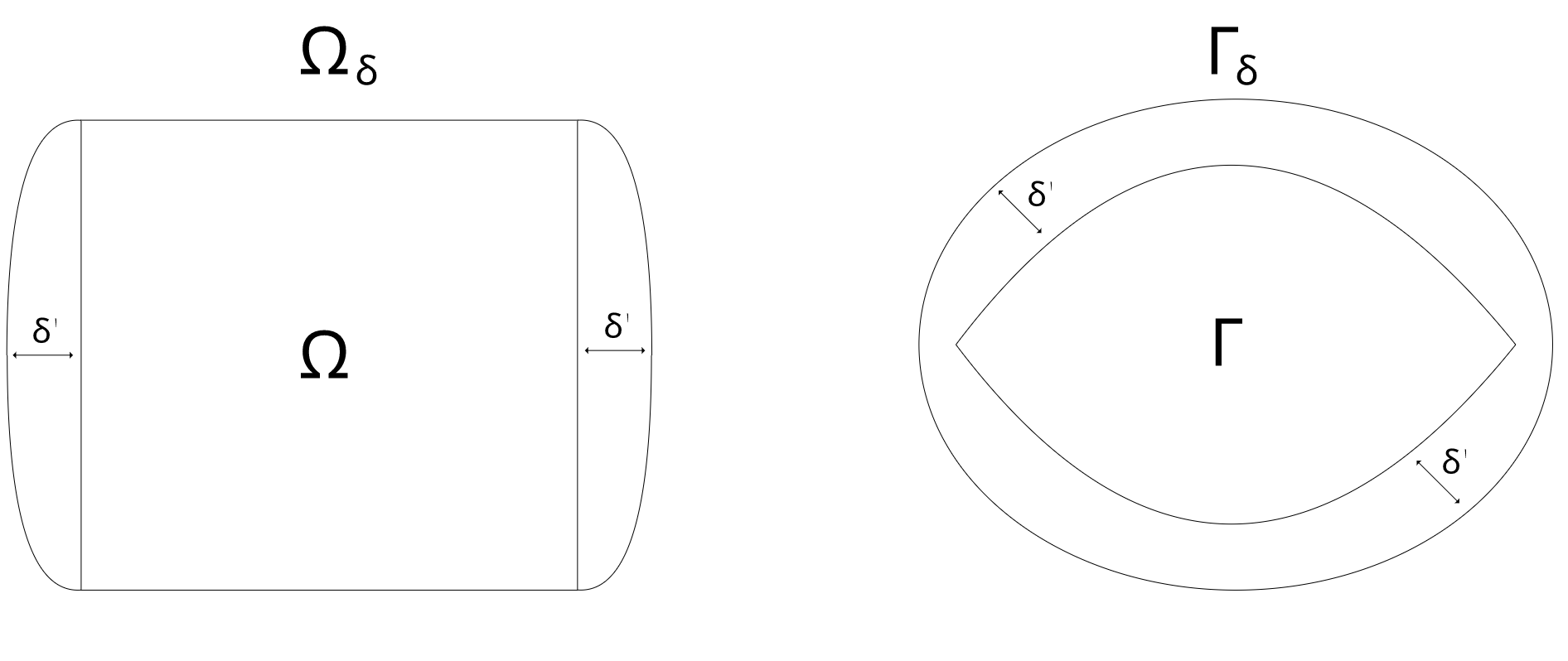}
\end{center}
\caption{The sets $\Omega_\delta$ and $\Gamma_\delta:= \{z=0\}\cap \partial \Omega_\delta$, visually represented by the vertical and the horizontal section, respectively, for some $\delta'<\delta$.}
\label{qq22sections}
\end{figure}
Obviously, as $\delta \to 0$, $\Omega_\delta \to \Omega $ and $\Gamma_\delta \to \Gamma$. The ALE mapping $A_w$ will be extended onto $\Omega_\delta \setminus \Gamma\times (-1,0)$ by $id$ and still denoted as $A_w$. \\

Denote by
\begin{eqnarray*}%%%%%%%%%%%%%
&&\Omega_{\delta}^w := \Omega^w \cup (\Omega_{\delta}\setminus \Omega), \quad \Omega_{\delta,\Gamma}^w: = \Omega_\delta^w \cup \Gamma^w, \quad Q_{\delta,T}^w: = [0,T]\times \Omega_\delta^w, \quad \Omega_{\delta,T,\Gamma}^w:= [0,T]\times \Omega_{\delta,\Gamma}^w,\\
&&\Omega_{\delta,\Gamma}: = \Omega_\delta \cup (\Gamma\times \{0\}), \quad Q_{\delta,T}: = [0,T]\times \Omega_\delta, \quad \Omega_{\delta,T,\Gamma}:= [0,T]\times \Omega_{\delta,\Gamma}. 
\end{eqnarray*}%%----------------------------%%
Let $\{ \bb{f}_i \}_{i\in \mathbb{N}}$ and $\{ \xi_i^f\}_{i\in \mathbb{N}}$ be the sets of eigenfunctions and eigenvalues determined by the following harmonic eigenvalue problem
\begin{eqnarray*}
\begin{cases}
- \Delta \bb{f} = \xi^f \bb{f}, ~&\text{in } \Omega_\delta, \\ 
\bb{f} =0, ~&\text{on } \partial \Omega_\delta .
\end{cases}
\end{eqnarray*}
For a given $s \in \mathcal{P}_{str}^k$ let $\text{Ext}[s] := r \bb{e}_3$, where $r$ is the  solution of the following problem
\begin{eqnarray*}%%%%%%%%%%%%%
\Delta r &=& 0, \quad \text{in } \Omega_\delta, \\
r &=& s \bb{e}_3, \quad \text{on } \Gamma \times \{0 \}, \\
r &=& 0, \quad \text{on } \partial \Omega_\delta \setminus (\Gamma \times \{0 \}).
\end{eqnarray*}%%----------------------------%%
We now intoduce the vector space $\mathcal{P}_{fl}^{k}: = \text{span}\{\bb{f}_i, \text{Ext}[s_i] \}_{ 1\leq i \leq k}$ with the corresponding projection denoted as $P_{fl}^k: L^2(\Omega_\delta) \to \mathcal{P}_{fl}^k$.
\begin{rem}\label{qq22isbasis}
To prove that the functions $\{\bb{f}_i, \text{Ext}[s_i] \}_{ 1\leq i \leq k}$ are linearly independent, for any $\bb{a} \in \mathbb{R}^{2k}$, it is easy to know that the linear combination $F=\sum_{i=1}^k a_i \bb{f}_i+ a_{i+k} \text{Ext}[s_i]$  satisfies the following problem
\begin{align}%%%%%%%%%%%%%
\Delta F &= \sum_{i=1}^k a_i \xi_i^f \bb{f}_i, \quad \text{in } \Omega_\delta, \nonumber \\
F &= \sum_{i=1}^k a_{i+k} s_i, \quad \text{on } \Gamma \times \{ 0\}, \nonumber \\
F &= 0, \quad \text{on } \partial\Omega_\delta \setminus \Gamma \times \{ 0 \}. \nonumber
\end{align}%%----------------------------%%
By using the uniqueness of the solution to this problem and the linear independency of the sets $\{\bb{f}_i \}_{ 1\leq i \leq k}$ and $\{s_i \}_{ 1\leq i \leq k}$, we have that $F = 0$ if and only if $\sum_{i=1}^k a_i \xi_i^f \bb{f}_i =0$ and $ \sum_{i=1}^k a_{i+k} s_i=0$, which is equivalent to $\bb{a} = 0$. \end{rem}

We define the solution spaces for the fluid sub-problem in the following way:
\begin{eqnarray*}
\mathcal{F}_k^{n+1}:=C^1([n \Delta t, (n+1)\Delta t]; \mathcal{P}_{fl}^k),
\end{eqnarray*}
for the approximate fluid velocity, and
\begin{eqnarray*}
& \mathcal{D}^{n+1} := H^1(n \Delta t, (n+1)\Delta t; L^2(\Omega_\delta)) \cap L^2(n \Delta t, (n+1)\Delta t; H^2(\Omega_\delta)),&
\end{eqnarray*}
for the approximate fluid density, with norm $||\cdot||_{\mathcal{D}^{n+1}}$ being naturally induced. 
${}$\\

\noindent
\underline{The fluid sub-problem (FSP):}\\
By induction on $n$, assuming that the approximate solution $w^{n+1} \in C^2([n\Delta t,(n+1)\Delta t];\mathcal{P}_{str})$ of (SSP), and $(r^n,\bb{U}^n)$ are given already, we determine $( r^{n+1},\bb{U}^{n+1}) \in \mathcal{D}^{n+1}\times \mathcal{F}_k^{n+1}$ from the following system:\footnote{The connection between (SSP), (FSP) and the original problem is explained in the section $\ref{qq22sec}$.}
\begin{eqnarray}\label{qq22FSP1}%%%%%%%%%%%%%%%
\begin{cases}
& \partial_t r^{n+1} - \bb{w}^{n+1} \cdot \nabla^{w} r^{n+1} + \nabla^{w} \cdot (r^{n+1} \bb{U}^{n+1})\\
&\hspace{.45in}  = \varepsilon \Big( \Delta r^{n+1} + \ddfrac{1}{J^{n+1}} \nabla J^{n+1} \cdot \nabla r^{n+1}\Big), \quad \text{a.e. in } Q_{\delta,T} , \\[2mm]
&\partial_n r^{n+1} =0, \quad \text{on }\partial \Omega_\delta
\\[4mm]
&\ddfrac{1}{2}\bint_{\Omega_\delta} \partial_t J^{n+1} r^{n+1} \bb{U}^{n+1} \cdot \bb{q} + \bint_{\Omega_\delta}J^{n+1}r^{n+1} \partial_t \bb{U}^{n+1} \cdot \bb{q} + \frac{1}{2}\bint_{\Omega_\delta} J^{n+1} \partial_t r^{n+1} \bb{U}^{n+1} \cdot \bb{q} \\[2mm]
&\hspace{.45in}+\ddfrac{1}{2}\bint_{\Omega_\delta} J^{n+1}(r^{n+1} \bb{U}^{n+1}-r^{n+1} \bb{w}^{n+1}) \cdot \big( \mathbf{q} \cdot \nabla^{w}\bb{U}^{n+1} - \bb{U}^{n+1}\cdot \nabla^{w} \mathbf{q} \big) \\[2mm]
&\hspace{.45in}+\mu \bint_{\Omega_\delta} J^{n+1}\nabla^{w} \mathbf{u}^{n+1}:\nabla^{w}\mathbf{q} +(\mu +\lambda)\bint_{\Omega_\delta} J^{n+1} (\nabla^{w}\cdot \mathbf{u}^{n+1})(\nabla^{w}\cdot\mathbf{q})\\
&\hspace{.45in}
-\bint_{\Omega_\delta} (J^{n+1} ((r^{n+1})^\gamma + \delta (r^{n+1})^a) (\nabla^{w}\cdot\bb{q}) +\frac{1}{2} \bint_\Gamma\ddfrac{v^{n+1}-\partial_t w^{n+1}}{\Delta t} \psi =0, \label{qq22FSP2} \\[4mm]
& r^{n+1}(n\Delta t, \cdot) = r^{n} (n \Delta t, \cdot), \quad \bb{U}^{n+1}(n\Delta t, \cdot) = \bb{U}^{n} (n \Delta t, \cdot),
\end{cases} ~~~~
\end{eqnarray}
for all $\bb{q} \in \mathcal{P}_{fl}^{k} $, with $\psi = \bb{q}_{| \Gamma\times \{0\}}$ and $n \Delta t \leq t \leq (n+1)\Delta t$, where $v^{n+1} := \bb{U}^{n+1}_{| \Gamma\times \{0\}} \cdot \bb{e}_3$ and $\nabla^w$ denotes $\nabla^{w^{n+1}}$ for simplicity. When $n=1$, the initial data $r_{\Delta t,k}^0 (0,X)$ can be chosen as a strictly positive smooth (and extended to $\Omega_\delta$) approximation of $\rho_0 \circ A_{w_0}$, $\bb{U}^0(0,X) = \mathcal{P}_{fl}(\bb{U}_{0,k,\delta})$ with
\begin{eqnarray*}%%%%%%%%%%%%%
\bb{U}_{0,k,\delta} := \begin{cases}
\frac{(r\bb{U})_0}{r_{\Delta t,k}^1(0) } \circ A_{w_{0}} - \text{Ext}[v_0 - v_{0,k}], \quad &(X,z) \in \Omega,\\
0, \quad &(X,z) \in \Omega_\delta \setminus\Omega.
\end{cases}
\end{eqnarray*}%%----------------------------%%
\subsection{Energy estimates of the solutions to (SSP) and (FSP)}
From now on, we will use the following notation 
\begin{eqnarray*}
g (t) := g^{n+1}(t), ~~\text{for } t\in [n\Delta t, (n+1) \Delta t),~~ 0\leq n \leq
N-1,
\end{eqnarray*}
with $g$ being one of the functions $r, \bb{U} , w, \theta$, to omit the superscript. Now, for the functions $r, \bb{U} , w, \theta$ solving the problems (SSP) and (FSP) on the interval $n\Delta t \leq t \leq (n+1)\Delta t$, and the corresponding fluid density and velocity $ \rho = r \circ A_{w^{-1}}, \bb{u} =\bb{U} \circ A_{w^{-1}}$ on the physical domain, we introduce the following appropriate forms of energies for both fixed and physical domain coordinates:
\begin{eqnarray*}%%%%%%%%%%%%%
F^{n+1}(t) &:=& \frac{1}{2}\int_{\Omega_\delta} (J r |\bb{U}|^2)+\int_{\Omega_\delta}J\Big(\frac{r^{\gamma}}{\gamma-1}+ \delta \frac{r^{a}}{a-1}\Big)\\
&=&\frac{1}{2}\int_{\Omega_\delta^{w}(t)} \rho |\bb{u}|^2 +\int_{\Omega_\delta^{w}(t)}\Big(\frac{\rho^{\gamma}}{\gamma-1}+ \delta \frac{\rho^{a}}{a-1}\Big), \\[2mm]
FD^{n+1}(t) &:=& \varepsilon \int_{n\Delta t}^{t} \int_{\Omega_\delta} J|\nabla r|^2(\gamma r^{\gamma-2}+\delta ar^{a-2})\\
&&+ \mu\int_{n \Delta t}^t \int_{\Omega_\delta} J |\nabla^{w}\bb{U}|^2+ (\mu +\lambda) \int_{n \Delta t}^t \int_{\Omega_\delta} J |\nabla^{w}\cdot\bb{U}|^2\\[2mm]
&=& \varepsilon \int_{n\Delta t}^{t} \int_{\Omega_\delta^{w}(t)} |\nabla^{w^{-1}} \rho|^2(\gamma \rho^{\gamma-2}+\delta a\rho^{a-2})\\
&&+\mu\int_{n \Delta t}^t \int_{\Omega_\delta^{w}(t)} |\nabla \bb{u}|^2+ (\mu +\lambda) \int_{n \Delta t}^t \int_{\Omega_\delta^{w}(t)} |\nabla\cdot\bb{u}|^2,\\[2mm]
S^{n+1}(t)&:=&\frac{1}{4}|| \partial_t w(t)||_{L^2(\Gamma)}^2+\frac{1}{2}|| \Delta w(t)||_{L^2(\Gamma)}^2 + \Pi(w(t)) \\[2mm]
&&+ \frac{1}{2}||\theta(t)||_{L^2(\Gamma)}^2+\frac{1}{2}\delta|| \nabla^3 w(t)||_{L^2(\Gamma)}^2 , \\[2mm]
SD^{n+1}(t) &:=& \int_{n\Delta t}^t \int_{\Gamma} | \nabla \theta|^2.
\end{eqnarray*}%%----------------------------%%

\subsubsection{The energy and the solution of (SSP)}
\begin{lem}\label{qq22sspestimateeps}
For a given function $v \in C^1([(n-1)\Delta t, n\Delta t] ; \mathcal{P}_{str}^k)$, there exists a unique solution of (SSP), $w \in C^2([n\Delta t, (n+1)\Delta t] ; \mathcal{P}_{str}^k)$, $\theta \in C([n \Delta t, (n+1)\Delta t]; \mathcal{P}_{heat}^k)$ that satisfies the following identity for all $t\in[n\Delta t,(n+1)\Delta t]$,
\begin{eqnarray}
&&\frac{1}{4\Delta t}\int_{n \Delta t}^{t}\big(||\partial_t
w-T_{\Delta t}v||_{L^2(\Gamma)}^2 +||\partial_t
w||_{L^2(\Gamma)}^2\big)+S^{n+1}(t)+SD^{n+1}(t) \nonumber \\
&&= S^n(n \Delta t)+\frac{1}{4\Delta t}\int_{n \Delta t}^{t}||T_{\Delta t}v||_{L^2(\Gamma)}^2. \label{qq22lem11eps}
\end{eqnarray}

\end{lem}
\begin{proof}
To prove $\eqref{qq22lem11eps}$, choose $\widetilde{\psi}=\theta$ and $\psi = \partial_t w$ in the first and the second equations in (SSP), respectively, sum them, integrate over $(n\Delta t, (n+1)\Delta t)$ and use the identities $\frac{d}{dt} \Pi(w)=(\mathcal{F}(w),\partial_t w)$ and $2(a-b)a= (a-b)^2+a^2-b^2$. By $\eqref{qq22lem11eps}$ and the coercivity property of the potential $\Pi$ given in (A2), we have the upper bounds for all the necessary norms of $w$ and $\theta$, which implies that $\mathcal{F}$ is uniformly Lipschitz continuous (see assumption (A1)). Now, to solve (SSP), we write $\theta(t) = \sum_{i=1}^k \alpha_i(t) h_i$ and $w(t) = \sum_{i=1}^k \beta_i(t) s_i$, and by choosing $\widetilde{\psi} = h_1,...h_k$ and $\psi = s_1,...,s_k$ in (SSP), we obtain the following problem for $\alpha(t) = [\alpha_1(t),...,\alpha_k(t)]^T$ and $\beta(t) = [\beta_1(t),...,\beta_k(t)]^T$ in the form
\begin{eqnarray}\label{qq22thesystem}%%%%%%%%%%%%%
\begin{cases}
\dot{\alpha}(t)+\text{diag}(\Xi_k^h)\alpha(t)+M_k^T\dot{\beta}(t) =0, \\[2mm]
\frac{1}{2} \ddot{\beta}(t) + \frac{1}{2} \frac{\dot{\beta}(t) - T_{\Delta t}V(t)}{\Delta t}+ \text{diag}(\Xi_k^s)\beta(t) - M_k \alpha(t)+F(\beta(t)) + \delta E_k \beta(t) = 0, \\[2mm]
\alpha_i(n\Delta t) = (\theta(n\Delta t),h_i), \quad \beta_i(n\Delta t) = (w(n\Delta t), s_i) \quad \dot{\beta}_i(n\Delta t) = (\partial_t w(n\Delta t),s_i),
\end{cases}
\end{eqnarray}%%----------------------------%%
where
\begin{eqnarray*}%%%%%%%%%%%%%
&&V(t) = [(v,s_1), ... , (v,s_k)]^T,\quad \Xi_k^s = [\xi_1^s,...,\xi_k^s]^T, \quad M_k = \langle (\nabla h_i, \nabla s_j) \rangle_{1\leq i,j \leq k}, \\
 &&E_k = \langle \nabla^3 s_i: \nabla^3 s_j \rangle_{1\leq i,j \leq k}, \quad F(\alpha(t)) =\Big [(\mathcal{F}(\sum_{i=1}^k \beta_i(t)), s_1),...,(\mathcal{F}(\sum_{i=1}^k \beta_i(t)), s_k)\Big]^T, \\
&& \Xi_k^h = [\xi_1^h,...,\xi_k^h]^T.
\end{eqnarray*}%%----------------------------%%
Now, the system given in $\eqref{qq22thesystem}$ can be written as an first order ODE system for the unknown $(\alpha(t),\gamma(t),\beta(t))$ in the following form
\begin{eqnarray*}%%%%%%%%%%%%%
\frac{d}{dt} \beta(t) &= &\gamma(t),\\
\frac{d}{dt} \gamma(t) &=& - \frac{\gamma(t) - T_{\Delta t}V(t)}{\Delta t}- 2\text{diag}(\Xi_k^s)\beta(t)+2M_k \alpha(t)-2F(\beta(t)) -2 \delta E_k \beta(t), \\
\frac{d}{dt} \alpha(t)&=&- \text{diag}(\Xi_k^h)\alpha(t)-M_k^T\gamma(t),
\end{eqnarray*}%%----------------------------%%
with the obvious choice for initial data, so by taking into consideration that $F$ is now uniformly Lipschitz, the local solution follows by the standard ODE theory. Now, by estimate $\eqref{qq22lem11eps}$, we obtain the solution on the whole time interval $[n\Delta t, (n+1)\Delta t]$, so the proof is complete.
\end{proof}

\subsubsection{A priori estimates of the fluid-sub problem (FSP)}
\begin{lem}
Any solution of (FSP) on the time interval $[n\Delta t, (n+1)\Delta t]$ satisfies
\begin{align}%%%%%%%%%%%%%
&F^{n+1}(t) + FD^{n+1}(t)+\int_{n \Delta t}^t\int_{\Gamma} S^w (r^\gamma+ \delta r^a )(v-\partial_t w)\nu^{w}\cdot\bb{e}_3  \nonumber \\
&+ \frac{1}{4\Delta t} \int_{n \Delta t}^t \Big(||v||_{L^2(\Gamma)}^2+ ||v - \partial_t w||_{L^2(\Omega)}^2\Big)= F^n(n \Delta t)+\frac{1}{4\Delta t} \int_{n \Delta t}^t||\partial_t w||_{L^2(\Omega)}^2. \label{qq22energyFSP1}
\end{align}%%----------------------------%%
Moreover, if the density $r$ is bounded from above by $C_r$ and if $\Delta t$ is small enough with respect to $C_r$ and $\delta$ so that $\eqref{qq22aaaa}$ holds, we have the following
inequality
\begin{align}%%%%%%%%%%%%%
&F^{n+1}(t) +FD^{n+1}(t) + \frac{1}{4\Delta t} \int_{n \Delta t}^t \Big(||v||_{L^2(\Gamma)}^2+ \frac{1}{2}||v - \partial_t w||_{L^2(\Gamma)}^2\Big) \nonumber \\
& \leq F(n \Delta t)+\frac{1}{4\Delta t} \int_{n \Delta t}^t||\partial_t w||_{L^2(\Omega)}^2 + (\Delta t)^{3/2}. \label{qq22energyFSP2} 
\end{align}%%----------------------------%%
\end{lem}
\begin{proof}
First, we multiply the first equation given in $\eqref{qq22FSP1}$ by $ Jr^{\gamma-1}$ and integrate over $\Omega_\delta$. The first two terms read:
\begin{eqnarray*}%%%%%%%%%%%%%
&&\int_{\Omega_\delta} (\partial_t r -\bb{w} \cdot \nabla^{w} r) J r^{\gamma-1} = \frac{1}{\gamma}\int_{\Omega_\delta^w(t)} \frac{d}{dt} (\rho^{\gamma}) = \frac{1}{\gamma}\Big[ \frac{d}{dt}\int_{\Omega_\delta^{w}(t)} \rho^{\gamma} - \int_{\Gamma^{w}(t)} \rho^\gamma \partial_t w\nu^w\cdot \bb{e}_3 \Big] \\
&&= \frac{1}{\gamma}\Big[ \frac{d}{dt}\int_{\Omega_\delta} J r^{\gamma} - \int_{\Gamma} r^\gamma \partial_t w \underbrace{S^w\nu^w\cdot \bb{e}_3}_{=1} \Big] .
\end{eqnarray*}%%----------------------------%%
By the divergence theorem,
\begin{align*}%%%%%%%%%%%%%
\int_{\Gamma^w(t)} \rho^\gamma v\nu^{w}\cdot\bb{e}_3 = \int_{\Omega_\delta^w(t)} \nabla \cdot (\rho^\gamma \bb{u}) = \gamma \int_{\Omega_\delta^w(t)}\rho^{\gamma-1}\bb{u} \cdot \nabla \rho+ \int_{\Omega_\delta^w(t)} \rho^\gamma(\nabla \cdot \bb{u}),
\end{align*}%%----------------------------%%
and by expressing
\begin{eqnarray*}%%%%%%%%%%%%%
\rho^{\gamma-1} \bb{u} \cdot \nabla\rho = - \rho^\gamma(\nabla \cdot \bb{u}) +\nabla \cdot (\rho \bb{u}) \rho^{\gamma-1},
\end{eqnarray*}%%----------------------------%%
we obtain
\begin{eqnarray*}%%%%%%%%%%%%%
&&\gamma \int_{\Omega_\delta^{w}(t)} \nabla \cdot (\rho \bb{u})\rho^{\gamma-1} = \int_{\Gamma^w(t)} \rho^\gamma v \nu^{w}\cdot\bb{e}_3 + (\gamma-1) \int_{\Omega_\delta^w(t)} \rho^\gamma (\nabla \cdot \bb{u})\\
&&\hspace{1.3in}=\int_{\Gamma} r^\gamma v + (\gamma-1) \int_{\Omega_\delta} J r^\gamma(\nabla^w \cdot \bb{U}).
\end{eqnarray*}%%----------------------------%%
The last term can be expressed as
\begin{eqnarray*}%%%%%%%%%%%%%
-\int_{\Omega_\delta} J \Big( \Delta r + \frac{1}{J}\nabla J\cdot \nabla r\Big) r^{\gamma-1} = (\gamma-1) \int_{\Omega_\delta} J |\nabla r|^2 r^{\gamma -2} .
\end{eqnarray*}%%----------------------------%%
By multiplying the equation $\eqref{qq22FSP1}$ by $\gamma/(\gamma-1)$ and using the obtained calculation
\begin{eqnarray*}%%%%%%%%%%%%%
&&\frac{1}{\gamma-1} \int_{\Omega_\delta}(Jr^\gamma) (t) + \int_{n\Delta t}^t \int_{\Omega_\delta} J r^\gamma (\nabla^w \cdot \bb{U}) +\varepsilon\gamma \int_{n\Delta t}^{t} \int_{\Omega_\delta}J |\nabla r|^2r^{\gamma-2} \\
&&= \int_{n \Delta t}^t\int_{\Gamma} r^\gamma (v-\partial_t w)
= \frac{1}{\gamma-1} \int_{\Omega_\delta} (J \rho^\gamma)(n\Delta t) .
\end{eqnarray*}%%----------------------------%%
Summing up this equality with $\eqref{qq22FSP1}$ multiplied by $ J \delta r^{a-1}$ and integrated on $[n\Delta t, (n+1)\Delta t] \times \Omega_\delta$, with the momentum equation for $\bb{q} = \bb{U}$ and integrated on $[n\Delta t, (n+1)\Delta t] \times \Omega_\delta$, we obtain the equality $\eqref{qq22energyFSP1}$\footnote{Notice that we don't multiply the approximate continuity equation $\eqref{qq22FSP1}_1$ by $\frac{1}{2}\bb{u}^2$ and add to the total energy.}. To obtain $\eqref{qq22energyFSP2}$ from $\eqref{qq22energyFSP1}$, it is enough to prove that the following term can be controlled:
\begin{align}%%%%%%%%%%%%%
\int_{n \Delta t}^t\int_{\Gamma^{w}(t)} (r^\gamma+\delta r^a) (v-\partial_t w) &\leq \int_{n \Delta t}^{(n+1)\Delta t} C(C_{r}^{\gamma}+ \delta C_{r}^{a})||v-\partial_t w||_{L^2(\Gamma)} \nonumber \\
&\leq \int_{n \Delta t}^{(n+1)\Delta t} C(C_r,\delta) ||v-\partial_t w||_{L^2(\Gamma)} \nonumber\\
&\leq (\Delta t)^{3/2} + \frac{1}{8\Delta t} \int_{n \Delta t}^{(n+1)\Delta t}|| v - \partial_t w||_{L^2(\Gamma)}^2, \label{qq22denen1}
\end{align}%%----------------------------%%
for $\Delta t$ small enough such that 
\begin{eqnarray}\label{qq22aaaa}%%%%%%%%%%%%%
8C(C_r,\delta)^2\sqrt{\Delta t} \leq 1
\end{eqnarray}%%----------------------------%%
which then gives us the estimate $\eqref{qq22energyFSP2}$.
\end{proof}

\subsubsection{A priori estimates of the whole system on the time interval $[0, n\Delta t]$}
We will use the following notations throughout the remainder of this paper:
\begin{mydef}
For a given $1<b \leq \infty$ and a domain $E$, denote by
\begin{eqnarray*}%%%%%%%%%%%%%
L^{b^-}(E) &:= \cap_{s<b} L^s(E), \\
W^{{b^-},p}(E) &:= \cap_{s<b} W^{s,p}(E), \\
W^{a,b^-}(E) &:= \cap_{s<b} W^{a,s}(E).
\end{eqnarray*}%%----------------------------%%
We say that a function $f$ converges weakly in $L^{b^-}(E)$ if it converges weakly in all $L^s(E)$, for $s<b$ (analogously for the weak convergence in $W^{b^-,p}(E)$ and $W^{a,b^-}(E)$). We will also write for any $1< b \leq \infty$,
\begin{eqnarray*}%%%%%%%%%%%%%
|| f ||_{L^{b^-}(E)}\leq D,
\end{eqnarray*}%%----------------------------%%
for a constant $D>0$, if for all $1\leq b' <b$, there exists a constant $C(b')$ such that
\begin{eqnarray}\label{qq22star}%%%%%%%%%%%%%
|| f ||_{L^{b'}(E)}\leq C(b')D.
\end{eqnarray}%%----------------------------%%
The same notation will be also used for the Sobolev spaces $W^{b^-,p}(E)$ and $W^{a,b^-}(E)$.
\end{mydef}
\begin{rem}
The constant $ C(b')$ appearing in $\eqref{qq22star}$ will usually blow up as $b'$ approaches $b$. However, this makes no essential difference in the calculation that follows, since at no point the limit $b' \to b$ occurs. Therefore, we will use this notation without additionally emphasizing this.
\end{rem}
We are ready to obtain the uniform bounds of the approximate solutions as follows.
\begin{lem}\label{qq22111}
For a given $\Delta t>0$, $T = N \Delta t$ and $n\leq N$, let $r(t),\mathbf{U}(t), v(t),w(t), \theta(t)
$ be the solutions of (SSP) and (FSP) obtained inductively on the time interval $[0,n\Delta t]$. If $\Delta t$ is small enough with respect to $k,\varepsilon,\delta$ and initial energy so that $\eqref{qq22andanotherone}$ holds, then one has for all $0\leq m \leq n$:
\begin{eqnarray}%%%%%%%%%%%%%
&&S(m\Delta t)+ F(m\Delta t) + \sum_{i=1}^m \Big[ SD(i\Delta t)+ FD(i\Delta t) \Big]+ \ddfrac{1}{4\Delta t}\int_{(m-1)\Delta t}^{m\Delta t}||T_{\Delta t}v||_{L^2(\Gamma)}^2 \nonumber \\
&&+\ddfrac{1}{4\Delta t} \bint_{0}^{m\Delta t} \Big[||\partial_t
w-T_{\Delta t}v||_{L^2(\Gamma)}^2+\ddfrac{1}{2}||v_{\Delta
t,k}-\partial_t w||_{L^2(\Gamma)}^2\Big]\leq C(E_0)+ T(\Delta t)^{1/2}. \quad \quad\quad \label{qq22ineqind}
\end{eqnarray}%%----------------------------%
Consequently, we have the following boundedness: 
\begin{enumerate}%%%%%%%%%%%%%
\item[(i)] $|| \partial_t w ||_{L^\infty(0,n\Delta t; L^2(\Gamma))}+ || w ||_{L^\infty(0,n\Delta t; H^2(\Gamma))}+ \delta || \nabla^3 w ||_{L^\infty(0,n\Delta t; L^2(\Gamma))}~~\\[2mm]
+|| \theta ||_{L^\infty(0,n\Delta t; L^2(\Gamma))}+|| \nabla \theta ||_{L^2(0,n\Delta t; L^2(\Gamma))} \leq C(E_0) ;$
\item[(ii)]$||r ||_{L^\infty(0,n\Delta t; L^\gamma(\Omega_\delta))}+|| r |\bb{U}|^2 ||_{L^\infty(0,n\Delta t; L^1(\Omega_\delta))}+ ||\nabla^w\bb{U} ||_{L^2(0,n\Delta t; L^2(\Omega_\delta))}\leq C(E_0)$, \\[2mm]
$|| r ||_{L^\infty(0,n\Delta t; L^a(\Omega_\delta))} \leq C(E_0,\delta)$,~~$|| r^{\gamma/2} ||_{L^2(0,n\Delta t; H^1(\Omega_\delta))} \leq C(E_0, \varepsilon)$, \\
$||r^{a/2} ||_{L^2(0,n\Delta t; H^1(\Omega_\delta))}\leq C(E_0,\delta,\varepsilon)$;
\item[(iii)] We can choose $T$ only depending on $E_0$ such that $0<c(E_0) \leq J = 1+w \leq C(E_0)$ for all $t \in [0,n\Delta t]$;
\item[(iv)]$||\partial_t w||_{L^2(0,n\Delta t; H^{\frac{1}{2}}(\Gamma) )} \leq C(E_0,\delta)$ and $||\partial_t w||_{L^2(0,n\Delta t; H^{(\frac{1}{2})^-}(\Gamma) )} \leq C(E_0)$; 
\item[(v)]$||\bb{w}||_{L^2(0,n\Delta t; L^{4}(\Gamma) )} \leq C(E_0,\delta)$ and $||\bb{w}||_{L^2(0,n\Delta t; L^{4^-}(\Gamma) )} \leq C(E_0)$; 
\item[(vi)] $\sqrt{\varepsilon}|| \nabla r ||_{L^2(0,n\Delta t; L^2(\Omega_\delta))}\leq C(E_0,\delta)$.
\end{enumerate}%%----------------------------%% 
\end{lem}
\begin{proof}
We sum up $\eqref{qq22lem11eps}$ and $\eqref{qq22energyFSP2}$ for $t=i\Delta t$ into one inequality, and then we sum up these inequalities over $i=1,...,m$, so by telescoping we obtain $\eqref{qq22ineqind}$. Next, from $\eqref{qq22ineqind}$ we have that $S(m\Delta t),F(m\Delta t) \leq C(E_0)$ for all $1\leq m \leq n$, which by $\eqref{qq22ineqind}$ and $\eqref{qq22energyFSP2}$ used at all times $t \in [0,m\Delta t]$, the coercivity estimate of the potential $\Pi$ given in (A2) and $\eqref{munu}$  imply the boundedness given in $(i)$ and $(ii)$. 

Now, from $(i)$, we have that $w$ is uniformly bounded in $C^{0,\alpha}(0,n\Delta t; C^{0,1-2\alpha}(\Gamma))$, and since $J(0) = w_{0,k}+1 \geq c > 0$, one obtains (for say $\alpha = \frac{1}{4}$) for any $t \in [0,n\Delta t]$
\begin{eqnarray*}%%%%%%%%%%%%%
||J(t)||_{C(\Gamma)} \geq || J(0) ||_{C(\Gamma)}-||J(t) - J(0)||_{C^{0,\frac{1}{4}}(0,n\Delta t; C^{0,\frac{1}{2}}(\Gamma))} \geq c - C(E_0)T^{\frac{1}{4}} \geq \frac{c}{2} >0,
\end{eqnarray*}%%----------------------------%%
as $T\leq \big(\frac{c}{2C(E_0)}\big)^4$, so $(iii)$ follows. Next, since $v$ is the ``Lagrangian'' trace of $\bb{u}$ on $\Gamma^{w}(t)$, we have 
\begin{eqnarray*}%%%%%%%%%%%%%
||v||_{L^2(0,n\Delta t; H^{1/2}(\Gamma))}\leq C(E_0,\delta)||\bb{u}||_{L^2(0,n\Delta t; H^1(\Omega_\delta^w(t)))} \leq C(E_0,\delta),
\end{eqnarray*}%%----
so
\begin{eqnarray*}%%%%%%%%%%%%%
||\partial_t w||_{L^2(0,n\Delta t; H^{1/2}(\Gamma))}& \leq& || \partial_t w - v ||_{L^2(0,n\Delta t; H^{1/2}(\Gamma))} + || v ||_{L^2(0,n\Delta t; H^{1/2}(\Gamma))}\\
& \leq& \sqrt{\Delta t}C(k,E_0)+ C(E_0,\delta) \leq C(E_0,\delta),
\end{eqnarray*}%%----------------------------%%
for
\begin{eqnarray}\label{qq22andanotherone}%%%%%%%%%%%%%
\sqrt{\Delta t}C(k,E_0,\delta) \leq 1,
\end{eqnarray}%%----------------------------%%
where we used the equivalence of spatial norms in a finite basis and $\eqref{qq22ineqind}$. The second bound in $(iv)$ can be obtained in the same way by using the weaker trace results for the domains with the H\"{o}lder regularity (see \cite{Boris}), since here we don't rely on the Lipschitz regularity of the domain $\Omega_\delta^w(t)$ that comes from the bound $\delta || \nabla^3 w||_{L^2(0,n\Delta t; L^{2}(\Gamma))}$. This boundedness in $(v)$ is just a consequence of the Sobolev imbedding theorem (see \cite{adams}).

Now, we multiply the continuity equation $\eqref{qq22FSP1}$ by $J{r}$ and integrate over $[0,n \Delta t] \times \Omega_\delta$, to obtain
\begin{eqnarray*}%%%%%%%%%%%%%
&&\frac{1}{2} ||\sqrt{J}(n\Delta t){r} (n\Delta t)||_{L^2(\Omega_\delta)}^2 + \varepsilon \int_0^{n\Delta t} \int_{\Omega_\delta} J|\nabla {r}|^2\\
&&= \frac{1}{2} ||\sqrt{J}(0){r} (0)||_{L^2(\Omega_\delta)}^2 - \frac{1}{2} \int_0^{n\Delta t} \int_{\Omega_\delta}J {r}^2 (\nabla^w \cdot \bb{U}) - \int_0^{n\Delta t}\int_{\Gamma} {r}^2 (v - \partial_t w) .
\end{eqnarray*}%%----------------------------%%
The second term on the right-hand side is majorized by \\ $C_J \sqrt{T} ||\rho||_{L^\infty(0,n\Delta t; L^4(\Omega_\delta))}^2||\nabla^w \cdot \bb{U}||_{L^2([0,n \Delta t] \times \Omega_\delta)}$, provided that $a\geq 4$, while the last term can be bounded in the same way as in $\eqref{qq22denen1}$, so $(vi)$ follows.

\end{proof}

\subsubsection{The solution of (FSP)}
Here we aim to solve the problem (FSP) by the Leray-Schauder fixed-point argument. This will be carried out in Lemma $\ref{qq22contraction3}$. We first solve the continuity equation for given fluid velocity:

\begin{lem}\label{qq22contraction1}
Let $\bb{U} \in \mathcal{F}_k^{n+1}$ with $|| \bb{U}||_{C([n\Delta t,(n+1)\Delta t]; L^2(\Omega_\delta))}^2\leq R$ and let $w \in C^2([n\Delta t,(n+1)\Delta t]; \mathcal{P}_{str}^k)$ be the solution of (SSP) on the time interval $[n\Delta t, (n+1)\Delta t]$. Then the equation
\begin{eqnarray}\label{qq22auxiliary1}
\partial_t r + \nabla^{w}\cdot (r \bb{U}) - \bb{w} \cdot \nabla^{w} r = \varepsilon \Big( \Delta r + \frac{1}{J}\nabla J\cdot \nabla r\Big) , \quad \text{a.e. in } Q_{\delta,T}, 
\end{eqnarray}
with $\partial_n r = 0$ on $\partial \Omega_\delta$ and $r(n\Delta t, \cdot)$ being given by the solution of (FSP) inductively obtained on the previous time interval $[(n-1)\Delta t, n\Delta t]$, has a unique solution $r \in \mathcal{D}^{n+1}$ such that
\begin{eqnarray}\label{qq22sobest}%%%%%%%%%%%%%
||r ||_{\mathcal{D}^{n+1}}^2 \leq C(\varepsilon,E_0,k,R)||r(n\Delta t)||_{ H^{1}(\Omega_\delta)}^2,
\end{eqnarray}%%----------------------------%%
and
\begin{align}\label{qq22lowup}%%%%%%%%%%%%%
\min\limits_{X \in \Omega_\delta}r(n\Delta t,X) e^{ -\int_{n\Delta t}^{t}|| \nabla^w \cdot\bb{U} (\tau)||_{L^\infty(\Omega_\delta)} d \tau} \leq r(t, X) \leq \max\limits_{X \in \Omega_\delta}r(n\Delta t,X) e^{\int_{n \Delta t}^{t}|| \nabla^w \cdot \bb{U} (\tau) ||_{L^\infty(\Omega_\delta)} d \tau}, \quad \quad
\end{align}%%----------------------------%%
for all $t \in [n\Delta t, (n+1)\Delta t]$.
\end{lem}

\begin{proof}
The equation $\eqref{qq22auxiliary1}$ is linear parabolic, so it has a unique solution by the classical theory. To obtain the estimate $\eqref{qq22sobest}$, we first multiply the equation $\eqref{qq22auxiliary1}$ by $r$ and integrate over $\Omega_\delta$ to obtain:
\begin{eqnarray}%%%%%%%%%%%%%
&&\frac{1}{2}\frac{d}{dt} ||r||_{L^2(\Omega_\delta)}^2 +\varepsilon ||\nabla r||_{L^2(\Omega_\delta)}^2 \nonumber\\
&&\leq \Big[|| \nabla^w \cdot r \bb{U} ||_{L^2(\Omega_\delta)} +|| r\nabla^w \cdot \bb{U} ||_{L^2(\Omega_\delta)}  \nonumber \\
&&\quad + || \bb{w}\cdot \nabla^w r||_{L^2(\Omega_\delta)} + || \frac{1}{J} \nabla J\cdot\nabla r||_{L^2(\Omega_\delta)} \Big] ||r||_{L^2(\Omega_\delta)} \nonumber \\
&&\leq \Big[ ||\nabla r||_{L^2(\Omega_\delta)} ||\bb{U}||_{L^\infty(\Omega_\delta)} || \nabla A_w^{-1} ||_{L^\infty(\Omega_\delta)} + || r||_{L^2(\Omega_\delta)} ||\nabla \bb{U}||_{L^\infty(\Omega_\delta)} || \nabla A_w^{-1} ||_{L^\infty(\Omega_\delta)} \nonumber \\
&&+ ||\bb{w}||_{L^\infty(\Omega_\delta)} || \nabla A_w^{-1} ||_{L^\infty(\Omega_\delta)} ||\nabla r||_{L^2(\Omega_\delta)} + || \frac{1}{J}||_{L^\infty(\Omega_\delta)} || \nabla J||_{L^\infty(\Omega_\delta)} ||\nabla r||_{L^2(\Omega_\delta)}
\Big] ||r||_{L^2(\Omega_\delta)} \nonumber \\
&&\leq C(R, E_0,k) \Big[||\nabla r||_{L^2(\Omega_\delta)}+ ||r||_{L^2(\Omega_\delta)} \Big]||r||_{L^2(\Omega_\delta)} \nonumber \\
&&\leq \frac{\varepsilon}{2} ||\nabla r||_{L^2(\Omega_\delta)}^2 + \frac{2}{\varepsilon C(R, E_0,k)^2}|| r||_{L^2(\Omega_\delta)}^2, \label{qq22sobestp1}
\end{eqnarray}%%----------------------------%%
where we used the equivalence of spatial norms in a finite basis, estimate $\eqref{qq22lem11eps}$ combined with the uniform bounds on the interval $[0,n\Delta t]$ given in Lemma $\ref{qq22111}$. Similarly, by multiplying the equation $\eqref{qq22auxiliary1}$ by $\partial_t r$ and $-\Delta r$ respectively, and integrating over $\Omega_\delta$, we obtain
\begin{eqnarray} \label{qq22sobestp2}%%%%%%%%%%%%%
||\partial_t r||_{L^2(\Omega_\delta)}^2 +\frac{\varepsilon}{2}\frac{d}{dt}||\nabla r||_{L^2(\Omega_\delta)}^2 &\leq& \frac{1}{2} || \partial_t r||_{L^2(\Omega_\delta)}^2 + \frac{2}{C(R, E_0,k)^2} ||\nabla r||_{L^2(\Omega_\delta)}^2, \quad
\end{eqnarray}%%----------------------------%%
and
\begin{eqnarray} \label{qq22sobestp3}%%%%%%%%%%%%%
\frac{1}{2}\frac{d}{dt} ||\nabla r||_{L^2(\Omega_\delta)}^2 +\varepsilon ||\Delta r||_{L^2(\Omega_\delta)}^2 &\leq& \frac{\varepsilon}{2}||\Delta r||_{L^2(\Omega_\delta)}^2 + \frac{2}{\varepsilon C(R, E_0,k)^2} ||\nabla r||_{L^2(\Omega_\delta)}^2 . \quad 
\end{eqnarray}%%----------------------------%%
Combining $\eqref{qq22sobestp1}$, $\eqref{qq22sobestp2}$ and $\eqref{qq22sobestp3}$, we obtain
\begin{eqnarray*}%%%%%%%%%%%%%
\frac{d}{dt} || r ||_{H^1(\Omega_\delta)}^2 + ||\partial_t r||_{L^2(\Omega_\delta)}^2 +\varepsilon ||\Delta r||_{L^2(\Omega_\delta)}^2 \leq C(R, E_0, k, \varepsilon) || r||_{H^1(\Omega_\delta)}^2 
\end{eqnarray*}%%----------------------------%%
which by the Gronwall inequality implies
\begin{eqnarray*}%%%%%%%%%%%%%
||\partial_t r||_{L^2(n\Delta t, (n+1)\Delta t; L^2(\Omega_\delta))}^2 + ||\Delta r||_{L^2(n\Delta t, (n+1)\Delta t; L^2(\Omega_\delta))}^2 \leq C(R, E_0, k, \varepsilon) || r(n\Delta t)||_{H^1(\Omega_\delta)}^2 ,
\end{eqnarray*}%%----------------------------%%
so the estimate $\eqref{qq22sobest}$ follows. Next, to prove $\eqref{qq22lowup}$ we introduce the function 
\begin{eqnarray*}%%%%%%%%%%%%%
d(t,X) := r(t,X) - \max\limits_{X \in \Omega_\delta}r(n\Delta t,X) e^{\int_{n \Delta t}^{t}|| \nabla^w \cdot \bb{U} (\tau) ||_{L^\infty(\Omega_\delta)} d \tau},
\end{eqnarray*}%%----------------------------%%
which obviously satisfies the following differential inequality
\begin{eqnarray*}%%%%%%%%%%%%%
\partial_t d + \nabla^{w}\cdot (d \bb{U}) - \bb{w}\cdot \nabla^{w} d - \varepsilon \Big( \Delta d + \frac{1}{J}\nabla J\cdot \nabla d\Big) \leq 0,
\end{eqnarray*}%%----------------------------%%
with $d(n\Delta t) \leq 0$ and $\partial_n d = 0$ on $\partial \Omega_\delta$. Now, multiplying this inequality by $d^+ := \max\{d,0\}$ and integrating over $\Omega_\delta$, we have
\begin{eqnarray*}%%%%%%%%%%%%%
\frac{d}{dt} ||d^+||_{L^2(\Omega_\delta)}^2 + \varepsilon ||\nabla d^+||_{L^2(\Omega_\delta)}^2 \leq C(\varepsilon, E_0,k,R) || d^+||_{L^2(\Omega_\delta)}^2 ,
\end{eqnarray*}%%----------------------------%%
which by Gronwall's inequality gives that $||d^+||_{L^2(\Omega_\delta)}^2 \leq 0$ and consequently the right inequality of $\eqref{qq22lowup}$. The left inequality of $\eqref{qq22lowup}$ is obtained similarly, so the proof is finished.
\end{proof}

\begin{lem}\label{qq22contraction2}%[[[===================]]]
Let $\widetilde{\bb{U}} \in \mathcal{F}_k^{n+1}$ with $|| \widetilde{\bb{U}}||_{C(n\Delta t,(n+1)\Delta t; L^2(\Omega_\delta))}^2\leq R$, let $w \in C^2([n\Delta t,(n+1)\Delta t]; \mathcal{P}_{str}^k)$ be the solution of (SSP) on the time interval $[n\Delta t, (n+1)\Delta t]$ and let $r= r(\widetilde{\bb{U}})$ be the corresponding solution of the equation $\eqref{qq22auxiliary1}$, obtained in Lemma $\ref{qq22contraction1}$. If $\Delta t$ is small enough with respect to $k,\varepsilon,\delta$ and initial energy so that $\eqref{qq22aaaa}$ and $\eqref{qq22aaaand2}$ hold, then the equation
\begin{align}%%%%%%%%%%%%%%%
& \frac{1}{2}\int_{\Omega_\delta}\partial_t J r \bb{U} \cdot \bb{q} +\int_{\Omega_\delta} J r\partial_t \bb{U}\cdot \bb{q} + \frac{1}{2}\int_{\Omega_\delta} J \partial_t r \bb{U} \cdot \bb{q}\nonumber \\
& + \frac{1}{2}\int_{\Omega_\delta}J(r \widetilde{\bb{U}}-r \bb{w}) \cdot \big(\mathbf{q}\cdot 
\nabla^{w}\bb{U} - \bb{U}\cdot \nabla^w \mathbf{q} \big) +\mu \int_{\Omega_\delta} J\nabla^{w}\bb{U} :\nabla^{w}\mathbf{q}\nonumber \\
 &+(\mu +\lambda) \int_{\Omega_\delta} J(\nabla^{w}\cdot \bb{U} )(\nabla^{w}\cdot \mathbf{q}) -\bint_{\Omega_\delta} J(r^\gamma+\delta r^a)(\nabla^{w}\cdot \bb{q} )+\frac{1}{2} \bint_\Gamma\ddfrac{v-\partial_t w}{\Delta t} \psi =0, &\label{qq22auxiliary2}
\end{align}%-----------------------------------%
for all $\bb{q} \in \mathcal{P}_{fl}^k$ with $\psi = \bb{q}_{| \Gamma\times \{0\}}$, has a unique solution $\bb{U} \in C^1([n\Delta t,(n+1)\Delta t]; \mathcal{P}_{fl}^k )$ that satisfies the inequality
\begin{align}%%%%%%%%%%%
& \int_{n\Delta t}^{t} \Big[ \frac{c(\mu,\lambda)}{2} ||\nabla^w  \bb{U} ||_{L^2(\Omega_\delta)}^2+\frac{1}{2\Delta t} || v - \partial_t w||_{L^2(\Omega_\delta)}^2+ \frac{1}{2\Delta t} || v||_{L^2(\Omega_\delta)}^2\Big] \nonumber \\%%
&+\frac{1}{2} \int_{\Omega_\delta} (J r|\bb{U}|^2)(t)\leq C(E_0) +\frac{1}{2} \int_{\Omega_\delta}(J r|\bb{U}|^2)(n \Delta t)
+\frac{1}{\Delta t}\int_{n\Delta t}^{t}|| \partial_t w||_{L^2(\Omega_\delta)}^2, \label{qq22contraction2est}
\end{align}%%%%%%%%%%%%%%
\end{lem}
\begin{proof}
To obtain the a priori estimate $\eqref{qq22contraction2est}$, we first choose $\bb{q}= \bb{U}$ in $\eqref{qq22auxiliary2}$ which by $\eqref{munu}$ gives us
\begin{align}%%%%%%%%%%%
&\frac{1}{2} \int_{\Omega_\delta} (J r|\bb{U}|^2)(t)  +\int_{n\Delta t}^{t} \Big[ c(\mu,\lambda) ||\nabla^w \bb{U} ||_{L^2(\Omega_\delta)}^2 +\frac{1}{2\Delta t} || v - \partial_t w||_{L^2(\Omega_\delta)}^2+ \frac{1}{2\Delta t} || v||_{L^2(\Gamma)}^2\Big]\nonumber \\%%
&\leq \frac{1}{2} \int_{\Omega_\delta} (J r|\bb{U}|^2)(n \Delta t)+\int_{n\Delta t}^{(n+1)\Delta t} \bint_{\Omega_\delta} J \big[ (r^\gamma+\delta r^{a}) \nabla^{w}\cdot \bb{U} \big]
+\frac{1}{\Delta t}\int_{n\Delta t}^{t}|| \partial_t w||_{L^2(\Omega_\delta)}^2. \nonumber 
\end{align}%%%%%%%%%%%%%%
For the second term on the right hand side, one has
\begin{eqnarray*}%%%%%%%%%%%%%
&&\int_{n\Delta t}^t \int_{\Omega_\delta} J (r^\gamma+ \delta r^a ) (\nabla^w \cdot \bb{U}) \\
&&\leq \int_{n\Delta t}^{(n+1)\Delta t} \int_{\Omega_\delta} \frac{c(\mu,\lambda)}{6} ||\nabla^w \cdot  \bb{U} ||_{L^2(\Omega_\delta)}^2 + \frac{6 \max \{ C_r^\gamma, \delta C_r^a \}C_J\mathcal{M}(\Omega_\delta) }{c(\mu,\lambda)}\Delta t \\
&&\leq \int_{n\Delta t}^{(n+1)\Delta t} \int_{\Omega_\delta}\frac{c(\mu,\lambda)}{2} ||\nabla^w  \bb{U} ||_{L^2(\Omega_\delta)}^2 + C(E_0), 
\end{eqnarray*}%%----------------------------%%
for $\Delta t$ small enough such that\footnote{Notice that $C_r$ only depends on $k,\varepsilon,\delta$ and initial energy by $\eqref{qq22lowup}$.}
\begin{eqnarray}\label{qq22aaaand2}%%%%%%%%%%%%%
 \frac{6 \max \{ C_r^\gamma, \delta C_r^a \}C_J\mathcal{M}(\Omega_\delta) }{c(\mu,\lambda)}\Delta t  \leq C(E_0),
\end{eqnarray}%%----------------------------%%
where $\mathcal{M}(\Omega_\delta)$ is the measure of $\Omega_\delta$. Therefore, the inequality $\eqref{qq22contraction2est}$ follows.

Next, for simplicity we denote $\bb{g}_i:=\bb{f}_i$ for $1 \leq i \leq k$ and $\bb{g}_i = \text{Ext}[s_{i-k}]$ for $i\leq k+1 \leq 2k$. Now, by writing $\bb{U} = \sum_{i=1}^{2k} \alpha_i(t) \bb{g}_i$, and choosing $\bb{q} = \bb{g}_1, ... ,\bb{g}_{2k}$ in $\eqref{qq22auxiliary2}$, respectively, we obtain a system of $2k$ equations, or in other form, an ODE system for unknown $\bb{\alpha }= [\alpha_1, ...,\alpha_{2k}]^T$, that can be written as
\begin{eqnarray*}%%%%%%%%%%%%%
M( J,r) \dot{\alpha } = N(J, r , \widetilde{\bb{U}}, w) \bb{\alpha },
\end{eqnarray*}%%----------------------------%%
where
\begin{eqnarray*}%%%%%%%%%%%%%
M( J,r) = \langle \int_{\Omega_\delta} Jr \bb{g}_i \cdot \bb{g}_j \rangle_{1 \leq i,j \leq k },
\end{eqnarray*}%%----------------------------%%
and $N(J, r , \widetilde{\bb{U}}, w)$ can be directly expressed from the remaining terms in $\eqref{qq22auxiliary2}$. To obtain a local solution, it is enough to prove that $M(J,r)$ is positive definite, i.e. that for every non-zero vector $\bb{a} \in \mathbb{R}^k$, $\bb{a}^T M( J,r) \bb{a}>0$. Since we can write $M(J,r) = f(J,r) \widetilde{\otimes} f(J,r)$, where $f(J,r) := [\sqrt{Jr} \bb{g}_1, ... , \sqrt{Jr} \bb{g}_k]$ and $a\widetilde{\otimes} b := \langle \int_{\Omega_\delta} a_i \cdot b_j\rangle_{1\leq i,j \leq k}$ for $a,b \in \mathcal{P}_{fl}^k$, we directly have that 
\begin{eqnarray*}%%%%%%%%%%%%%
\bb{v}^T M( J,r) \bb{v} = \bb{v}^T f(J,r) \widetilde{\otimes} f(J,r) \bb{v} = \int_{\Omega_\delta}\big( \sum_{i=1}^k \sqrt{Jr} v_i \bb{g}_i \big)^2 \geq c_J c_r \int_{\Omega_\delta}\big( \sum_{i=1}^k v_i \bb{g}_i \big)^2.
\end{eqnarray*}%%----------------------------%%
where $c_J$ and $c_r$ are the lower bounds for $J$ and $r$. The last term is obviously positive since $\{\bb{g}_i\}_{1\leq i\leq 2k}=\{\bb{f}_i, \text{Ext}[s_{i+k}] \}_{ 1\leq i \leq k}$ are linearly independent (see Remark $\ref{qq22isbasis}$). Now, we use the estimate $\eqref{qq22contraction2est}$ to prove that the solution exists on the whole time interval $[n\Delta t, (n+1)\Delta t]$ and this finishes the proof.
\end{proof}

\begin{lem}\label{qq22contraction3}
Let $\Delta t$ be small enough with respect to the approximation parameters $k,\varepsilon,\delta$, initial energy and given constants in the system so that $\eqref{qq22aaaa}$, $\eqref{qq22aaaand2}$ and $\eqref{qq22aaand}$ hold. Then, the system (FSP) has a solution $(r^{n+1},\bb{U}^{n+1})\in \mathcal{D}^{n+1} \times \mathcal{F}^{n+1}$.
\end{lem}
\begin{proof}
The solution is obtained by the fixed-point argument. We first introduce the iteration set as
\begin{eqnarray*}%%%%%%%%%%%%%
S_R : = \big\{\bb{U} \in C([n\Delta t, (n+1)\Delta t]; \mathcal{P}_{fl})&:& \bb{U}(n\Delta t) = \bb{U}^n(n\Delta t), \\ 
&&||\bb{U}||_{C([n\Delta t, (n+1)\Delta t] ; L^2(\Omega_\delta))} \leq R \big\},
\end{eqnarray*}%%----------------------------%%
where $R=\frac{10C(E_0)}{c_J c_{r}}$, the constant $C(E_0)$ is given in Lemma $\ref{qq22111}$, $c_J$ is the lower bound of the Jacobian and $0<c_r \leq \text{min}_{X \in \Omega_\delta}r(n\Delta t,X)$. Define the operator as
\begin{eqnarray*}%%%%%%%%%%%%%
\mathcal{A}:S_R &\to& S_R, \\
\widetilde{\bb{U}} &\mapsto &\bb{U} = \bb{U}(r(\widetilde{\bb{U}} )),
\end{eqnarray*}%%----------------------------%%
where $r(\widetilde{\bb{U}} )$ is the solution of the equation $\eqref{qq22auxiliary1}$ for given $\widetilde{\bb{U}}$ and $\bb{U}(r(\widetilde{\bb{U}}))$ is the solution of the equation $\eqref{qq22auxiliary2}$ for given $r(\widetilde{\bb{U}})$, obtained in Lemmas $\ref{qq22contraction1}$ and $\ref{qq22contraction2}$, respectively.\\

\textit{Step 1: Boundedness}. Let $\widetilde{\bb{U}} \in S_R$. First, let us prove that $c_r \leq \text{min}_{X \in \Omega_\delta}~ r(n\Delta t)$ and $C_r \geq \text{max}_{X \in \Omega_\delta}~ r(n\Delta t)$ can be chosen uniformly with respect to $n$ and $\Delta t$. From $\eqref{qq22lowup}$, we inductively have that
\begin{eqnarray}%%%%%%%%%%%%%
\min\limits_{X \in \Omega_\delta}r_{0,\delta}(X) e^{ -\int_{0}^{n\Delta t}|| \nabla^w \cdot\bb{U} (\tau)||_{L^\infty(\Omega_\delta)} d \tau} \leq r(n\Delta t, X) \leq \max\limits_{X \in \Omega_\delta}r_{0,\delta}(X) e^{\int_{0}^{n\Delta t}|| \nabla^w \cdot \bb{U} (\tau) ||_{L^\infty(\Omega_\delta)} d \tau}.\nonumber \\ \label{qq22uniformrhobound1}
\end{eqnarray}%%----------------------------%%
From the equivalence of the spatial norms in a finite basis, we have 
\begin{eqnarray}\label{qq22uniformrhobound2}%%%%%%%%%%%%%
\int_{0}^{n\Delta t}|| \nabla^w \cdot \bb{U} (\tau) ||_{L^\infty(\Omega_\delta)} d \tau \leq T^{1/2} C_k || \nabla^w \cdot \bb{U}||_{L^2(0, n\Delta t; L^2(\Omega_\delta))},
\end{eqnarray}%%----------------------------%%
so by combining $\eqref{qq22uniformrhobound1}$, $\eqref{qq22uniformrhobound2}$ and $\eqref{qq22ineqind}$, we obtain
\begin{eqnarray*}%%%%%%%%%%%%%
c_r : = \frac{1}{2C( k, E_0)}\min\limits_{X \in \Omega_\delta}r_{0,\delta}(X) \leq r(n\Delta t, X) \leq 2C(k, E_0)\max\limits_{X \in \Omega_\delta}r_{0,\delta}(X) =: C_r.
\end{eqnarray*}%%----------------------------%%
With this in mind, the density $r(\widetilde{\bb{U}})$, from the inequality $\eqref{qq22lowup}$, satisfies
\begin{eqnarray}%%%%%%%%%%%%%
r(\widetilde{\bb{U}}) &\leq C_r \text{exp}(\Delta t ||\nabla^w \cdot\widetilde{\bb{U}}||_{L^\infty([n\Delta t, (n+1)\Delta t] \times \Omega_\delta)}) &\leq C_r \text{exp}(\Delta t C(k,E_0,R)),\quad\label{qq22lowup2} \\
r(\widetilde{\bb{U}}) &\geq c_r \text{exp}(-\Delta t ||\nabla^w \cdot\widetilde{\bb{U}}||_{L^\infty([n\Delta t, (n+1)\Delta t] \times \Omega_\delta)}) &\geq c_r \text{exp}(-\Delta t C(k,E_0,R)),\quad\label{qq22lowup3}
\end{eqnarray}%%----------------------------%%
where we bounded $||\nabla^w \cdot\widetilde{\bb{U}}||_{L^\infty([n\Delta t, (n+1)\Delta t] \times \Omega_\delta)}$ by using the equivalence of the spatial norms in a finite basis, the bound for $\widetilde{\bb{U}}$ from the iteration set $S_R$ and Lemma $\ref{qq22111}(ii)$ to bound the components of the transformed divergence $\nabla^w \cdot \bb{U}$ which depend on $\nabla A_w$. Now, choosing $\Delta t$ small enough such that, say
\begin{eqnarray}\label{qq22aaand}%%%%%%%%%%%%%
\text{exp}(\Delta t C(k,E_0,R)) \leq 2
\end{eqnarray}%%----------------------------%%
the estimates $\eqref{qq22lowup2}$ and $\eqref{qq22lowup3}$ imply
\begin{eqnarray}\label{qq22zvezda}%%%%%%%%%%%%%
\frac{1}{2}c_r \leq r(\widetilde{\bb{U}}) \leq 2 C_r.
\end{eqnarray}%%----------------------------%%
Now, we can bound the terms on the right-hand side given in the inequality $\eqref{qq22contraction2est}$ in Lemma $\ref{qq22contraction2}$ as
\begin{eqnarray*}%%%%%%%%%%%%%%%%%%%%%%%%
\frac{1}{2} \int_{\Omega_\delta} (J r|\bb{U}|^2)(n \Delta t) \leq 2C(E_0), \quad \quad 
\frac{1}{\Delta t}\int_{n\Delta t}^{t}|| \partial_t w||_{L^2(\Omega_\delta)}^2 \leq 2C(E_0)
\end{eqnarray*}%%%%%%%%%%
from $\eqref{qq22ineqind}$ (here w.l.o.g., we bound $T(\Delta t)^{1/2}\leq C(E_0)$ in the equality $\eqref{qq22ineqind}$), so by $\eqref{qq22zvezda}$ we conclude $\bb{U}(r(\widetilde{\bb{U}})) \in S_R$. \\ \\
\textit{Step 2: A compact subset}. 
Next, notice that we can bound
\begin{eqnarray*}%%%%%%%%%%%%%
||\partial_t \bb{U} (t)||_{L^2(\Omega_\delta)} \leq C(k,\varepsilon, E_0)\Big(1+\frac{1}{\sqrt{\Delta t}}\Big),
\end{eqnarray*}%%----------------------------%%
by taking $\bb{q} = \partial_t \bb{U}$ in $\eqref{qq22FSP2}$, integrating over $\Omega_\delta$ and using $\eqref{qq22contraction2est}$. Now, this implies 
\begin{eqnarray*}%%%%%%%%%%%%%
\mathcal{A}: S_R &\to &\Big\{\bb{U} \in \mathcal{F}_{k}^{n+1}: \bb{U}(n\Delta t) = \bb{U}^n(n\Delta t), ~ 
||\bb{U}||_{C([n\Delta t, (n+1)\Delta t] ; L^2(\Omega_\delta))} \leq R ,\\
& &||\partial_t \bb{U} ||_{C([n\Delta t, (n+1)\Delta t]; L^2(\Omega_\delta))} \leq C(k,\varepsilon, E_0)\Big(1+\frac{1}{\sqrt{\Delta t}}\Big)\Big\} \Subset S_R,
\end{eqnarray*}%%----------------------------%%
so by the Leray-Schauder fixed-point theorem, we obtain a solution $(r^{n+1},\bb{U}^{n+1})$ of the system (FSP), and the proof is complete.
\end{proof}
\subsubsection{Coupling back the decoupled system}\label{qq22sec}
Now we have solved the systems $\eqref{qq22SSP}$ and $\eqref{qq22FSP1}$ and inductively obtained the solutions $( r,\bb{U}, w, \theta)$ on the whole time interval $[0,T]$. These solutions also satisfy the energy estimates given in Lemma $\ref{qq22111}$ on $[0,T]$ and the following system:
\begin{eqnarray}\label{qq22thirdlevelsystem}%%%%%%%%%%%%%
\begin{cases}
\bint_{\Gamma} \partial_t \theta \widetilde{\psi} + \bint_{\Gamma}\nabla \theta\cdot \nabla \widetilde{\psi} +\bint_{\Gamma}\nabla \partial_t w \cdot \nabla \widetilde{\psi} = 0, \\[3.5mm]
\partial_t r - \bb{w} \cdot \nabla^{w} r + \nabla^{w}\cdot (r \bb{U}) = \ddfrac{1}{J} \varepsilon \nabla \cdot ( \nabla r J ), \quad \quad \text{a.e. in } Q_{\delta,T} \\[3.5mm] %%
\bint_{\Omega_\delta} \partial_t (Jr \bb{U}) \cdot \bb{q} + \bint_{\Omega_\delta} J\big[(r \bb{w}-r \bb{U}) \cdot \nabla^{w}\bb{q}\big]\cdot \mathbf{u} +\mu\bint_{\Omega_\delta} J\nabla^{w} \mathbf{u}:\nabla^{w}\bb{q} \\[2.5mm]%%
\hspace{.5in}+(\mu +\lambda) \bint_{\Omega_\delta} J (\nabla^{w}\cdot \mathbf{u})(\nabla^{w}\cdot \bb{q}) -\bint_{\Omega_\delta} (J (r^\gamma + \delta r^a) \nabla^{w}\cdot \bb{q})\\[2.5mm]
\hspace{.5in}+ \varepsilon\bint_{\Omega_\delta} J \nabla r \cdot ( \bb{q}\cdot \nabla \bb{U}+ \bb{U}\cdot\nabla \bb{q})
+ \ddfrac{1}{2}\bint_\Gamma\partial_t^2 w\bb{\psi} + \ddfrac{1}{2} \bint_\Gamma\ddfrac{v-T_{\Delta t} v}{\Delta t}\bb{\psi}\\[2.5mm] %%
\hspace{.5in}
+\bint_\Gamma\Delta w \Delta \bb{\psi}+\bint_\Gamma \mathcal{F}(w) \psi -\bint_\Gamma \nabla \theta\cdot \nabla \psi+ \delta\bint_\Gamma \nabla^3 w: \nabla^3 \psi=0, 
\end{cases}
\end{eqnarray}%%----------------------------%%
for $\widetilde{\psi}\in \mathcal{P}_{heat}^k$ and $(\bb{q},\psi) \in \mathcal{P}_{fl}^k\times \mathcal{P}_{str}^k$ satisfying $\bb{q}_{| \Gamma\times \{0\}} = \psi\bb{e}_3$, where the third equation was obtained by summing the momentum equation $\eqref{qq22FSP1}_2$ with test functions $(\bb{q},\psi)$, the structure equation $\eqref{qq22SSP}_1$ with test function $\psi$, and the continuity equation $\eqref{qq22FSP1}_1$ multiplied by $\frac{1}{2} J \bb{U}\cdot \bb{q}$ and integrated over $\Omega_\delta$. This system will eventually converge to the desired original weak form of the problem $\eqref{qq22structureeqs}-\eqref{qq22initialdata}$ in the sense of Definition $\ref{qq22weaksolmoving}$.

\section{The operator splitting time step and the number of Galerkin basis functions limits}\label{qq22section4}
\sectionmark{Time step and the number of Galerkin basis functions limits}
The approximate solutions contructed in the previous section by solving (FSP) and (SSP), inductively, on the whole time interval $[0,T]$ satisfying the system $\eqref{qq22thirdlevelsystem}$ will be denoted as $(r_{\Delta t,k},\bb{U}_{\Delta t,k}, w_{\Delta t,k}, \theta_{\Delta t,k})$. We first introduce the function spaces
\begin{eqnarray*}%%%%%%%%%%%%%
\mathcal{W}_{FS,\varepsilon} (0,T)&: =& \big\{ (w,\bb{U}): w \in L^\infty(0,T; H_0^2(\Gamma)) \cap W^{1,\infty}(0,T; L^2(\Gamma)), \\
&&\nabla^w \bb{U} \in L^2(0,T; L^2(\Omega_\delta)),  \gamma_{|\Gamma\times \{0\}} \bb{U} = \partial_t w \bb{e}_3, \bb{U} = 0 \text{ on } \partial \Omega_\delta\setminus \Gamma\times\{0\} \big\}, \\
\mathcal{W}_{D,\varepsilon}(0,T) &: =& \big\{ {r} \in L^\infty(0,T; L^a(\Omega_\delta))\cap L^{\frac{4}{3}a}(Q_{\delta,T}): \partial_t r, \Delta r \in L^{\frac{6}{5}} (0,T; L^{\frac{36}{25}}(\Omega_\delta)),\\
&&\nabla r \in L^{2}(Q_{\delta,T}(\Omega_\delta)) \big\}.
\end{eqnarray*}%%----------------------------%%
Since $||\nabla A_w||_{L^\infty(Q_{\delta,T})} \leq C(E_0,\delta)$, we know
\begin{eqnarray*}%%%%%%%%%%%%%
\mathcal{W}_{FS,\varepsilon} (0,T)& =& \big\{ (w,\bb{U}): w \in L^\infty(0,T; H_0^2(\Gamma)) \cap W^{1,\infty}(0,T; L^2(\Gamma)),\\
&& \bb{U} \in L^2(0,T; H^{1}(\Omega_\delta)),  \gamma_{|\Gamma\times \{0\}} \bb{U} = \partial_t w \bb{e}_3, ~\bb{U} = 0 \text{ on } \partial \Omega_\delta\setminus \Gamma\times\{0\} \big\}.
\end{eqnarray*}%%----------------------------%%

We now introduce the following weak solution, suitable for the limiting process of the functions that solve the system $\eqref{qq22thirdlevelsystem}$:
\begin{mydef}\label{qq22solweakartvis}
We say that $r\in \mathcal{W}_{D,\varepsilon}^w(0,T), ~(\bb{U},w) \in \mathcal{W}_{FS,\varepsilon}^w(0,T), ~\theta \in \mathcal{W}_H(0,T)$ is the weak solution to the coupled fluid-structure interaction problem with artificial density damping and artificial pressure on the fixed reference domain $\Omega_\delta$ if:
\begin{enumerate}
\item The following heat equation
\begin{align}
\int_{\Gamma_T}\theta \partial_t \widetilde{\psi} -\int_{\Gamma_T} \nabla \theta \cdot \nabla \widetilde{\psi}+\int_{\Gamma_T}\nabla w \cdot \nabla \partial_t \widetilde{\psi}= \int_0^T \frac{d}{dt}\int_{\Gamma}\theta\widetilde{\psi} +\int_0^T \frac{d}{dt} \int_{\Gamma}\nabla w \cdot \nabla \widetilde{\psi}, \label{qq22heat1}
\end{align}
holds for all $\widetilde{\psi} \in C^\infty(\Gamma_T)$.
\item The following damped continuity equation
\begin{eqnarray}\label{qq22dampcont1}%%%%%%%%%%%%%
\partial_t {r} + \nabla^{w}\cdot ({r} \bb{U}) - \bb{w}\cdot \nabla^{w} {r} = \varepsilon \Big( \Delta {r} + \frac{1}{J} \nabla J\cdot \nabla {r}\Big) , 
\end{eqnarray}%%----------------------------%%
holds a.e. in $Q_{\delta,T}$.
\item The following coupled momentum equation holds
\begin{eqnarray}%%%%%%%%%%%%%%%
&& \int_{Q_{\delta,T}}J{r} \bb{U}\cdot \partial_t \bb{q} +\int_{Q_{\delta,T}}J\big[({r} \bb{U}-{r} \bb{w}) \cdot \nabla^{w}\bb{q}\big]\cdot \mathbf{u}-\mu\int_{Q_{\delta,T}} J\nabla^{w} \mathbf{u}:\nabla^{w}\bb{q} \nonumber \\
&&-(\mu +\lambda) \int_{Q_{\delta,T}}J (\nabla^{w}\cdot \mathbf{u})(\nabla^{w}\cdot \bb{q})
+\int_{Q_{\delta,T}} (J ({r}^\gamma + \delta {r}^a) \nabla^{w}\cdot \bb{q}) \nonumber \\
&&-\varepsilon \int_{Q_{\delta,T}}  J \nabla {r} \cdot ( \bb{q}\cdot\nabla \bb{U} + \bb{U} \cdot\nabla \bb{q})+\bint_{\Gamma_T} \partial_t w \partial_t \psi -\bint_{\Gamma_T}\Delta w \Delta \psi \nonumber\\
&&-\bint_{\Gamma_T}\mathcal{F}(w) \psi +\bint_{\Gamma_T}\nabla \theta \cdot \nabla \psi-\bint_{\Gamma_T}\delta \nabla^3 w : \nabla^3 \psi \nonumber\\
&&=\int_0^T \frac{d}{dt} \int_{\Omega_\delta}J {r} \bb{U}\cdot\bb{q} +\int_0^T \frac{d}{dt} \int_{\Gamma} \partial_t w \psi, \quad\quad\quad \label{qq22summedup1}
\end{eqnarray}
for all $\bb{q} \in C_0^\infty(Q_{\delta,T,\Gamma})$ and $\psi\in C_0^\infty( \Gamma_T )$, satisfying $\bb{q}_{|\Gamma\times \{0\}} = \psi\bb{e}_3$.
\end{enumerate}
\end{mydef}
The main result of this section is the following one:
\begin{thm}\label{qq22artdisth}
There exists a weak solution $({r},\bb{U},w,\theta)$ in the sense of Definition $\ref{qq22solweakartvis}$ that satisfies the following energy inequality for all $t \in [0,T]$,
\begin{eqnarray*}%%%%%%%%%%%%%
E_\delta^w(t) + D_\varepsilon^w(t) \leq E^w(0),
\end{eqnarray*}%%----------------------------%%
where 
\begin{eqnarray*}%%%%%%%%%%%%%
E_\delta^w(t) &:=& E^w(t) + \frac{1}{2} \delta||\nabla^3 w(t)||_{L^2(\Gamma)}^2 + \delta \int_{\Omega_\delta} \frac{(Jr^a)(t)}{a-1}, \\
D_\varepsilon^w(t)& :=& D^w(t) +\varepsilon \int_0^t \int_{\Omega_\delta}J|\nabla r|^2(\gamma r^{\gamma-2} +\delta a r^{a-2})
\end{eqnarray*}%%----------------------------%%
with $E^w(t)$ and $D^w(t)$ being defined in $\eqref{qq22energies2}$.
\end{thm}
We will prove this result by passing the limit in $\Delta t$ and $k$ in the system $\eqref{qq22thirdlevelsystem}$. 

\subsection{Passing to the limit}
Throughout the remainder of this section, we assume that $\Delta t$ is small enough with respect to the approximation parameters $k,\varepsilon,\delta$, initial energy and given constants in the system, for which Lemmas $\ref{qq22111}$ and $\ref{qq22contraction3}$ hold. The goal is to prove that the equations $\eqref{qq22thirdlevelsystem}_1$, $\eqref{qq22thirdlevelsystem}_2$ and $\eqref{qq22thirdlevelsystem}_3$ converge to $\eqref{qq22heat1}$, $\eqref{qq22dampcont1}$ and $\eqref{qq22summedup1}$, respectively. We start with the following result:
\begin{lem}\label{qq22weakconvthirdlevel}
The following convergences hold as $\Delta t \to 0, k \to +\infty$:
\begin{enumerate}%%%%%%%%%%%%%
\item[(i)] $\bb{U}_{k,\Delta t} \rightharpoonup \bb{U}, \text{ weakly in } L^2(0,T; H^{1}(\Omega_\delta))$;
\item[(ii)] ${r}_{k,\Delta t} \rightharpoonup {r}, \text{ weakly* in } L^\infty(0,T; L^a (\Omega_\delta))$, ~~ $\nabla {r}_{\Delta t,k} \rightharpoonup \nabla {r}, \text{ weakly in } L^2(Q_{\delta,T})$;
\item[(iii)] Independently of $\delta$, we have:
\begin{enumerate}
\item[(iiia)] $w_{\Delta t,k} \rightharpoonup w$, weakly* $L^\infty(0,T; H_0^2(\Gamma))$ and $W^{1,\infty}(0,T; L^2(\Gamma))$;
\item[(iiib)] $w_{\Delta t,k} \to w$, in $C^{0,\alpha}([0,T]; H^{2\alpha}(\Gamma))$, for $0<\alpha<1$ ;
\item[(iiic)] $J_{\Delta t,k} \to J$ and $1/J_{\Delta t,k} \to 1/J$, in $C^{0,\alpha}([0,T]; C^{0,1-2\alpha}(\Gamma))$ for $0<\alpha<1/2$;
\item[(iiid)] $\theta_{\Delta t,k} \to \theta$ weakly* in $L^\infty(0,T; L^2(\Gamma))$ and weakly in $L^2(0,T; H^1(\Gamma))$;
\item[(iiie)] $v_{\Delta t,k} \rightharpoonup \partial_t w$ weakly in $L^2(0,T; L^{2} (\Gamma))$;
\item[(iiif)] $\mathcal{F}(w_{\Delta t,k}) \to \mathcal{F}(w)$ in $C([0,T]; H^{-2}(\Gamma))$;
\item[(iiig)] $\int_0^T \int_\Gamma \ddfrac{v_{\Delta t,k}-T_{\Delta t} v_{\Delta t,k}}{\Delta t} P_{str}^k(\bb{\psi}) \to - \int_0^T \int_\Gamma \partial_t w \partial_t \psi + \int_0^T \frac{d}{dt} \int_\Gamma \partial_t w \psi$, \\for any $\psi \in C_0^{\infty}(\Gamma_T)$;
\end{enumerate}%%----------------------------%%
\item[(iv)] $ \nabla ^3 w_{\Delta t,k} \rightharpoonup \nabla^3 w$, weakly* in $L^\infty(0,T; L^2(\Gamma))$;
\item[(v)] ${r}_{\Delta t,k} \to {r}$, in $L^{(\frac{4}{3} a)^-}(Q_{\delta,T})$;
\item[(vi)] ${r}_{\Delta t,k}\bb{U}_{\Delta t,k} \rightharpoonup {r}\bb{U}$, weakly in $L^2(0,T; L^{\frac{6a}{a+6}}(\Omega_\delta) )$ and weakly* in $L^\infty(0,T; L^{\frac{2a}{a+1}}(\Omega_\delta))$;
\item[(vii)] $J_{\Delta t,k}{r}_{\Delta t,k} \bb{U}_{\Delta t,k}\otimes \bb{U}_{\Delta t,k} \rightharpoonup J{r} \bb{U} \otimes \bb{U}$, in $L^1(Q_{\delta,T})$;
\item[(viii)] $J_{\Delta t,k}{r}_{\Delta t,k} \bb{U}_{\Delta t,k}\otimes \bb{w}_{\Delta t,k} \rightharpoonup J{r} \bb{U} \otimes \bb{w}$, in $L^1(Q_{\delta,T})$.
\end{enumerate}%%----------------------------%%
\end{lem}
\begin{proof}
The convergences $(i)$,$(ii)$ and $(iiia)-(iiie)$ and $(iv)$ follow from Lemma $\ref{qq22111}$. Now, $(iiif)$ follows by the assumption $(A1)$ for the nonlinear function $\mathcal{F}$ and $(iiib)$ for $2\alpha > 2-\epsilon$, and to prove $(iiig)$, we calculate
\begin{eqnarray*}%%%%%%%%%%%%%
&&\int_0^T \int_\Gamma \ddfrac{v_{\Delta t,k}-T_{\Delta t} v_{\Delta t,k}}{\Delta t} P_{str}^k(\bb{\psi}) =\ddfrac{1}{\Delta t} \Big[ \int_{\Delta t}^T \int_\Gamma v_{\Delta t,k} P_{str}^k(\bb{\psi}) - \int_{0}^{T-\Delta t} \int_\Gamma v_{\Delta t,k} T_{-\Delta t}P_{str}^k(\bb{\psi}) \Big]\\
&&=- \int_{\Delta t}^{T-\Delta t} \int_\Gamma v_{\Delta t,k}\frac{ T_{-\Delta t}P_{str}^k(\bb{\psi}) - P_{str}^k(\bb{\psi})}{\Delta t} + \ddfrac{1}{\Delta t}\int_{T- \Delta t}^T \int_\Gamma v_{\Delta t,k} P_{str}^k(\bb{\psi}) \\
&& - \ddfrac{1}{\Delta t}\int_{0}^{\Delta t} \int_\Gamma v_{\Delta t,k} T_{-\Delta t}P_{str}^k(\bb{\psi}) \to - \int_0^T \int_\Gamma \partial_t w \partial_t \psi + \int_0^T \frac{d}{dt} \int_\Gamma \partial_t w \psi,
\end{eqnarray*}%%----------------------------%%
as $\Delta t\to 0, k \to +\infty$. Next, from the equation $\eqref{qq22thirdlevelsystem}_2$ and the estimates given in Lemma $\ref{qq22111}$, we have
\begin{align*}%%%%%%%%%%%%%
||\partial_t (J_{\Delta t,k} {r}_{\Delta t,k})||_{L^2(0,T; [W_0^{1, p}(\Omega_\delta)]')} &\leq C(E_0), \quad p>2a/(a+1),
\end{align*}%%----------------------------%%
and by using the estimates for ${r}_{\Delta t,k}$ and $w_{\Delta t,k}$ in Lemma $\ref{qq22111}$, we can also bound
\begin{eqnarray*}%%%%%%%%%%%%%
&&|| \nabla (J_{\Delta t,k} {r}_{\Delta t,k})||_{L^2(Q_{\delta,T})}\\
&& \leq ||\nabla {r}||_{L^2(Q_{\delta,T})} ||J||_{L^\infty(Q_{\delta,T})} + ||{r}||_{L^\infty(0,T; L^a(\Omega_\delta))} ||\nabla J||_{L^2(0,T; L^{\frac{2a}{a-2}}( \Omega_\delta))} \leq C(E_0, \varepsilon),
\end{eqnarray*}%%----------------------------%% 
which by interpolation and the Aubin-Lions lemma implies
\begin{eqnarray*}%%%%%%%%%%%%%
J_{\Delta t,k} {r}_{\Delta t,k} \to J {r},~ \text{in } L^2(Q_{\delta,T}),
\end{eqnarray*}%%----------------------------%%
so
\begin{eqnarray}\label{qq22rhoL2strong1}%%%%%%%%%%%%%
{r}_{\Delta t,k} \to{r},~ \text{in } L^2(Q_{\delta,T}).
\end{eqnarray}%%----------------------------%%
From Lemma $\ref{qq22111}$, we have that $||{r}_{\Delta t,k}^{a/2} ||_{L^2(0,n\Delta t; H^1(\Omega_\delta))}\leq C(E_0,\varepsilon, \delta)$ which by the Sobolev imbedding implies that $||{r}_{\Delta t,k}||_{L^{a/2}(0,n\Delta t; L^{3a}(\Omega_\delta))}\leq C(E_0, \varepsilon,\delta)$. Since $||{r}_{\Delta t,k}||_{L^\infty(0,n\Delta t; L^a(\Omega_\delta))}\\ \leq C(E_0, \delta)$, by the interpolation of the Lebesque spaces we obtain $||{r}_{\Delta t,k} ||_{L^{\frac{3}{2} a}(Q_{\delta,T})}\leq C(E_0,\varepsilon,\delta)$ which combined with $\eqref{qq22rhoL2strong1}$ implies $(v)$.

Next, from Lemma $\ref{qq22111}$ and the imbedding of $\bb{U} \in L^2((0,T); H^{1}(\Omega_\delta))$ into \\$L^2((0,T); L^{6}(\Omega_\delta))$, we can bound
\begin{eqnarray}%%%%%%%%%%%%
&&||{r}_{\Delta t,k} \bb{U}_{\Delta t,k} ||_{L^2((0,T); L^{\frac{6a}{a+6}}(\Omega_\delta))} \nonumber \\
&&\leq ||{r}_{\Delta t,k}||_{L^\infty(0,T); L^{a}(\Omega_\delta))} || \bb{U}_{\Delta t,k}||_{L^2(0,T; L^{6}(\Omega_\delta))} \leq C(E_0,\delta), \quad\quad\label{qq22sec3ineq5} 
\end{eqnarray}%%----------------------------%%
and
\begin{eqnarray}%%%%%%%%%%%%
&&||{r}_{\Delta t,k} \bb{U}_{\Delta t,k} ||_{L^\infty(0,T; L^{\frac{2a}{a+1}}(\Omega_\delta))} \nonumber \\
&&\leq || {r}_{\Delta t,k} ||_{L^\infty(0,T; L^a(\Omega_\delta))}^{1/2} || {r}_{\Delta t,k} |\bb{U}_{\Delta t,k}|^2 ||_{L^\infty(0,T; L^1(\Omega_\delta))}^{1/2} \leq C(E_0,\delta), \quad\quad \quad \label{qq22sec3ineq1}
\end{eqnarray}%%----------------------------%%
so $(vi)$ follows by $(i)$ and $(v)$.

Now, from $\eqref{qq22thirdlevelsystem}_3$, one can bound\footnote{This bound is certaintly not optimal, but it is sufficient.}
\begin{eqnarray*}%%%%%%%%%%%%%
||P_{fl}^k(\partial_t (J{r} \bb{U}))||_{L^{2}(0,T; H^{-3}(\Omega_\delta))} \leq C(E_0),
\end{eqnarray*}%%----------------------------%%
and since $L^{\frac{6a}{a+6}}(\Omega_\delta)$ is compactly imbedded into $H^{-s}(\Omega_\delta)$ for $a>3/2$ and some $0<s<1$, by $\eqref{qq22sec3ineq5}$ and the Aubin-Lions lemma, one gets
\begin{eqnarray}\label{qq22similarly}%%%%%%%%%%%%%
J_{\Delta t,k}{r}_{\Delta t,k} \bb{U}_{\Delta t,k} \to J{r} \bb{U} ~\text{ in } L^{2}(0,T; H^{-1}(\Omega_\delta)).
\end{eqnarray}%%----------------------------%%
Now, let $\varphi_i \in C_0^\infty(\Omega_\delta)$ for $i\in \mathbb{N}$ be a sequence of functions such that $||\varphi_i - 1||_{L^p(\Omega_\delta)} \to 0$ as $i \to \infty$, for some large $p$. By the weak convergence of $\bb{U}_{\Delta t,k}$ in $L^2(0,T; H^{1}(\Omega_\delta))$, one has the weak convergence of $\varphi_i \bb{U}_{\Delta t,k}$ in $L^2(0,T; H_0^{1}(\Omega_\delta))$, which by $\eqref{qq22similarly}$ implies
\begin{eqnarray}\label{qq22papa1}%%%%%%%%%%%%%
J_{\Delta t,k}{r}_{\Delta t,k} \bb{U}_{\Delta t,k} \otimes \big[\varphi_i\bb{U}_{\Delta t,k} \big] \to J {r} \bb{U} \otimes \big[\varphi_i\bb{U}\big],
\end{eqnarray}%%----------------------------%%
weakly in $L^1(Q_{\delta,T})$. Next, since
\begin{eqnarray}%%%%%%%%%%%%%
&&||{r}_{\Delta t,k} \bb{U}_{\Delta t,k} \otimes \bb{U}_{\Delta t,k}||_{L^2(0,T; L^{\frac{6a}{4a+3}}(\Omega_\delta))}\nonumber \\
 &&\leq ||{r}_{\Delta t,k} \bb{U}_{\Delta t,k}||_{L^\infty(0,T; L^{\frac{2a}{a+1}}(\Omega_\delta))} ||\bb{U}_{\Delta t,k}||_{L^2(0,T; L^{6}(\Omega_\delta))} \leq C(E_0,\delta), \label{qq22sec3ineq8}
\end{eqnarray}%%----------------------------%%
we obtain that
\begin{eqnarray*}%%%%%%%%%%%%%
&&||J_{\Delta t,k}{r}_{\Delta t,k} \bb{U}_{\Delta t,k} \otimes \big[(1-\varphi_i) \bb{U}\big]||_{L^2(0,T; L^{q}(\Omega_\delta))} \\
&&\leq ||{r}_{\Delta t,k} \bb{U}_{\Delta t,k} \otimes \bb{U}_{\Delta t,k}||_{L^2(0,T; L^{\frac{6a}{4a+3}}(\Omega_\delta))} || (1-\varphi_i) ||_{L^p(\Omega_\delta)}\leq C(E_0,\delta)|| (1-\varphi_i) ||_{L^p(\Omega_\delta)}
\end{eqnarray*}%%----------------------------%%
for a large $p$ and $q$ such that $\frac{1}{q}=\frac{1}{p}+ \frac{4a+3}{6a}$, which implies 
\begin{eqnarray*}%%%%%%%%%%%%%
J_{\Delta t,k}{r}_{\Delta t,k} \bb{U}_{\Delta t,k} \otimes (1-\varphi_i) \bb{U} &\rightharpoonup& G_i, \quad \text{weakly in } L^2(0,T; L^{q}(\Omega_\delta)), 
\end{eqnarray*}%%----------------------------%%
as $\Delta t\to 0, k\to +\infty$, and 
\begin{eqnarray*}%%%%%%%%%%%%%
G_i &\rightharpoonup& 0,\quad \text{weakly in } L^2(0,T; L^{q}(\Omega_\delta)),
\end{eqnarray*}%%----------------------------%%
as $i\to+\infty$. Thus, by $\eqref{qq22papa1}$, one obtains
\begin{eqnarray*}%%%%%%%%%%%%%
&&J_{\Delta t,k}{r}_{\Delta t,k} \bb{U}_{\Delta t,k} \otimes \bb{U}_{\Delta t,k} \\
&&= J_{\Delta t,k}{r}_{\Delta t,k} \bb{U}_{\Delta t,k} \otimes \big[\varphi_i\bb{U}_{\Delta t,k}\big] + J_{\Delta t,k}{r}_{\Delta t,k} \bb{U}_{\Delta t,k} \otimes \big[(1-\varphi_i) \bb{U}_{\Delta t,k}\big] \\
&&\rightharpoonup J {r} \bb{U} \otimes \big[\varphi_i\bb{U}\big]+ G_i.
\end{eqnarray*}%%----------------------------%%
weakly in $L^1(Q_{\delta,T})$, so by letting $i\to\infty$, the convergence in $(vii)$ follows. Now, by using the fact that $\gamma>12/7$, similarly as in $\eqref{qq22similarly}$, one has
\begin{eqnarray*}%%%%%%%%%%%%%
J_{\Delta t,k}r_{\Delta t,k} \bb{U}_{\Delta t,k} \to J r \bb{U}, \quad \text{in } L^2(0,T ; H^{-s}(\Omega_\delta)),
\end{eqnarray*}%%----------------------------%%
for some $0<s<1/2$, and since $\bb{w}_{\Delta t,k} \rightharpoonup \bb{w}$ weakly in $L^2(0,T ; H^{(\frac{1}{2})^-}(\Omega))$, the convergence given in $(viii)$ follows in a same way as the convergence in $(vii)$. Thus, the proof is finished.
\end{proof}
Now, one can conclude that the limiting functions $(w,\theta)$ satisfy the heat equation $\eqref{qq22heat1}$ and that $(r,\bb{U},w)$ satisfy the following damped continuity equation in the weak form 
\begin{eqnarray}%%%%%%%%%%%%%
\int_0^T \frac{d}{dt}\int_{\Omega_\delta} J{r} \varphi + \int_0^T \int_{\Omega_\delta}\big[ -J{r}\partial_t \varphi + J({r}_\varepsilon \bb{w} - {r} \bb{U}_\varepsilon)\cdot \nabla^w \varphi\big] =\varepsilon \int_0^T\int_{\Omega_\delta} J \nabla {r} \cdot \nabla \varphi. \nonumber \\
\label{qq22someweakform}
\end{eqnarray}%%----------------------------%%
We will prove that $(r,\bb{U},w)$ satisfy the damped continuity equation in the strong form $\eqref{qq22dampcont1}$:
\begin{lem}\label{qq22lemmanicela}
For $a \geq 9$ we have\footnote{Notice that we were not able to prove such a result in the previous section since the trace of the fluid velocity $v_{\Delta t,k}$ and $\partial_t w_{\Delta t,k}$ were not necessarily equal at this level of approximation, due to the operator splitting.}
\begin{eqnarray}%%%%%%%%%%%%%
\varepsilon||\nabla r||_{L^{3}(0,T; L^{\frac{9}{4}}(\Omega_\delta))} &\leq & C(E_0,\delta)\label{qq22nicela2},\\
\varepsilon^{3/4} ||\nabla {r}||_{L^{\frac{12}{5}}(0,T; L^{\frac{36}{17}}(\Omega_\delta))} &\leq& C(E_0,\delta). \label{qq22nicela}
\end{eqnarray}%%----------------------------%%
Moreover, the limiting functions $(r,\bb{U},w)$ satisfy the continuity equation $\eqref{qq22dampcont1}$ and the following estimate holds
\begin{eqnarray}\label{qq22maxreg}%%%%%%%%%%%%%
|| \partial_t r ||_{L^{\frac{6}{5}}(0,T; L^{\frac{36}{25}}(\Omega_\delta))}+\varepsilon|| \Delta r||_{L^{\frac{6}{5}}(0,T; L^{\frac{36}{25}}(\Omega_\delta))} \leq C(E_0,\delta,\varepsilon).
\end{eqnarray}%%----------------------------%%

\end{lem}

\begin{proof}
The equation $\eqref{qq22someweakform} \pm \varepsilon\int_{Q_{\delta,T}}r\nabla J\cdot \nabla\varphi$ can be written in the following form
\begin{eqnarray}\label{qq22mepawn}%%%%%%%%%%%%%
\int_0^T \frac{d}{dt}\int_{\Omega_\delta} (J{r}) \varphi - \int_0^T \int_{\Omega_\delta} (J{r}) \partial_t \varphi+\varepsilon \int_0^T\int_{\Omega_\delta} \nabla (J {r})\cdot \nabla \varphi= - \int_0^T \int_{\Omega_\delta} \bb{f} \cdot \nabla \varphi, \quad\quad
\end{eqnarray}%%----------------------------%%
for all $\varphi \in C^\infty([0,T]\times \overline{\Omega_\delta})$, where
\begin{eqnarray*}%%%%%%%%%%%%%
\bb{f}:=\Big[J{r} \big[(\nabla A_{w}^{-1})\circ A_w\big]^T( \bb{w} - \bb{U}) + \varepsilon \nabla J r\Big].
\end{eqnarray*}%%----------------------------%%
Now, since the right-hand side of $\eqref{qq22mepawn}$ is in the divergence form, one has that the unique solution to this equation $(J{r})$ satisfies (see \cite[Lemma 7.3]{novotnystraskraba})
\begin{eqnarray}\label{qq22great}%%%%%%%%%%%%%
&&\varepsilon^{1- \frac{1}{p}}|| J{r} ||_{L^\infty(0,T ; L^{q}(\Omega_\delta))}+\varepsilon || J r ||_{L^p(0,T ; W^{1,q}(\Omega_\delta))}\nonumber \\ &&\leq C(p,q,\delta)\Big[\varepsilon^{1- \frac{1}{p}}||r(0)||_{L^{q}(\Omega_\delta)} + ||\bb{f} ||_{L^p(0,T ; L^{q}(\Omega_\delta))}\Big],\quad\quad\quad\quad \label{qq22great}
\end{eqnarray}%%----------------------------%%
for some $1<p,q<\infty$ such that the right-hand side is finite. By interpolation
\begin{eqnarray*}%%%%%%%%%%%%%
||r\bb{U}||_{L^3(0,T ; L^{\frac{18a}{5a+15}}(\Omega_\delta))} \leq ||r\bb{U}||_{L^2(0,T ; L^{\frac{6a}{a+6}}(\Omega_\delta))}^{\frac{2}{3}} ||r\bb{U}||_{L^\infty(0,T ; L^{\frac{2a}{a+1}}(\Omega_\delta))}^{\frac{1}{3}}\leq C(E_0,\delta),
\end{eqnarray*}%%----------------------------%%
and
\begin{eqnarray*}%%%%%%%%%%%%%
||r \bb{w}||_{L^3(0,T ; L^{\frac{9}{4}}(\Omega_\delta))}& \leq& ||r ||_{L^\infty(0,T; L^a(\Omega_\delta))} ||\bb{w}||_{L^3(0,T ; L^{3}(\Omega_\delta))}\\
&\leq &||r ||_{L^\infty(0,T; L^a(\Omega_\delta))} ||\bb{w}||_{L^2(0,T ; L^{4}(\Omega_\delta))}^{\frac{2}{3}}||\bb{w}||_{L^\infty(0,T ; L^{2}(\Omega_\delta))}^{\frac{1}{3}}\leq C(E_0,\delta), \\
\end{eqnarray*}%%----------------------------%%
for $a\geq 9$, so we obtain that the term on the right-hand side of $\eqref{qq22great}$ can be bounded for $p=3$ and $q = \frac{9}{4}$. Now, one easily obtains
\begin{eqnarray*}%%%%%%%%%%%%%
\varepsilon||J \nabla r||_{L^3(0,T; L^{\frac{9}{4}}(\Omega_\delta))} \leq \varepsilon||\nabla (J r)||_{L^3(0,T; L^{\frac{9}{4}}(\Omega_\delta))} + \varepsilon||\nabla J r||_{L^3(0,T; L^{\frac{9}{4}}(\Omega_\delta))} \leq C(E_0,\delta),
\end{eqnarray*}%%----------------------------%%
so $\eqref{qq22nicela2}$ follows by the uniform lower bound of the Jacobian $J$ given in Lemma $\ref{qq22111}(iii)$. Consequently, by Lemma $\ref{qq22111}(4c)$ and interpolation
\begin{eqnarray*}%%%%%%%%%%%%%
\varepsilon^{3/4} ||\nabla {r}||_{L^{\frac{12}{5}}(0,T; L^{\frac{36}{17}}(\Omega_\delta))} \leq (\sqrt{\varepsilon}||\nabla {r}||_{L^{2}(0,T; L^{2}(\Omega_\delta))})^{1/2} ( \varepsilon ||\nabla {r}||_{L^{3}(0,T; L^{\frac{9}{4}}(\Omega_\delta))})^{1/2} \leq C(E_0,\delta), \quad \quad
\end{eqnarray*}%%----------------------------%%
so $\eqref{qq22nicela}$ follows. Now, by Lemma $\ref{qq22111}$ and $\eqref{qq22nicela2}$, one can see that
\begin{eqnarray*}%%%%%%%%%%%%%
\nabla\cdot \bb{f} \in L^{\frac{6}{5}}(0,T; L^{\frac{36}{25}}(\Omega_\delta))
\end{eqnarray*}%%----------------------------%%
so the equation $\eqref{qq22mepawn}$, can be solved in the strong sense. Since the weak solution $(Jr)$ of $\eqref{qq22mepawn}$ is unique, one can conclude that $(Jr)$ coincides with this strong solution and thus satisfies $\eqref{qq22mepawn}$ in the strong sense. This gives us that $(r,\bb{U},w)$ satisfy $\eqref{qq22dampcont1}$, while the inequality $\eqref{qq22maxreg}$ follows by the maximal regularity estimates \cite[Chapter III]{amann} so the proof is finished.
\end{proof}
Now, by Lemma $\ref{qq22weakconvthirdlevel}$, one can obtain the convergence of all the necessary terms in $\eqref{qq22thirdlevelsystem}$ except the term $\varepsilon\int_{Q_{\delta,T}} J_{\Delta t,k} \nabla r_{\Delta t,k} \cdot ( \bb{q}\cdot \nabla \bb{U}_{\Delta t,k}+ \bb{U}_{\Delta t,k}\cdot\nabla \bb{q})$. This convergence is proved in the following lemma:
\begin{lem}
The following convergences hold as $\Delta t\to 0$ and $k\to +\infty$:
\begin{enumerate}
\item[(i)] $\nabla {r}_{\Delta t,k} \to \nabla {r}$, in $L^2(Q_{\delta,T})$;
\item[(ii)] $\nabla {r}_{\Delta t,k} \cdot \nabla \bb{U}_{\Delta t,k} \rightharpoonup \nabla {r} \cdot \nabla \bb{U}$ and $\nabla {r}_{\Delta t,k} \cdot \bb{U}_{\Delta t,k} \rightharpoonup \nabla {r} \cdot \bb{U}$, weakly in $L^{1}(Q_{\delta,T})$.
\end{enumerate}
\end{lem}
\begin{proof}
To prove $(i)$, we multiply the continuity equation by $J_{\Delta t,k}{r}_{\Delta t,k}$ and integrate on $Q_{\delta,T}$ to obtain 
\begin{eqnarray*}%%%%%%%%%%%%%
&&\frac{1}{2}||(\sqrt{J_{\Delta t,k}}{r}_{\Delta t,k})(t) ||_{L^2(\Omega_\delta)}^2 + \int_0^t\int_{\Gamma}{r}_{\Delta t,k}^2(v_{\Delta t,k}- \partial_t w_{\Delta t,k})+ \varepsilon \int_0^t \int_{\Omega_\delta} J_{\Delta t,k}|\nabla {r}_{\Delta t,k}|^2 \\
&&= \frac{1}{2}||(\sqrt{J_{\Delta t,k}}{r}_{\Delta t,k})(0) ||_{L^2(\Omega_\delta)}^2- \frac{1}{2} \int_0^t \int_{\Omega_\delta} J_{\Delta t,k} {r}_{\Delta t,k}^2 \nabla^{w_{\Delta t,k}} \cdot \bb{U}_{\Delta t,k},
\end{eqnarray*}%%----------------------------%%
and then again integrate over $(0,T)$ to obtain
\begin{align}%%%%%%%%%%%%%
&\frac{1}{2}||\sqrt{J_{\Delta t,k}}{r}_{\Delta t,k}||_{L^2(Q_{T,\delta})}^2 + \int_0^T (T-t)\Big[ \int_{\Gamma}( {r}_{\Delta t,k}^2(v_{\Delta t,k} - \partial_t w_{\Delta t,k}))(t)\Big]dt\nonumber \\
&+ \varepsilon \int_0^T(T-t)\Big[\int_{\Omega_\delta} (J_{\Delta t,k} |\nabla {r}_{\Delta t,k}|^2)(t)\Big] dt \nonumber \\
&= \frac{T}{2}||(\sqrt{J_{\Delta t,k}}{r}_{\Delta t,k})(0) ||_{L^2(\Omega_\delta)}^2- \frac{1}{2} \int_0^T(T-t) \Big[\int_{\Omega_\delta}( J_{\Delta t,k}{r}_{\Delta t,k}^2 \nabla^{w_{\Delta t,k}} \cdot \bb{U}_{\Delta t,k})(t)\Big] dt. \label{qq22idk1}
\end{align}%%----------------------------%%
On the other hand, since $\eqref{qq22dampcont1}$ holds for the limiting functions ${r}$ and $\bb{U}$, we multiply $\eqref{qq22dampcont1}$ by $r$, and again integrate over $(0,T)\times Q_{\delta,T}$ to obtain
\begin{eqnarray}%%%%%%%%%%%%%
&&\frac{1}{2}||\sqrt{J}{r}||_{L^2(Q_{T,\delta})}^2 + \varepsilon \int_0^T(T-t)\Big[\int_{\Omega_\delta} (J |\nabla {r}|^2)(t)\Big] dt \nonumber \\
&&= \frac{T}{2}||(\sqrt{J}{r})(0) ||_{L^2(\Omega_\delta)}^2- \frac{1}{2} \int_0^T(T-t) \Big[\int_{\Omega_\delta}(J {r}^2 \nabla^w \cdot \bb{U})(t)\Big] dt.  \label{qq22idk2}
\end{eqnarray}%%----------------------------%%
Now, by passing the limit in $\Delta t,k$ in $\eqref{qq22idk1}$ and comparing it with $\eqref{qq22idk2}$, one has
\begin{eqnarray*}%%%%%%%%%%%%%
\lim_{\Delta t,k}~ \varepsilon \int_0^T(T-t)\Big[\int_{\Omega_\delta} (J_{\Delta t,k} |\nabla {r}_{\Delta t,k}|^2)(t)\Big] dt =\varepsilon \int_0^T(T-t)\Big[\int_{\Omega_\delta} J |\nabla {r}|^2)(t)\Big] dt.
\end{eqnarray*}%%----------------------------%%
or equivalently
\begin{eqnarray*}%%%%%%%%%%%%%
|| \sqrt{J_{\Delta t,k}(T-t)}\nabla {r}_{\Delta t,k}||_{L^2(Q_{T,\delta})} \to || \sqrt{J(T-t)}\nabla {r}||_{L^2(Q_{T,\delta})}, 
\end{eqnarray*}%%----------------------------%%
which by the weak convergence of $\nabla {r}_{\Delta t,k}$ in $L^2(Q_{\delta,T})$ implies $(i)$, while the convergences given in $(ii)$ follow by Lemma $\ref{qq22weakconvthirdlevel}(i)$ and $(i)$.
\end{proof}

\subsection{Renormalized continuity inequality}
Here, it is more suitable to work on the physical domain $\Omega_\delta^w$. Recall that we denote the density and the velocity on $\Omega_\delta^w$ by $\rho= {r}\circ A_w^{-1}$ and $\bb{u} = \bb{U} \circ A_w^{-1}$, respectively, where $r$ and $\bb{U}$ are the limiting functions from the previous section that satisfy the damped continuity equation $\eqref{qq22dampcont1}$. In the physical domain $\Omega_\delta^w$, the continuity equation can be written as
\begin{align}\label{qq22movingdomainsystem}%%%%%%%%%%%%%%%
\partial_t \rho+ \nabla \cdot (\rho \bb{u}) = \ddfrac{1}{J}\varepsilon \nabla^{w^{-1}} \cdot(\nabla^{w^{-1}} \rho J), 
\end{align}%%%%%%%%
where $\nabla^{w^{-1}}$ is the push forward of the gradient by $A_w^{-1}$. Notice that $J$ doesn't depend on the vertical coordinate (which is the only one that gets affected by the domain transformation) so we keep the same notation on the physical domain.

In order to introduce our renormalized continuity equation, we want to extend the velocity $\bb{u}$ to be defined in $\mathbb{R}^3$ in such a way that the extension preserves the Sobolev regularity of $\bb{u}$. Assuming that displacement $w$ is given, we first introduce the scaled-symmetric domain
\begin{eqnarray*}%%%%%%%%%%%%%
\Omega_\delta^E:=\{(X,z): (X, -\frac{z}{L})\in\Omega_\delta \}.
\end{eqnarray*}%%----------------------------%%
where\footnote{The choice of $L$ ensures that the elastic structure $\Gamma^w$ is uniformly distant from the plane $z = L$. This way, a function can be properly extended from $\Omega_\delta^w$ to 
$\Omega_\delta \cup \Omega_\delta^E$.} $L \geq ||w||_{C([0,T]\times \Omega)}+1$ can be chosen so it only depends on the initial energy $E(0)$. The extension will be first defined in the fixed reference domain coordinates as a scaled-symmetric mapping from $\Omega_\delta$ to $\Omega_\delta^E$ and then transformed back to the physical domain coordinates by means of an extended domain transformation, or precisely:
\begin{mydef}
Let function $f\in W^{1,p}(\Omega_\delta^{w}(t))$ for $p\in(1,\infty)$ be such that its trace $\gamma_{|\Sigma} f=0$ and let $\hat{f}: = f \circ A_w$. We define $\mathcal{E}^w: W^{1,p}(\Omega_\delta^{w}(t)) \to W_0^{1,p}(\mathbb{R}^3)$ as
\begin{eqnarray*}%%%%%%%%%%%%%
\mathcal{E}^w[f] := \begin{cases}
\hat{f}^E \circ A_w^E, \text{ in } \Omega_\delta\cup\Omega_\delta^E,\\
0, \text{ elsewhere in } \mathbb{R}^3, 
\end{cases}
\end{eqnarray*}%%----------------------------%%
where
\begin{eqnarray*}%%%%%%%%%%%%%
\hat{f}^E : = \begin{cases}
\hat{f}(X,z), \text{ in } \Omega_\delta, \\
\hat{f}(X,-\frac{z}{L}), \text{ in } \Omega_\delta^E,
\end{cases}
\end{eqnarray*}%%----------------------------%%
and 
\begin{eqnarray*}%%%%%%%%%%%%%
A_w^E(t,X,z) := \begin{cases}
A_w(t,X,z), \text{ in } Q_T, \\
(X,w\frac{L-z}{L}+z), \text{ for } (t,X,z) \in (0,T)\times\Omega_\delta^E, \\
(X,z), \text{ elsewhere in } \mathbb{R}^3,
\end{cases} 
\end{eqnarray*}%%----------------------------%%

\end{mydef}

For a function $f \in L_{loc}^1(\mathbb{R}^3)$ we define the convolution (with respect to the physical domain coordinates) in the following way:
\begin{eqnarray*}%%%%%%%%%%%%%
f_\kappa(y):= \int_{\mathbb{R}^3} f(x) \omega_\kappa(y-x) dx,
\end{eqnarray*}%%----------------------------%%
where $\omega$ is a non-negative smooth function on $\mathbb{R}^3$ such that $\text{supp}~\omega =B(0,1):= \{ x\in \mathbb{R}^3: |x| \leq 1\}$, $\int_{B(0,1)}\omega =1 $ and $\omega_\kappa(x) = \kappa^3 \omega(\frac{x}{\kappa})$. By the standard theory, without proof we state:
\begin{lem}\label{qq22lemmaconv}
The following hold:
\begin{enumerate}
\item[(i)] For $f \in L_{loc}^1(\mathbb{R}^3)$, then $f_\kappa \in C^{\infty}(\mathbb{R}^3)$;
\item[(ii)] If $f \in W^{s,p}(\mathbb{R}^3)$, $1\leq p \leq \infty$ and $s \geq 0$, then $f_\kappa \to f$ strongly in $W^{s,p}(\mathbb{R}^3)$ as $\kappa\to 0$;
\item[(iii)] If $f \in W^{1,p}(\mathbb{R}^3)$, then $f_\kappa \circ A_w^E=:\hat{f}_\kappa \in W^{1,p}(\mathbb{R}^3)$ and $\hat{f}_\kappa \to \hat{f}$ in $W^{1,p}(\mathbb{R}^3)$ as $\kappa\to 0$.
\end{enumerate}
\end{lem}

\noindent
Now, we are ready to prove:
\begin{thm}\label{qq22weakomegaRCE}
Any weak solution $(\rho,\bb{u},w,\theta)$ in the sense of Definition $\ref{qq22solweakartvis}$ satisfies the following renormalized continuity inequality
\begin{eqnarray}%%%%%%%%%%%%%
&&\int_0^T \frac{d}{dt} \int_{\mathbb{R}^3} b(\rho) \varphi-\int_0^T \int_{\mathbb{R}^3} \big( b(\rho) \partial_t \varphi + b(\rho)\mathcal{E}^w[\bb{u}]\cdot \nabla \varphi \big) \nonumber \\
&&\geq-\int_0^T \int_{\mathbb{R}^3} \big( \rho b'(\rho) -b(\rho) \big)(\nabla \cdot \mathcal{E}^w[\bb{U}])\varphi \nonumber \\
&&\quad + \int_0^T \int_{\mathbb{R}^3}\varepsilon\chi_{\Omega^{w}(t)} \big(- \nabla^{w^{-1}} b(\rho)\cdot \nabla^{w^{-1}} \varphi +\frac{1}{J}\nabla^{w^{-1}} J\cdot \nabla^{w^{-1}} b(\rho) \varphi \big), \label{qq22RCEepsmov}
\end{eqnarray}%%----------------------------%%
where $\rho$ is extended by $0$ to $\mathbb{R}^3$, for any non-negative $\varphi \in C^\infty([0,T] \times \mathbb{R}^3)$ and any convex $b \in C^2(\mathbb{R}^+ ; \mathbb{R}^+)$ such that $b(0) = 0$, $|b'(x)|\leq cx$ for large $x$ and $b''(x)\leq C$, for some positive constants $c,C$.
\end{thm}
\begin{proof}
The term on the right-hand side of the equation $\eqref{qq22movingdomainsystem}$ can be written as
\begin{eqnarray*}%%%%%%%%%%%%%
\varepsilon \big( \Delta^{w^{-1}}\rho+\frac{1}{J} \nabla^{w^{-1}} \rho \cdot \nabla^{w^{-1}} J \big),
\end{eqnarray*}%%----------------------------%%
where the transformed Laplacian can be expressed as
\begin{eqnarray*}%%%%%%%%%%%%%
\Delta^{w^{-1}} \rho := (\Delta r) \circ (A_{w})^{-1} = \sum_{i,j,k=1}^3 \partial_{x_k}A_i^{-1} \partial_{x_k} A_j^{-1} \partial_{x_i x_j}^2 \rho + \sum_{i,j=1}^3 \partial_{x_i x_i}^2 A_j^{-1} \partial_{x_j}\rho,
\end{eqnarray*}%%----------------------------%%
with $A_i^{-1}:=A_{w}^{-1} \bb{e}_i$. 
To apply the convolution to the equation $\eqref{qq22movingdomainsystem}$, we want it's left-hand side to be defined on $\mathbb{R}^3$, so we extend $\rho$ by $0$ and $\bb{u}$ by $\mathcal{E}^w[\bb{u}]$, and then apply the convolution in the following way:
\begin{eqnarray}\label{qq22somestar}%%%%%%%%%%%%%
\partial_t \rho_\kappa + \nabla \cdot (\rho_\kappa \mathcal{E}^w[\bb{U}])-\bb{r}_\kappa= \big(\chi_{\Omega^w}\ddfrac{1}{J}\varepsilon \nabla^{w^{-1}} \cdot(\nabla^{w^{-1}} \rho J) \big)_\kappa,
\end{eqnarray}%%----------------------------%%
where $\bb{r}_\kappa = \nabla \cdot(\rho_\kappa \mathcal{E}^w[\bb{u}] ) - \nabla \cdot(\rho\mathcal{E}^w[\bb{u}] )_\kappa$. By the Friedrichs commutator lemma (see \cite[Lemma 3.1]{novotnystraskraba}), we have
\begin{eqnarray*}%%%%%%%%%%%%%
||\bb{r}_\kappa ||_{L^q(\mathbb{R}^3)} \leq || \mathcal{E}^w[\bb{u}] ||_{W^{1,p}(\mathbb{R}^3)} ||\rho||_{L^r(\mathbb{R}^3)}, \quad 1/q = 1/p+1/r, ~~ p \leq 2,~~ r\leq a, 
\end{eqnarray*}%%----------------------------%%
for a.e. $t \in [0,T]$, and
\begin{eqnarray*}%%%%%%%%%%%%%
\bb{r}_\kappa \to 0, \quad \text{in } L^2(0,T; L^q(\mathbb{R}^3)), \quad \text{as } \kappa\to 0.
\end{eqnarray*}%%----------------------------%%
Now, we multiply the equation $\eqref{qq22somestar}$ by $b'(\rho_\kappa)$, to obtain
\begin{eqnarray}%%%%%%%%%%%%%
&&\partial_t b(\rho_\kappa)+\nabla \cdot(b(\rho_\kappa)\mathcal{E}^w[\bb{u}] )+ (\rho_\kappa b' (\rho_\kappa) - b(\rho_\kappa)) (\nabla \cdot \mathcal{E}^w[\bb{u}] )-\bb{r}_\kappa b'(\rho_\kappa) \nonumber\\
&&= \big(\chi_{\Omega^w}\ddfrac{1}{J}\varepsilon \nabla^{w^{-1}} \cdot(\nabla^{w^{-1}} \rho J) \big)_\kappa b'(\rho_\kappa)\nonumber\\[2mm]
&&= \underbrace{\Big[ \big(\chi_{\Omega^w} \frac{1}{J}\varepsilon \nabla^{w^{-1}} \cdot( \nabla^{w^{-1}} \rho J) \big)_\kappa- \chi_{\Omega^w} \frac{1}{J}\varepsilon \big( \nabla^{w^{-1}} \cdot( \nabla^{w^{-1}} \rho_\kappa J) \big)\Big] b'(\rho_\kappa)}_{:=\bb{r}_\kappa'} \nonumber \\
&&\quad +\chi_{\Omega^w}\ddfrac{1}{J}\varepsilon \big(\nabla^{w^{-1}} \cdot(\chi_{\Omega^w}\nabla^{w^{-1}} \rho_\kappa J) \big)b'(\rho_\kappa). \quad\quad \quad\quad \label{qq22almostocnteq}
\end{eqnarray}%%----------------------------%%
Now, since $ J \in L^\infty(0,T; W^{1,\infty}(\Omega_\delta))$, $1/J \in L^\infty(Q_{\delta,T})$ and $\Delta r \in L^{\frac{6}{5}}(0,T; L^{\frac{36}{25}}(\Omega_\delta))$ by $\eqref{qq22maxreg}$ one obtains that $\bb{r}_\kappa' \in L^{\frac{6}{5}}(0,T; L^{\frac{36}{25}}(\Omega_\delta))$, so by Lemma $\ref{qq22lemmaconv}(iii)$ we obtain that $\bb{r}_\kappa' \to 0$ in $L^{\frac{6}{5}}(0,T; L^{\frac{36}{25}}(\Omega_\delta))$. Next, to deal with the last term in $\eqref{qq22almostocnteq}$, we express it on the fixed domain $\Omega_\delta$
\begin{eqnarray*}%%%%%%%%%%%%%
\chi_\Omega \frac{1}{J} \varepsilon \big( \nabla \cdot \big[ \nabla {r}_\kappa J b'({r}_\kappa) \big]\big) - \varepsilon\chi_\Omega \big(|\nabla {r}_\kappa|^2 b''({r}_\kappa))\big),
\end{eqnarray*}%%----------------------------%%
where by a slight abuse of notation we denoted $ {r}_\kappa= \rho_\kappa\circ A_{w}^E$. By using the growth conditions for function $b$, $|b(x)|\leq cx^2$ for large $x$, so by Lemma $\ref{qq22lemmaconv}$ and the Vitali convergence theorem (see \cite[Theorem II.6.15]{vitali}), one has
\begin{eqnarray*}%%%%%%%%%%%%%
b(r_\kappa) &\to b(r),\quad b(\rho_\kappa) &\to b(\rho), \text{ in } L^{\infty^-}(0,T;L^{(\frac{a}{2})^-}(\mathbb{R}^3)), \\
b(r_\kappa) &\rightharpoonup b(r), \quad b(\rho_\kappa) &\rightharpoonup b(\rho), \text{ weakly in } L^{\infty}(0,T; L^{\frac{a}{2}}(\mathbb{R}^3)), \\[3mm]
b'(r_\kappa) &\to b'(r),\quad b'(\rho_\kappa) &\to b'(\rho), \text{ in } L^{\infty^-}(0,T;L^{a^-}(\mathbb{R}^3)), \\
b'(r_\kappa) &\rightharpoonup b'(r), \quad b'(\rho_\kappa) &\rightharpoonup b'(\rho), \text{ weakly in } L^{\infty}(0,T; L^{a}(\mathbb{R}^3)), \\[3mm]
b''(r_\kappa) &\to b(r),\quad b''(\rho_\kappa) &\to b''(\rho), \text{ in } L^{\infty^-}((0,T) \times \mathbb{R}^3), \\
b''(r_\kappa) &\rightharpoonup b''(r),\quad b''(\rho_\kappa) &\rightharpoonup b''(\rho), \text{ weakly}^* \text{ in } L^{\infty}((0,T) \times \mathbb{R}^3),
\end{eqnarray*}%%----------------------------%%
which in particular gives us that
\begin{eqnarray*}%%%%%%%%%%%%%
\chi_\Omega \big(|\nabla {r}_\kappa|^2 b''({r}_\kappa))\big) &\rightharpoonup& \chi_\Omega |\nabla r|^2 b''(r) \geq 0, \quad\quad \text{in } L^1((0,T)\times \mathbb{R}^3)
\end{eqnarray*}%%----------------------------%%
by the convexity of $b$, and
\begin{eqnarray*}%%%%%%%%%%%%%
\chi_\Omega \frac{1}{J} \big( \nabla \cdot \big[ \nabla {r}_\kappa J b'({r}_\kappa) \big]\big) &\rightharpoonup& \chi_\Omega \frac{1}{J} \big( \nabla \cdot \big[ (\underbrace{\nabla r J b'(r)}_{=J\nabla b(r) } \big]\big), \quad\quad \text{in } L^1((0,T)\times \mathbb{R}^3)
\end{eqnarray*}%%----------------------------%%
by $\eqref{qq22maxreg}$ and the uniform estimates. We multiply $\eqref{qq22almostocnteq}$ by a non-negative function $\varphi \in C^\infty([0,T] \times \mathbb{R}^3)$, integrate over $[0,T] \times \mathbb{R}^3$, and by partial integration, the convexity of $b$ and the above convergences, we can pass to the limit $\kappa \to 0$ to obtain $\eqref{qq22RCEepsmov}$, so the proof is finished.
\end{proof}

\section{The vanishing artificial viscosity limit}\label{qq22section5}
Denote the solution obtained in Theorem $\ref{qq22artdisth}$ by $( r_\varepsilon, \bb{U}_\varepsilon,w_\varepsilon,\theta_\varepsilon)$. In this section we aim to pass the limit in $\varepsilon \to 0$. Introduce the function spaces
\begin{eqnarray*}%%%%%%%%%%%%%
\mathcal{W}_{FS,\delta} (0,T): = &&\big\{ (\bb{U},w): w \in L^\infty(0,T; H_0^2(\Gamma)) \cap W^{1,\infty}(0,T; L^2(\Gamma)), \\
&&\bb{U} \in L^2(0,T; H^{1}(\Omega_\delta)), \gamma_{|\Gamma\times \{0\}}\bb{U}(t, X,0) = \partial_t w(t,X) \bb{e}_3, ~\bb{U} = 0 \text{ on } \Omega_\delta \big\},
\end{eqnarray*}%%----------------------------%%
and the following weak solution, suitable for the limiting functions in this limit passage:
\begin{mydef}\label{qq22solweakartpr}
We say that $r\in C_w(0,T; L^a(\Omega_\delta)) ,~(\bb{U},w) \in \mathcal{W}_{FS,\delta}(0,T),\theta \in \mathcal{W}_H(0,T)$ are the weak solutions to the fluid-structure interaction problem with artificial pressure if
\begin{enumerate}

\item The following heat equation
\begin{align}
\int_{\Gamma_T}\theta \partial_t \widetilde{\psi} - \int_{\Gamma_T}\nabla \theta \cdot\nabla \widetilde{\psi}+\int_{\Gamma_T}\nabla w \cdot\nabla \partial_t \widetilde{\psi} =\int_0^T \frac{d}{dt}\int_{\Gamma}\theta\widetilde{\psi} +\int_0^T \frac{d}{dt} \int_{\Gamma}\nabla w \cdot \nabla \widetilde{\psi},\nonumber \\ \label{qq22heat3}
\end{align}
holds for all $\widetilde{\psi} \in C_0^\infty(\Gamma_T)$.
\item The following continuity equation holds
\begin{eqnarray}\label{qq22cont3}%%%%%%%%%%%%%
\int_{Q_T} J\rho \partial_t \varphi+ \int_{Q_T}J(\rho \bb{U} - \rho \bb{w}) \cdot\nabla^w \varphi =\int_0^T \frac{d}{dt}\int_{\Omega_\delta} J\rho \varphi ,
\end{eqnarray}%%----------------------------%%
for all $\varphi \in C^\infty([0,T] \times \overline{\Omega_\delta})$.
\item The following coupled momentum equation holds
\begin{eqnarray}%%%%%%%%%%%%%%%
&&\int_{Q_{T}}J{r} \bb{U}\cdot \partial_t \bb{q} +\int_{Q_{T}}J\big[({r} \bb{U}-{r} \bb{w}) \cdot \nabla^{w}\bb{q}\big]\cdot \mathbf{u}-\mu\int_{Q_{T}} J\nabla^{w} \mathbf{u}:\nabla^{w}\bb{q}\nonumber\\
&&-(\mu +\lambda)\int_{Q_{T}} J (\nabla^{w}\cdot \mathbf{u})(\nabla^{w}\cdot \bb{q})+ \int_{Q_{T}}(J ({r}^\gamma + \delta {r}^a) \nabla^{w}\cdot \bb{q}) 
+\bint_{\Gamma_T} \partial_t w \partial_t \psi \nonumber \\
&&-\bint_{\Gamma_T}\Delta w \Delta \psi
-\bint_{\Gamma_T}\mathcal{F}(w) \psi +\bint_{\Gamma_T} \nabla \theta \cdot \nabla \psi -\delta \bint_{\Gamma_T} \nabla^3 w : \nabla^3 \psi\nonumber\\
&&=\int_0^T \frac{d}{dt} \int_{\Omega_\delta}J {r} \bb{U}\cdot\bb{q} +\int_0^T \frac{d}{dt} \int_{\Gamma} \partial_t w \psi, \label{qq22momentumeps}
\end{eqnarray} 
for all $\bb{q} \in C_0^\infty(Q_{\varepsilon,T,\Gamma})$ and $\psi\in C_0^\infty( \Gamma_T )$, such that $\bb{q}_{|\Gamma\times \{0\}} = \psi\bb{e}_3$.\\
\end{enumerate}

\end{mydef}

In this section, we will work both on the fixed and physical domain coordinates. The fluid density and velocity on the physical domain $\Omega_\delta^{w_\varepsilon}$ will be denoted by $\rho_\varepsilon = {r}_\varepsilon \circ A_{w_\varepsilon}^{-1}$ and $\bb{u}_\varepsilon = \bb{U}_\varepsilon \circ A_{w_\varepsilon}^{-1}$, and similarly for the limiting functions $\rho = {r} \circ A_{w}^{-1}$ and $\bb{u} = \bb{U} \circ A_{w}^{-1}$. The continuity equation and coupled momentum equation on the physical domain $\Omega_\delta^{w_\varepsilon}$ corresponding to $\eqref{qq22dampcont1}$ and $\eqref{qq22summedup1}$, respectively, read
\begin{align} \label{qq22movingdomainsystem1234}%%%%%%%%%%%%%%%
\begin{cases} 
&\partial_t \rho_\varepsilon+ \nabla \cdot (\rho_\varepsilon \bb{u}_\varepsilon) = \ddfrac{1}{J_\varepsilon}\varepsilon \nabla^{w_\varepsilon^{-1}} \cdot(\nabla^{w_\varepsilon^{-1}} \rho_\varepsilon J_\varepsilon), \quad \text{a.e. in } Q_{\delta,T} \\[4.5mm]
&\bint_{Q_{\delta,T}^{w_\varepsilon}}\rho_\varepsilon \bb{u}_\varepsilon \cdot \partial_t \bb{q} +\bint_{Q_{\delta,T}^{w_\varepsilon}}\big(\rho_\varepsilon \bb{u}_\varepsilon \otimes \bb{u}_\varepsilon \big) :\nabla \bb{q}-\mu\bint_{Q_{\delta,T}^{w_\varepsilon}} \nabla\bb{u}_\varepsilon:\nabla\bb{q}  \\[3.5mm]
&\hspace{.5in}-(\mu +\lambda)\bint_{Q_{\delta,T}^{w_\varepsilon}}(\nabla\cdot \bb{u}_\varepsilon)(\nabla\cdot \bb{q} )+\bint_{Q_{\delta,T}^{w_\varepsilon}} (\rho_\varepsilon^{\gamma} + \delta \rho_\varepsilon^{a}) (\nabla \cdot \bb{q})\\[3.5mm]
&\hspace{.5in}-\varepsilon \bint_{Q_{\delta,T}^{w_\varepsilon}} \nabla^{w_\varepsilon} \rho_\varepsilon \cdot ( \bb{q}\cdot \nabla^{w_\varepsilon^{-1}} \bb{u}_\varepsilon+ \bb{u}_\varepsilon\cdot \nabla^{w_\varepsilon^{-1}} \bb{q}) +\bint_{\Gamma_T} \partial_t w_\varepsilon \partial_t \psi  \\[3.5mm]
&\hspace{.5in}-\bint_{\Gamma_T}\Delta w_\varepsilon \Delta \psi - \bint_{\Gamma_T}\mathcal{F}(w_\varepsilon) \psi+ \bint_{\Gamma_T}\nabla \theta\cdot \nabla \psi-\delta\bint_{\Gamma_T} \nabla^3 w_\varepsilon : \nabla^3 \psi\\[3.5mm]
&\hspace{.5in} =\bint_0^T \frac{d}{dt} \bint_{\Omega_\delta^{w_\varepsilon}(t)}\rho_\varepsilon \bb{u}_\varepsilon \cdot\bb{q} + \bint_0^T \frac{d}{dt} \bint_\Gamma \partial_t w \psi ,
\end{cases}
\end{align}%%%%%%%%
for all $\bb{q} \in C_0^\infty(Q_{\delta,T,\Gamma}^w)$ and $\psi \in C_0^\infty(\Gamma_T^w)$ such that $\bb{q}_{\Gamma^w} = \psi \bb{e}_3$. Next, the continuity equation and the coupled momentum equation with artificial pressure corresponding on the physical domain $\Omega_\delta^{w}$ corresponding to $\eqref{qq22cont3}$ and $\eqref{qq22momentumeps}$, respectively, read
\begin{align} \label{qq22movingdomainsystem123}%%%%%%%%%%%%%%%
\begin{cases} 
&\bint_{Q_{\delta,T}^{w}}\rho\partial_t\varphi+\bint_{Q_{\delta,T}^{w}} \rho \bb{u}\cdot \nabla\varphi=\bint_{0}^T \frac{d}{dt} \bint_{\Omega_\delta^w}\rho \varphi, \\[4mm]
&\bint_{Q_{\delta,T}^{w}}\rho \bb{u} \cdot \partial_t \bb{q} +\bint_{Q_{\delta,T}^{w}}\big(\rho \bb{u} \otimes \bb{u} \big) :\nabla \bb{q}-\mu \bint_{Q_{\delta,T}^{w}} \nabla\bb{u}:\nabla\bb{q}-(\mu +\lambda)\bint_{Q_{\delta,T}^{w}}(\nabla\cdot \bb{u})(\nabla\cdot \bb{q} ) \\[3.5mm]
&\quad+ \bint_{Q_{\delta,T}^{w}} (\rho^{\gamma} + \delta \rho^{a}) (\nabla \cdot \bb{q})
+\bint_{\Gamma_T} \partial_t w \partial_t \psi -\bint_{\Gamma_T}\Delta w \Delta \psi - \bint_{\Gamma_T}\mathcal{F}(w) \psi+ \bint_{\Gamma_T}\nabla \theta\cdot \nabla \psi\\[3.5mm]
&\quad- \delta \bint_{\Gamma_T}\nabla^3 w : \nabla^3 \psi =\bint_0^T \frac{d}{dt} \bint_{\Omega_\delta^w(t)}\rho \bb{u} \cdot\bb{q}+ \bint_0^T \frac{d}{dt} \bint_\Gamma \partial_t w \psi ,
\end{cases}
\end{align}%%%%%%%%
for all $\bb{q} \in C_0^\infty(Q_{T,\delta,\Gamma}^w)$ and $\psi \in C_0^\infty(\Gamma_T^w)$ such that $\bb{q}_{\Gamma^w} = \psi \bb{e}_3$ and $\varphi \in C^\infty([0,T]\times \overline{\Omega_\delta^w(t)})$.
\begin{thm}\label{qq22artprth}
There exists a weak solution $({r},\bb{U},w,\theta)$ in the sense of Definition $\ref{qq22solweakartpr}$ that satisfies the following energy inequality for all $t \in [0,T]$
\begin{eqnarray*}%%%%%%%%%%%%%
E_\delta^w(t) + D^w(t) \leq C(E_0)
\end{eqnarray*}%%----------------------------%%
where $E_{\delta}^w$ is defined in Theorem $\ref{qq22artdisth}$ and $D^w$ is defined in $\eqref{qq22energies2}$. Moreover, the corresponding functions $(\rho,\bb{u})$ on physical domain $\Omega_\delta^{w}$ satisfy the system $\eqref{qq22movingdomainsystem123}$.
\end{thm}

\begin{thm}\label{qq22RCEepsconv}
The $(\rho,\bb{u},w)$ obtained in Theorem $\ref{qq22artprth}$ on the physical domain $\Omega_\delta^{w}$ satisfies the following renormalized continuity equation 
\begin{align}%%%%%%%%%%%%%
\int_0^T \frac{d}{dt} \int_{\mathbb{R}^3} b(\rho) \varphi-\int_0^T \int_{\mathbb{R}^3} \big( b(\rho) \partial_t \varphi + b(\rho)\bb{u}\cdot \nabla \varphi \big) = -\int_0^T \int_{\mathbb{R}^3} \big( \rho b'(\rho) -b(\rho) \big)(\nabla\cdot \mathcal{E}^w[\bb{u}])\varphi, \quad\label{qq22RCEeps2}
\end{align}%%----------------------------%%
for any $b \in C^1(\mathbb{R})$ such that $b'(x) = 0$, for all $x\geq M_b$, where $M_b$ is a constant.
\end{thm}

The proof of Theorem $\ref{qq22artprth}$ will be carried out through the remainder of this section, partially on the fixed reference domain $\Omega_\delta$ and partially on the physical domain $\Omega_\delta^{w_\varepsilon}$. The former one is less involved and can be obtained in a rather straightforward fashion. The later one deals with the convergence of the pressure and is the most involved part of the theory of the weak solutions for the compressible fluids. We wish to emphasize that the analysis on the fixed domain, even though more suitable for solving the approximate problems and obtaining certain convergences, proves to be rather ineffective since some quantities, like divergence, lose its meaning. Since the key part in proving the convergence is solving the equation $\nabla \cdot \bb{q} = \rho_\varepsilon$, this choice is more suitable. We follow the approach from \cite{compressible}, which combines the approach from \cite{feireisl2} that localizes the standard approach, and \cite{kukucka} in which it is proved that the mass of the pressure doesn't concentrate near the boundary. This is an alternative to the standard approach with the Bogovskii operator, which fails in this framework because the domain that we work on is not Lipschitz in general.
\subsection{Convergence on the fixed reference domain $\Omega_\delta$}

\begin{lem}\label{qq22weakconveps}
The following convergences hold for solutions $(r_\varepsilon, \bb{U}_\varepsilon,w_\varepsilon, \theta_\varepsilon)$ obtained in Theorem $\ref{qq22artdisth}$ as $\varepsilon \to 0$: 
\begin{enumerate}%%%%%%%%%%%%%
\item[(i)] $\bb{U}_\varepsilon \rightharpoonup \bb{U}, \text{ weakly in } L^2(0,T; H^{1}(\Omega_\delta))$;
\item[(ii)] ${r}_\varepsilon \rightharpoonup {r}, \text{ weakly* in } L^\infty(0,T; L^a (\Omega_\delta))$, ~~ $\varepsilon\nabla {r}_{\varepsilon} \rightharpoonup 0, \text{ weakly in } L^2(Q_{\delta,T})$;
\item[(iii)] Independently of $\delta$, we have:
\begin{enumerate}
\item[(iiia)] $w_{\varepsilon} \rightharpoonup w$, weakly* $L^\infty(0,T; H_0^2(\Gamma))$ and $W^{1,\infty}(0,T; L^2(\Gamma))$;
\item[(iiib)] $w_{\varepsilon} \to w$, in $C^{0,\alpha}([0,T]; H^{2\alpha}(\Gamma))$, for $0<\alpha<1$;
\item[(iiic)] $J_{\varepsilon} \to J$ and $1/J_{\varepsilon} \to 1/J$, in $C^{0,\alpha}([0,T]; C^{0,1-2\alpha}(\Gamma))$ for $0<\alpha<1/2$;
\item[(iiid)] $\theta_{\varepsilon} \to \theta$ weakly* in $L^\infty(0,T; L^2(\Gamma))$ and weakly in $L^2(0,T; H_0^1(\Gamma))$;
\item[(iiie)] $\mathcal{F}(w_{\varepsilon}) \rightharpoonup \mathcal{F}(w)$ in $C([0,T]; H^{-2}(\Gamma))$;
\end{enumerate}%%----------------------------%%
\item[(iv)] $ \nabla^3 w_{\varepsilon} \rightharpoonup \nabla^3 w$, weakly* in $L^\infty(0,T; L^2(\Gamma))$;
\item[(v)] $J_\varepsilon {r}_\varepsilon \to J{r}$, in $C_w(0,T; L^a(\Omega_\delta))$ and $L^2(0,T; H^{-1}(\Omega_\delta))$;
\item[(vi)] $J_\varepsilon {r}_\varepsilon \bb{U}_\varepsilon \rightharpoonup J{r} \bb{U}$, weakly in $L^2(0,T; L^{\frac{6a}{a+6}}(\Omega_\delta))$ and weakly* in $L^\infty(0,T; L^{\frac{2a}{a+1}}(\Omega_\delta))$; 
\item[(vii)] $J_\varepsilon {r}_\varepsilon \bb{U}_\varepsilon \otimes \bb{U}_\varepsilon \rightharpoonup J {r} \bb{U} \otimes \bb{U}$, in $L^1(Q_{\delta,T})$;
\item[(viii)] $J_\varepsilon {r}_\varepsilon \bb{U}_\varepsilon \otimes \bb{w}_\varepsilon \rightharpoonup J {r} \bb{U} \otimes \bb{w}$, in $L^1(Q_{\delta,T})$;
\item[(ix)] $\varepsilon\nabla {r}_\varepsilon \cdot \nabla \bb{U}_\varepsilon \to 0$ and $\varepsilon \nabla {r}_\varepsilon \cdot \bb{U}_\varepsilon \to 0$, in $L^{1}(Q_{\delta,T})$.
\end{enumerate}%%----------------------------%%
\end{lem}
\begin{proof}
To prove $(i) - (iv)$, $(vii)$ and $(viii)$, one can use the same arguments as in Lemma $\ref{qq22weakconvthirdlevel}$, which rely on the uniform estimates given in Lemma $\ref{qq22111}$. Next, to prove the statement $(v)$, one can infer from $\eqref{qq22someweakform}$ that $J_\varepsilon {r}_\varepsilon$ is uniformly continuous in $W^{-1,\frac{2a}{a+1}}(\Omega_\delta)$. Since ${r}_\varepsilon \in \mathcal{W}_{D,\varepsilon}$, we know that $J_\varepsilon{r}_\varepsilon$ is in $C_w(0,T; L^a(\Omega_\delta))$. Due to boundedness of $J_\varepsilon {r}_\varepsilon \in L^\infty(0,T; L^a(\Omega_\delta))$, and compact embedding of $L^a(\Omega_\delta)$ into $H^{-1}(\Omega_\delta)$ for $a>\frac{6}{5}$, the convergences $(v)$ follow (see \cite[Lemma 6.2]{novotnystraskraba}). Consequently by $(i)$ and the uniform bounds given in $\eqref{qq22sec3ineq1}$ and $\eqref{qq22sec3ineq5}$, the convergences in $(vi)$ follow as well. Finally, $(ix)$ follows by the uniform bounds from Lemma $\ref{qq22111}(ii)$ and $(vi)$, so the proof is complete.
\end{proof}

\subsection{Convergence of the pressure on the physical domain $\Omega_\delta^{w_\varepsilon}$}
Here we aim to prove that
\begin{eqnarray*}%%%%%%%%%%%%%
\rho_\varepsilon^\gamma + \delta \rho_\varepsilon^a \rightharpoonup \rho^\gamma + \delta \rho^a\text{ in } L^1((0,T)\times \mathbb{R}^3), \quad \text{as }\varepsilon \to 0.
\end{eqnarray*}%%----------------------------%%
The $L^1$ bound of the pressure is not enough to obtain weak limit in $L^1$ as the function can also converge to a measure, and since we lose additional spatial regularity in the damped continuity equation as $\varepsilon \to 0$, we also cannot obtain any spatial compactness for $\rho$. The alternative is the weak compactness method which consists of proving the weak convergence towards the limit that we later identify by proving the convergence of effective viscous flux and utilizing the renormalized continuity equation to obtain the strong convergence of density in $L^1((0,T)\times \mathbb{R}^3)$.
\subsubsection{The weak convergence of the pressure}
Here we aim to prove that there is a function $\overline{p}$ such that
\begin{eqnarray*}%%%%%%%%%%%%%
\rho_\varepsilon^\gamma + \delta \rho_\varepsilon^a \rightharpoonup \overline{p}\text{ in } L^1((0,T)\times \mathbb{R}^3) \quad \text{as }\varepsilon \to 0.
\end{eqnarray*}%%----------------------------%%
For a set $S$ with regular boundary and $p \in (1, \infty)$, we introduce the following inverse Laplace operator 
\begin{eqnarray*}%%%%%%%%%%%%%
\Delta_S^{-1}:L^p(S) &\to& W^{2,p}(S)\cap W_0^{1,p^*}(S), \\
f &\mapsto& \Delta^{-1} f,
\end{eqnarray*}%%----------------------------%%
with $p^*$ being the Sobolev conjugate index of $p$, which satisfies 
\begin{eqnarray}%%%%%%%%%%%%%
||\nabla \Delta_{S}^{-1}[f] ||_{W^{1,a}(S)} &\leq& C(S)||f||_{L^a(S)},\label{qq22linfty22} \\
||\nabla \Delta_{S}^{-1}[f] ||_{L^\infty(S)} &\leq& C(S)||f||_{L^a(S)}, \label{qq22linfty}
\end{eqnarray}%%----------------------------%%
where the second inequality holds for $a>3$.

\begin{lem}\label{qq22lemmaQproof2}
For any set $Q = I \times B \Subset (0,T) \times \Omega^w(t)$ where $B$ has a regular boundary, the following holds
\begin{eqnarray*}%%%%%%%%%%%%%
\int_Q (\rho_\varepsilon^{\gamma+1} + \delta \rho_\varepsilon^{a+1}) \leq C(Q),
\end{eqnarray*}%%----------------------------%%
where the constant $C(Q)$ is independent of $\varepsilon$.
\end{lem}
\begin{proof}
The proof is a localized version of the standard approach and it was first done in \cite{feireisl2} in the context of rigid bodies immersed in the compressible fluid (see also \cite[Lemma 6.3]{compressible}). First, we define a set $\tilde{Q} = \tilde{I} \times \tilde{B}$ such that $Q \Subset \tilde{Q} \Subset (0,T) \times \Omega^{w_\varepsilon}(t)$ and without loss of generality $|\tilde{I}| \leq 2|I|$ and $\mathcal{M}(\tilde{Q})\leq 2\mathcal{M}(Q)$. Notice that we can always define such a set for a small $\varepsilon$ due to the strong convergence of $w_\varepsilon$ in $C^{0,\alpha}([0,T]; C^{0,1-2\alpha}(\Gamma))$. Now, we choose $\bb{q} = \varphi\ \nabla \Delta_{\tilde{B}}^{-1}[\rho_\varepsilon]$ in $\eqref{qq22movingdomainsystem1234}_2$, where $\varphi \in C_0^\infty(\tilde{Q})$, $\varphi=1$ in $Q$ and $\varphi \geq 0$ in $\tilde{Q}$, to obtain:
\begin{align}%%%%%%%%%%%%%%%
&\int_{\tilde{Q}} \varphi(\rho_\varepsilon^{\gamma+1} + \delta \rho_\varepsilon^{a+1}) \nonumber\\%%
&=-\int_{\tilde{Q}}\rho_\varepsilon \bb{u}_\varepsilon \cdot \big(\partial_t \varphi \nabla \Delta_{\tilde{B}}^{-1}[\rho_\varepsilon]+\varphi \rho_\varepsilon \bb{u}_\varepsilon\big)+\varepsilon \int_{\tilde{Q}} \varphi\rho_\varepsilon \bb{u}_\varepsilon \cdot \nabla \Delta_{\tilde{B}}^{-1} \Big[\frac{1}{J_\varepsilon} \nabla^{w^{-1}} \cdot(\nabla^{w^{-1}} \rho_\varepsilon J_\varepsilon)\Big] \nonumber \\%%
&\quad -\int_{\tilde{Q}}(\rho_\varepsilon \bb{u}_\varepsilon \otimes \bb{u}_\varepsilon ):\big( \nabla\varphi \otimes \nabla \Delta_{\tilde{B}}^{-1}[\rho_\varepsilon] +\varphi \nabla^2 \Delta_{\tilde{B}}^{-1}[\rho_\varepsilon] \big) \nonumber \\
 &\quad +\mu \int_{\tilde{Q}}\nabla \bb{u}_\varepsilon: \Big( \nabla \varphi\cdot\nabla \Delta_{\tilde{B}}^{-1}[\rho_\varepsilon] +\varphi \nabla^2 \Delta_{\tilde{B}}^{-1}[\rho_\varepsilon] \Big) \nonumber \\
&\quad +(\mu +\lambda)\int_{\tilde{Q}} (\nabla\cdot \bb{u}_\varepsilon) \big((\nabla\cdot\varphi)\nabla \Delta_{\tilde{B}}^{-1}[\rho_\varepsilon] +\varphi \rho_\varepsilon \big) \nonumber \\%%
&\quad+ \varepsilon \bint_{\tilde{Q}} \nabla^{w_\varepsilon^{-1}} \rho_\varepsilon \cdot \Big( \varphi\nabla \Delta_{\tilde{B}}^{-1}[\rho_\varepsilon] \cdot \nabla^{w_\varepsilon^{-1}} \bb{u}_\varepsilon+ \big(\nabla^{w_\varepsilon^{-1}}\varphi \cdot \nabla \Delta_{\tilde{B}}^{-1}[\rho_\varepsilon]+\nabla^{w_\varepsilon^{-1}}\nabla \Delta_{\tilde{B}}^{-1}[\rho_\varepsilon]\varphi \big)
\cdot \bb{u}_\varepsilon\Big] \nonumber \\
&:=\sum_{k=1}^6 I_k, \label{qq22thatsimilar}
\end{align}%%%%%%%%
where from the equation $\eqref{qq22movingdomainsystem}_2$ we expressed
\begin{eqnarray*}%%%%%%%%%%%%%
\partial_t (\varphi\ \nabla \Delta_{\tilde{B}}^{-1}[\rho_\varepsilon])= \partial_t \varphi \nabla \Delta_{\tilde{B}}^{-1}[\rho_\varepsilon]-\varphi \rho_\varepsilon \bb{u}_\varepsilon+ \varphi\varepsilon \nabla \Delta_{\tilde{B}}^{-1} \Big[ \frac{1}{J_\varepsilon} \nabla^{w^{-1}} \cdot(\nabla^{w^{-1}} \rho_\varepsilon J_\varepsilon)\Big].
\end{eqnarray*}%%----------------------------%%
First, we want to bound the $\varepsilon$ terms ($I_2$ and $I_6$) by $C(Q)$. The difficulty here is that the inverse divergence operator $ \nabla \Delta_{\tilde{B}}^{-1}$ defined with respect to the physical domain coordinates is acting onto the artificial density damping term, which depends on the second order derivatives with respect to the fixed domain coordinates. In other words, there exists a mismatch of coordinates which creates difficulties when one wants to obtain certain estimates. We start by studying the transformed Laplacian. First, we write
\begin{eqnarray*}%%%%%%%%%%%%%
\Delta^{w_\varepsilon^{-1}} \rho_\varepsilon = \underbrace{\sum_{i,j,k=1}^3 \partial_{x_k}A_i^{-1} \partial_{x_k} A_j^{-1} \partial_{x_i x_j}^2 \rho_\varepsilon }_{:=J_1}+ \underbrace{\sum_{i,j=1}^3 \partial_{x_i x_i}^2 A_j^{-1} \partial_{x_j} \rho_\varepsilon}_{:=J_2}.
\end{eqnarray*}%%----------------------------%%
Thus, the transformed Laplacian behaves as $(\nabla A_{w_\varepsilon})^2 \Delta \rho_\varepsilon + \Delta A_{w_\varepsilon} \nabla \rho_\varepsilon$, where we can also write $(\nabla A_{w_\varepsilon})^2 \Delta \rho_\varepsilon = \nabla \cdot ((\nabla A_{w_\varepsilon})^2 \nabla \rho_\varepsilon) - \nabla(\nabla A_{w_\varepsilon})^2 \nabla \rho_\varepsilon$. We know that we can control the inverse divergence of $\nabla \cdot ((\nabla A_{w_\varepsilon})^2 \nabla \rho_\varepsilon)$ suitably, while the remaining terms only have $\nabla\rho_\varepsilon$ so one can estimate them directly. More precisely, one has
\begin{eqnarray*}%%%%%%%%%%%%%
J_1 = \underbrace{\sum_{i=1}^3 \partial_{x_i} \Big( \sum_{j,k=1}^3 \partial_{x_k}A_i^{-1} \partial_{x_k} A_j^{-1} \partial_{x_j} \rho_\varepsilon \Big)}_{:=K_1} - \underbrace{\sum_{i,j,k=1}^3 \Big( \partial_{x_i x_k}A_i^{-1} \partial_{x_k} A_j^{-1} + \partial_{ x_k}A_i^{-1} \partial_{x_i x_k} A_j^{-1} \Big) \partial_{x_j}\rho_\varepsilon}_{:=K_2}.
\end{eqnarray*}%%----------------------------%%
with $A_i^{-1}:=A_{w_\varepsilon}^{-1} \bb{e}_i$. Notice that $K_1$ is a divergence of a vector-valued function, so we can write $K_1 = \nabla \cdot \bb{K}_1$. It is easy to have
\begin{eqnarray*}%%%%%%%%%%%%%
(\Delta^{w_\varepsilon^{-1}} \rho_\varepsilon) J_\varepsilon = (\nabla \cdot \bb{K}_1)J_\varepsilon - K_2 J_\varepsilon - J_2 J_\varepsilon= \nabla \cdot (\bb{K}_1 J_\varepsilon ) - \bb{K}_1 \nabla J_\varepsilon - K_2 J_\varepsilon -J_2 J_\varepsilon,
\end{eqnarray*}%%----------------------------%%

\noindent
so we deduce
\begin{eqnarray}%%%%%%%%%%%%%
&&|| \nabla \Delta_{\tilde{B}}^{-1} \big[\Delta^{w_\varepsilon^{-1}} \rho_\varepsilon \big]||_{L^{3}(\tilde{I}; L^{3/2}(\tilde{B}))}\nonumber\\[2mm]%%
&& \leq ||\bb{K}_1 J_\varepsilon||_{L^{3}(\tilde{I}; L^{3/2}(\tilde{B}))} + C|| \bb{K}_1 \nabla J_\varepsilon+ K_2 J_\varepsilon+ J_2 J_\varepsilon||_{L^{3}(\tilde{I}; L^{3/2}(\tilde{B}))} \nonumber\\[2mm]%%
&&\leq C\Big[ || (\nabla A_{w_\varepsilon})^2||_{L^\infty(0,T;L^{\infty}(\Omega_\delta))}( || A_{w_\varepsilon}||_{L^\infty(0,T; L^{\infty}(\Omega_\delta))}+|| \nabla A_{w_\varepsilon}||_{L^\infty(0,T; L^{\infty}(\Omega_\delta))}) \nonumber\\
&&\quad +|| \nabla^2 A_{w_\varepsilon}||_{L^\infty(0,T; L^6(\Omega_\delta))}(|| \nabla A_{w_\varepsilon}||_{L^\infty(0,T; L^{\infty}(\Omega_\delta))} + 1)\Big] || \nabla \rho_\varepsilon||_{L^{3}(\tilde{I}; L^{2}(\tilde{B}))}\nonumber \\[2mm]
&&\leq C(E_0,\delta) || \nabla \rho_\varepsilon||_{L^{3}(\tilde{I}; L^{2}(\tilde{B}))}. \label{qq22eps000}
\end{eqnarray}%%----------------------------%%
by using $\eqref{qq22linfty22}$ and $\eqref{qq22linfty}$. Now, from $\eqref{qq22eps000}$, we have
\begin{eqnarray}%%%%%%%%%%%%%
&&\varepsilon \int_{\tilde{Q}}\varphi \rho_\varepsilon \bb{u}_\varepsilon\cdot \nabla \Delta_{\tilde{B}}^{-1} \Big[\Delta^{w_\varepsilon^{-1}} \rho_\varepsilon \Big] \nonumber \\[2mm]
&&\leq C(E_0) \varepsilon ||1||_{L^{6}(0;T; L^{ p}(\tilde{Q}) )} ||\rho_\varepsilon\bb{u}_\varepsilon||_{L^{2}(0;T; L^{ \frac{6a}{a+6}}(\tilde{Q}) )}|| \nabla \Delta_{\tilde{B}}^{-1} \big[\Delta^{w_\varepsilon^{-1}} \rho_\varepsilon \big]||_{L^{3}(0,T; L^{3/2}(\tilde{Q}))}\nonumber \\[2mm]
&&\leq 4C(E_0,\delta)|I|^{1/8} |\mathcal{M}(B)|^{1/p}\varepsilon||\nabla \rho_\varepsilon||_{L^{3}(\tilde{I}; L^{2}(\tilde{B}))}\nonumber \\[2mm]
&&\leq C(E_0,\delta)|I|^{1/8} |\mathcal{M}(B)|^{1/p} =C(Q), \label{qq22similar}
\end{eqnarray}%%----------------------------%%
by $\eqref{qq22nicela}$, for $a>9$ and $p>\frac{6a}{a-6}$, and similarly
\begin{eqnarray*}%%%%%%%%%%%%%
\varepsilon \int_{\tilde{Q}} \varphi \rho_\varepsilon \bb{u}_\varepsilon\cdot \nabla \Delta_{\tilde{B}}^{-1} \Big[\frac{1}{J_\varepsilon}\nabla^{w_\varepsilon^{-1}} \rho_\varepsilon \nabla^{w_\varepsilon^{-1}} J_\varepsilon\Big] \leq C(Q),
\end{eqnarray*}%%----------------------------%%
since the integrand of this term has better integrability than the integrand of the integral studied in $\eqref{qq22similar}$. Combining previous two inequalities we obtain $I_2 \leq C(Q)$.

Next, to estimate the term $I_6$, by using the fact that
\begin{eqnarray*}%%%%%%%%%%%%%
\nabla^{w_\varepsilon^{-1}}\nabla\Delta_{\tilde{B}}^{-1}[\rho_\varepsilon]=\nabla^2 \Delta_{\tilde{B}}^{-1}[\rho_\varepsilon] (\nabla A_{w_\varepsilon})\circ A_w^{-1},
\end{eqnarray*}%%----------------------------%%
we have
\begin{eqnarray}\label{qq22somebound1}%%%%%%%%%%%%%
||\nabla^{w_\varepsilon^{-1}} \nabla \Delta_{\tilde{B}}^{-1}[\rho_\varepsilon] ||_{L^{a}(B)}\leq C(E_0,\delta)|| \nabla^2 \Delta_{\tilde{B}}^{-1}[\rho_\varepsilon]||_{L^a(\tilde{B})} \leq C(E_0,\delta)|| \rho_\varepsilon ||_{L^a(\tilde{B})},
\end{eqnarray}%%----------------------------%%
and similarly
\begin{eqnarray*}%%%%%%%%%%%%%
||\nabla^{w_\varepsilon^{-1}} \Delta_{\tilde{B}}^{-1}[\rho_\varepsilon] ||_{L^{\infty}(\tilde{B})}\leq C(E_0,\delta)|| \nabla \Delta_{\tilde{B}}^{-1}[\rho_\varepsilon]||_{L^\infty(\tilde{B})} \leq C(E_0,\delta)|| \rho_\varepsilon ||_{L^\infty(\tilde{B})},
\end{eqnarray*}%%----------------------------%%
provided $a>3$, so from $\eqref{qq22linfty}$ and $\eqref{qq22somebound1}$, one has
\begin{eqnarray*}%%%%%%%%%%%%%
I_6 &\leq& C || 1 ||_{L^{6}(0,T; L^{18}(\Omega_\delta))} \varepsilon ||\nabla \rho_\varepsilon||_{L^{3}(0,T; L^{\frac{9}{4}}(\Omega_\delta))}\times \\
&&\Big[|| \nabla \bb{u}_\varepsilon||_{L^2(0,T; L^{2}(\Omega_\delta^{w_\varepsilon}))} ||\nabla \Delta_{\tilde{B}}^{-1}[\rho_\varepsilon] ||_{L^\infty(0,T; L^{\infty}(\tilde{B}))}\\
&&+ \big(||\nabla \Delta_{\tilde{B}}^{-1}[\rho_\varepsilon] ||_{L^\infty(0,T; L^{a}(\tilde{B}))} +||\nabla^2 \Delta_{\tilde{B}}^{-1}[\rho_\varepsilon] ||_{L^\infty(0,T; L^{a}(\tilde{B}))} \big)||\bb{u}_\varepsilon||_{L^2(0,T; L^{6}(\Omega_\delta^{w_\varepsilon}))}\Big] \\
&\leq& 4C(E_0,\delta)|I|^{\frac{1}{3}} |\mathcal{M}(B)|^{\frac{1}{18}} =C(Q),
\end{eqnarray*}%%----------------------------%%
for $a\geq 9$. It remains to bound the terms $I_1,I_3,I_4,I_5$ by a constant $C(Q)$. We will study the "worst" term, the convective term $I_3$,
\begin{eqnarray*}%%%%%%%%%%%%%
\hspace*{-0.8cm}&&I_3 \leq C ||1||_{L^2(\tilde{I}; L^{p}(\tilde{B}))} ||\rho_\varepsilon\bb{u}_\varepsilon\otimes \bb{u}_\varepsilon||_{L^2(\tilde{I}; L^{ \frac{6a}{4a+3}}(\tilde{B}) )}  \Big(||\nabla \Delta_{\tilde{B}}^{-1}[\rho_\varepsilon] ||_{L^\infty(\tilde{Q})}+ ||\nabla^2 \Delta_{\tilde{B}}^{-1}[\rho_\varepsilon] ||_{L^\infty(\tilde{I}; L^{a}(\tilde{B}))} \Big) \\
\hspace*{-0.8cm}&&~~~\leq 4 C |I|^{1/2} |\mathcal{M}(B)|^{1/p} C(E_0,\delta) =: C(Q),
\end{eqnarray*}%%----------------------------%%
for $a > 9/2$ and $p=\frac{6a}{2a-9}$, where we used $\eqref{qq22sec3ineq8}$, $\eqref{qq22linfty}$, $\eqref{qq22somebound1}$ and H\"{o}lder's inequality. The remaining terms $I_1,I_4,I_5$ can be estimated in a similar fashion since they even have better regularity, so we finish the proof.
\end{proof}

The proof of the following result is given in Appendix B:
\begin{lem}\label{qq22lemkappa}
For any $\kappa>0$, there is a measurable set $\mathcal{A}_\kappa \Subset (0,T) \times (\Omega_\delta^w(t))$ such that
\begin{eqnarray*}%%%%%%%%%%%%%
\int_{((0,T)\times (\Omega_\delta^w(t)) \setminus \mathcal{A}_\kappa} (\rho_\varepsilon^{\gamma} + \delta \rho_\varepsilon^{a}) \leq \kappa.
\end{eqnarray*}%%----------------------------%%
\end{lem}

\noindent

Combining the previous two lemmas, we have:
\begin{cor}\label{qq22somecor}
There exists $\overline{p}$ such that as $\varepsilon\to 0$,
\begin{eqnarray*}%%%%%%%%%%%%%
\rho_\varepsilon^\gamma+ \delta \rho_\varepsilon^a \rightharpoonup \overline{p}, \quad \text{in } L^1((0,T)\times \mathbb{R}^3).
\end{eqnarray*}%%----------------------------%% 
Moreover, for every $\kappa>0$, there is a set $A_\kappa\Subset Q_T^w$ such that 
\begin{eqnarray*}%%%%%%%%%%%%%
\int_{Q_T^w \setminus A_\kappa} \overline{p} \leq \kappa,
\end{eqnarray*}%%----------------------------%%
and $\overline{p}\rho_\varepsilon \in L^1(A_\kappa)$.
\end{cor}
\noindent

Now, by using this Corollary and Lemma $\ref{qq22weakconveps}$, we have that the limiting functions $(r,\bb{U},w,\theta)$ satisfy the heat equation $\eqref{qq22heat3}$, the continuity equation $\eqref{qq22cont3}$ (or equivalently $\eqref{qq22movingdomainsystem123}_1$) and the following coupled momentum equation:
\begin{align}%%%%%%%%%%%%%%%
&\int_{Q_{\delta,T}^{w}}\rho \bb{u}\cdot \partial_t \bb{q} +\int_{Q_{\delta,T}^{w}}\big[\rho \bb{u} \otimes \bb{u}\big]:\nabla \bb{q}-\mu \int_{Q_{\delta,T}^{w}} \nabla \bb{u}:\nabla\bb{q} -(\mu +\lambda)\int_{Q_{\delta,T}^{w}}( \nabla\cdot \bb{u})(\nabla\cdot \bb{q}) \nonumber \\[2mm]%%
&+\int_{Q_{\delta,T}^{w}} \overline{p} (\nabla \cdot \bb{q})+\bint_{\Gamma_T} \Big[ +\partial_t w \partial_t \psi -\Delta w \Delta \psi
-\mathcal{F}(w) \psi + \nabla \theta\cdot \nabla \psi-\delta \nabla^3 w: \nabla^3 \psi\Big] \nonumber\\
&= \int_0^T \frac{d}{dt}\int_{\Omega_\delta} J \rho \bb{u} \cdot \bb{q} + \int_0^T \frac{d}{dt}\int_{\Gamma} \partial_t w \psi ,\label{qq22movingdomainsystem22}
\end{align}%%%%%%%%
for all $\bb{q} \in C_0^\infty(Q_{T,\delta,\Gamma}^w)$ and $\psi \in C_0^\infty(\Gamma_T^w)$ such that $\bb{q}_{\Gamma^w} = \psi \bb{e}_3$. We will now focus on identifying the limiting function $\overline{p}$.
\subsubsection{The convergence of the effective viscous flux.}\label{qq22secEVF}
Here we want to prove the following convergence of effective viscous flux
\begin{eqnarray}\label{qq22effectvisceps}%%%%%%%%%%%%%
\int_{Q_{\delta,T}^{w_\varepsilon}}\varphi^2(\rho_\varepsilon^\gamma+ \delta \rho_\varepsilon^a- (\mu + 2\lambda)\nabla \cdot \bb{u}_\varepsilon)\rho_\varepsilon \to \int_{Q_{\delta,T}^{w}}\varphi^2(\overline{p} - (\mu + 2\lambda)\nabla \cdot \bb{u})\rho, \text{ in } \Omega_{\varepsilon_0,\delta}^w(t),~~~~
\end{eqnarray}%%----------------------------%%
when $\varepsilon \to 0$, for any $\varphi \in C_0^\infty([0,T] \times \Omega_{\varepsilon_0,\delta}^w(t))$, where
\begin{eqnarray*}%%%%%%%%%%%%%
\Omega_{\varepsilon_0,\delta}^w(t) := \cap_{\varepsilon \leq \varepsilon_0} \Omega_\delta^{w_\varepsilon}(t), \quad \varepsilon_0 \in (0,1).
\end{eqnarray*}%%----------------------------%%
The proof of this convergence that follows is merely a localized version of the standard approach (see \cite{novotnystraskraba}) and it is given here for completeness.\\

We start by first choosing $(\bb{q},\psi)=(\varphi \nabla \Delta^{-1}[\varphi \rho_\varepsilon],0)$ in $\eqref{qq22movingdomainsystem1234}_2$ to obtain
\begin{align}%%%%%%%%%%%%%%%
&I_0: = \int_0^T\int_{\mathbb{R}^3} \varphi^2(\rho_\varepsilon^{\gamma+1} + \delta \rho_\varepsilon^{a+1}) \nonumber\\%%
&=-\int_0^T\int_{\mathbb{R}^3}\rho_\varepsilon \bb{u}_\varepsilon \cdot \partial_t \varphi \nabla \Delta^{-1}[\varphi \rho_\varepsilon]-\int_0^T\int_{\mathbb{R}^3}\rho_\varepsilon \bb{u}_\varepsilon \cdot \varphi \nabla \Delta^{-1}[\partial_t \varphi \rho_\varepsilon]\nonumber \\
&\quad +\int_0^T\int_{\mathbb{R}^3}\rho_\varepsilon \bb{u}_\varepsilon \cdot \varphi \nabla \Delta^{-1}[\varphi \nabla \cdot (\rho_\varepsilon \bb{u}_\varepsilon)] -\int_0^T\int_{\mathbb{R}^3}(\rho_\varepsilon \bb{u}_\varepsilon \otimes \bb{u}_\varepsilon ):\big( \nabla\varphi \otimes \nabla \Delta^{-1}[\varphi\rho_\varepsilon] \big) \nonumber\\
&\quad - \int_0^T\int_{\mathbb{R}^3}(\rho_\varepsilon \bb{u}_\varepsilon \otimes \bb{u}_\varepsilon ):\big(\varphi \nabla^2 \Delta^{-1}[\varphi \rho_\varepsilon] \big)+\mu \int_0^T\int_{\mathbb{R}^3}\nabla \bb{u}_\varepsilon: \big( \nabla \varphi\cdot\nabla \Delta^{-1}[\varphi\rho_\varepsilon]\big)  \nonumber\\
&\quad+\mu \int_0^T\int_{\mathbb{R}^3}\nabla \bb{u}_\varepsilon: \big( \varphi \nabla^2 \Delta^{-1}[\varphi\rho_\varepsilon] \big)+(\mu +\lambda)\int_0^T\int_{\mathbb{R}^3} (\nabla\cdot \bb{u}_\varepsilon) (\nabla\cdot\varphi)\nabla \Delta^{-1}[\varphi \rho_\varepsilon] 
\nonumber \\
&\quad+(\mu +\lambda)\int_0^T\int_{\mathbb{R}^3} (\nabla\cdot \bb{u}_\varepsilon) \varphi^2 \rho_\varepsilon-\int_0^T\int_{\mathbb{R}^3} (\rho_\varepsilon^{\gamma} + \delta \rho_\varepsilon^{a}) \nabla \varphi \cdot \nabla \Delta^{-1}[\varphi \rho_\varepsilon] \nonumber \\
&\quad +\varepsilon \int_0^T\int_{\mathbb{R}^3} \varphi \rho_\varepsilon \bb{u}_\varepsilon \cdot \nabla \Delta^{-1} \Big[ \varphi \frac{1}{J_\varepsilon} \nabla^{w^{-1}} \cdot(\nabla^{w^{-1}} \rho_\varepsilon J_\varepsilon)\Big] \nonumber \\
&\quad+ \varepsilon \int_0^T\int_{\mathbb{R}^3} \nabla^{w_\varepsilon^{-1}} \rho_\varepsilon \cdot \Big[ \varphi\nabla \Delta^{-1}[\varphi \rho_\varepsilon] \cdot \nabla^{w_\varepsilon^{-1}} \bb{u}_\varepsilon\nonumber \\
&\quad \quad \quad \quad \quad \quad\quad \quad \quad \quad + \bb{u}_\varepsilon \cdot \big(\nabla^{w_\varepsilon^{-1}}\varphi  \nabla \Delta^{-1}[\varphi \rho_\varepsilon]+ \varphi \nabla^{w_\varepsilon^{-1}}\nabla \Delta^{-1}[\varphi \rho_\varepsilon]\big) \Big] \nonumber \\
&:=I_1 + ... +I_{10}+E_1+E_2, \label{qq22papiga1}
\end{align}%%%%%%%%
and then we choose $(\bb{q},\psi)=(\varphi \nabla \Delta^{-1}[\varphi \rho],0)$ in $\eqref{qq22movingdomainsystem22}$ to obtain
\begin{align}%%%%%%%%%%%%%%%
J_0&: = \int_0^T\int_{\mathbb{R}^3} \varphi^2 \overline{p}\rho \nonumber\\%%
&=-\int_0^T\int_{\mathbb{R}^3}\rho \bb{u} \cdot \partial_t \varphi \nabla \Delta^{-1}[\varphi \rho]-\int_0^T\int_{\mathbb{R}^3}\rho \bb{u} \cdot \varphi \nabla \Delta^{-1}[\partial_t \varphi \rho]   \nonumber\\%%
&\quad+\int_0^T\int_{\mathbb{R}^3}\rho \bb{u} \cdot \varphi \nabla \Delta^{-1}[\varphi \nabla \cdot (\rho \bb{u})] -\int_0^T\int_{\mathbb{R}^3}(\rho \bb{u} \otimes \bb{u} ):\big( \nabla\varphi \otimes \nabla \Delta^{-1}[\varphi\rho] \big) \nonumber \\%%
&\quad - \int_0^T\int_{\mathbb{R}^3}(\rho \bb{u} \otimes \bb{u} ):\big(\varphi \nabla^2 \Delta^{-1}[\varphi \rho] \big)+\mu \int_0^T\int_{\mathbb{R}^3}\nabla \bb{u}: \big( \nabla \varphi\cdot\nabla \Delta^{-1}[\varphi\rho]\big)  \nonumber\\
&\quad+\mu \int_0^T\int_{\mathbb{R}^3}\nabla \bb{u}: \big( \varphi \nabla^2 \Delta^{-1}[\varphi\rho] \big)+(\mu +\lambda)\int_0^T\int_{\mathbb{R}^3} (\nabla\cdot \bb{u}) (\nabla\cdot\varphi)\nabla \Delta^{-1}[\varphi \rho] 
\nonumber \\
&\quad+(\mu +\lambda)\int_0^T\int_{\mathbb{R}^3} (\nabla\cdot \bb{u}) \varphi^2 \rho -\int_0^T\int_{\mathbb{R}^3} \overline{p} \nabla \varphi \cdot \nabla \Delta^{-1}[\varphi \rho]\nonumber \\
&:=J_1 + ... +J_{10}, \label{qq22papiga2}
\end{align}%%%%%%%%
Defining the operator $\mathcal{R}$ as $\mathcal{R}_{ij}:= \partial_i \Delta^{-1} \partial_j$, one can write $\eqref{qq22papiga1}$ as
\begin{eqnarray}%%%%%%%%%%%%%
&&\int_0^T\int_{\mathbb{R}^3} \varphi^2(\rho_\varepsilon^{\gamma} + \delta \rho_\varepsilon^{a}- (\lambda + 2\mu) \nabla \cdot \bb{u}_\varepsilon)\rho_\varepsilon \nonumber\\
&&=I_1+I_2+I_3'+I_4+I_5+I_6+I_8+I_{10}+E_1+E_2 \nonumber\\
&&\quad + \sum_{i,j=1}^3 \int_0^T \int_{ \mathbb{R}^3} u_\varepsilon^i \big(\varphi \rho_\varepsilon \mathcal{R}_{ij}[\varphi \rho_\varepsilon u_\varepsilon^j] - \varphi \rho_\varepsilon u_\varepsilon^j \mathcal{R}_{ij}[\varphi \rho_\varepsilon] \big) \label{qq22papiga11}
\end{eqnarray}%%----------------------------%%
where
\begin{eqnarray*}%%%%%%%%%%%%%
I_3':= - \int_0^T \int_{\mathbb{R}^3} \varphi \rho_\varepsilon \bb{u}_\varepsilon \cdot \nabla \Delta^{-1}[\nabla \varphi \cdot \bb{u}_\varepsilon \rho_\varepsilon]
\end{eqnarray*}%%----------------------------%%
and similarly for $\eqref{qq22papiga2}$
\begin{eqnarray}%%%%%%%%%%%%%
&&\int_0^T\int_{\mathbb{R}^3} \varphi^2(\overline{p}- (\lambda + 2\mu) \nabla \cdot \bb{u})\rho\nonumber\\ &&=J_1+J_2+J_3'+J_4+J_5+J_6+J_8+J_{10} \nonumber\\
&&\quad + \sum_{i,j} \int_0^T \int_{ \mathbb{R}^3} u^i \big(\varphi \rho \mathcal{R}_{ij}[\varphi \rho u^j] - \varphi \rho u^j \mathcal{R}_{ij}[\varphi \rho] \big) \quad \quad \quad  \label{qq22papiga22}
\end{eqnarray}%%----------------------------%%
where
\begin{eqnarray*}%%%%%%%%%%%%%
J_3':= - \int_0^T \int_{\mathbb{R}^3} \varphi \rho \bb{u} \cdot \nabla \Delta^{-1}[\nabla \varphi \cdot \bb{u} \rho]
\end{eqnarray*}%%----------------------------%%
Taking the difference of $\eqref{qq22papiga11}$ and $\eqref{qq22papiga22}$, one obtains
\begin{eqnarray*}%%%%%%%%%%%%%
&&\int_0^T\int_{\mathbb{R}^3} \varphi^2(\rho_\varepsilon^{\gamma} + \delta \rho_\varepsilon^{a}- (\lambda + 2\mu) \nabla \cdot \bb{u}_\varepsilon)\rho_\varepsilon - \int_0^T\int_{\mathbb{R}^3} \varphi^2(\overline{p}- (\lambda + 2\mu) \nabla \cdot \bb{u}) \rho \\
&&= I_1-J_1+I_2-I_2+I_3'-I_3'+...+I_6-J_6+I_8-J_8+I_{10}-J_{10}+E_1+E_2 \\
&&\quad + \sum_{i,j=1}^3 \int_0^T \int_{ \mathbb{R}^3} u_\varepsilon^i \big(\varphi \rho_\varepsilon \mathcal{R}_{ij}[\varphi \rho_\varepsilon u_\varepsilon^j] - \varphi \rho_\varepsilon u_\varepsilon^j \mathcal{R}_{ij}[\varphi \rho_\varepsilon] \big) \\
&&\quad -\sum_{i,j=1}^3 \int_0^T \int_{ \mathbb{R}^3} u^i \big(\varphi \rho \mathcal{R}_{ij}[\varphi \rho u^j] - \varphi \rho u^j \mathcal{R}_{ij}[\varphi \rho] \big) .
\end{eqnarray*}%%----------------------------%%
The goal is to prove that the right-hand side of the above identity converges to zero as $\varepsilon \to 0$. First, it is straightforward to see that the differences in the second line $I_1 - J_1$, ... $I_{10} - J_{10}$ converge to zero by Lemma $\ref{qq22weakconveps}$ and Corollary $\ref{qq22somecor}$. Next, by using $\eqref{qq22nicela}$ and the uniform bounds, estimating similarly as in $\eqref{qq22similar}$, one can easily obtain
\begin{eqnarray*}%%%%%%%%%%%%%
|E_1|+|E_2| \leq \varepsilon^{1/4}C(E_0,\delta)
\end{eqnarray*}%%----------------------------%%
so $E_1,E_2 \to 0$ as $\varepsilon \to 0$. It remains to prove that the last difference of the commutator terms converges to zero. Since
\begin{eqnarray*}%%%%%%%%%%%%%
\rho_\varepsilon &\rightharpoonup & \rho \quad \text{in } L^a(\mathbb{R}^3) \quad \text{a.e. in } (0,T), \\
\rho_\varepsilon\bb{u}_\varepsilon &\rightharpoonup& \rho\bb{u} \quad \text{in } L^{\frac{2a}{a+1}}(\mathbb{R}^3) \quad \text{a.e. in } (0,T), 
\end{eqnarray*}%%----------------------------%%
by \cite[Lemma 3.4]{global} (which is a direct consequence of div-curl lemma), one has
\begin{eqnarray*}%%%%%%%%%%%%%
\sum_{i,j=1}^3 \big(\varphi \rho_\varepsilon \mathcal{R}_{ij}[\varphi \rho_\varepsilon u_\varepsilon^j] - \varphi \rho_\varepsilon u_\varepsilon^j \mathcal{R}_{ij}[\varphi \rho_\varepsilon] \big) \rightharpoonup \sum_{i,j} \big(\varphi \rho \mathcal{R}_{ij}[\varphi \rho u^j] - \varphi \rho u^j \mathcal{R}_{ij}[\varphi \rho] \big),
\end{eqnarray*}%%----------------------------%%
in $L^r(\mathbb{R}^3)$ and a.e. in $(0,T)$, where
\begin{eqnarray*}%%%%%%%%%%%%%
\frac{1}{r} = \frac{1}{a} + \frac{a+1}{2a} < \frac{5}{6}
\end{eqnarray*}%%----------------------------%%
provided that $a>9/2$. Since $L^r(S)$ is compactly imbedded into $W^{-1,2}(S)$ for any compact set $S$ in $\mathbb{R}^3$, one also obtains
\begin{eqnarray*}%%%%%%%%%%%%%
\sum_{i,j} \big(\varphi \rho_\varepsilon \mathcal{R}_{ij}[\varphi_\varepsilon u_\varepsilon^j] - \varphi \rho_\varepsilon u_\varepsilon^j \mathcal{R}_{ij}[\varphi \rho_\varepsilon] \big) \to \sum_{i,j} \big(\varphi \rho \mathcal{R}_{ij}[\varphi u^j] - \varphi \rho u^j \mathcal{R}_{ij}[\varphi \rho] \big),
\end{eqnarray*}%%----------------------------%%
in $W^{-1,2}(\mathbb{R}^3)$ and a.e. in $(0,T)$. Now, by the uniform estimates given in Lemma $\ref{qq22111}$, one has that $\sum_{i,j} \big(\varphi \rho_\varepsilon \mathcal{R}_{ij}[\varphi_\varepsilon u_\varepsilon^j] - \varphi \rho_\varepsilon u_\varepsilon^j \mathcal{R}_{ij}[\varphi \rho_\varepsilon] \big) $ is uniformly bounded in $L^p(0,T; W^{-1,2}(\mathbb{R}^3))$ for some $p>2$, so the Lebesque dominated convergence theorem combined with the interpolation of Sobolev spaces gives us
\begin{eqnarray*}%%%%%%%%%%%%%
\sum_{i,j} \big(\varphi \rho_\varepsilon \mathcal{R}_{ij}[\varphi_\varepsilon u_\varepsilon^j] - \varphi \rho_\varepsilon u_\varepsilon^j \mathcal{R}_{ij}[\varphi \rho_\varepsilon] \big) \to \sum_{i,j} \big(\varphi \rho \mathcal{R}_{ij}[\varphi u^j] - \varphi \rho u^j \mathcal{R}_{ij}[\varphi \rho] \big),
\end{eqnarray*}%%----------------------------%%
in $L^2(0,T; W^{-1,2}(\mathbb{R}^3))$, which then by the weak convergence of $\varphi \bb{u}_\varepsilon$ in $L^2(0,T; H^1(\mathbb{R}^3))$ implies
\begin{eqnarray*}%%%%%%%%%%%%%
\sum_{i,j} \int_0^T \int_{ \mathbb{R}^3} u_\varepsilon^i \big(\varphi \rho_\varepsilon \mathcal{R}_{ij}[\varphi_\varepsilon u_\varepsilon^j] - \varphi \rho_\varepsilon u_\varepsilon^j \mathcal{R}_{ij}[\varphi \rho_\varepsilon] \big) \to \sum_{i,j} \int_0^T \int_{ \mathbb{R}^3} u^i \big(\varphi \rho \mathcal{R}_{ij}[\varphi u^j] - \varphi \rho u^j \mathcal{R}_{ij}[\varphi \rho] \big) .
\end{eqnarray*}%%----------------------------%%
Thus, we have concluded the convergence given in $\eqref{qq22effectvisceps}$.
\subsubsection{The strong convergence of density}
First, we can prove Theorem $\ref{qq22RCEepsconv}$ in the same way as in Theorem $\ref{qq22weakomegaRCE}$ (without the $\varepsilon$ terms) by relying on the fact that $\rho \in L^\infty(0,T;L^2(\Omega_\delta^{w_\varepsilon}(t)))$, since we have obtained that the equation $\eqref{qq22movingdomainsystem123}_1$ holds. The only difference in the proof is that it is done in the space of distributions, because we do not have the information about the integrability of $\partial_t r$ and $\nabla r$.

To obtain that the limiting pressure $\overline{p}$ is indeed equal to $\rho^\gamma + \delta \rho^a$ in $ L^1(Q_{\delta,T}^w)$, it is enough to prove the strong convergence of the density. For a non-negative $\varphi \in C_0^\infty(Q_{\delta,T}^w)$, by using the convergence of the effective viscous flux $\eqref{qq22effectvisceps}$ and the monotonicity of the function $x \mapsto x^\gamma + \delta x^a$, we obtain
\begin{eqnarray*}%%%%%%%%%%%%%
&&\lim\limits_{\varepsilon\to 0} \int_{Q_{\delta,T}^w}(\mu + 2\lambda) \varphi (\nabla \cdot \bb{u}_\varepsilon \rho_\varepsilon - \nabla \cdot \bb{u} \rho) \\
&&= \lim\limits_{\varepsilon\to 0} \int_{Q_{\delta,T}^w} \varphi \Big[ \big((\mu + 2\lambda) \nabla \cdot \bb{u}_\varepsilon - \rho_\varepsilon^\gamma- \delta \rho_\varepsilon^a\big)\rho_\varepsilon+ \big(\overline{p} - (\mu + 2\lambda)\nabla \cdot \bb{u}\big)\rho \Big] \\
&&\quad +\lim\limits_{\varepsilon\to 0} \int_{Q_{\delta,T}^w} \varphi \big[ \rho_\varepsilon^{\gamma+1}+\delta \rho_\varepsilon^{a+1} - \overline{p} \rho \big] \\
&&=\lim\limits_{\varepsilon\to 0} \int_{Q_{\delta,T}^w} \varphi \big[( \rho_\varepsilon^{\gamma}+\delta \rho_\varepsilon^{a} - \overline{p})(\rho_\varepsilon - \rho) \big] \geq 0,
\end{eqnarray*}%%----------------------------%%
and since $\varphi $ was arbitrary, we have
\begin{eqnarray}\label{qq22c1}%%%%%%%%%%%%%
\overline{\nabla \cdot \bb{u} \rho} \geq \nabla \cdot \bb{u} \rho, \quad \text{a.e. in }Q_{\delta, T}^w,
\end{eqnarray}%%----------------------------%%
where $\overline{\nabla \cdot \bb{u} \rho}$ is the weak limit of $\nabla \cdot \bb{u}_\varepsilon \rho_\varepsilon$. Now, for any function $b$ satisfying the assumptions from Lemma $\ref{qq22weakomegaRCE}$, the following holds
\begin{eqnarray*}%%%%%%%%%%%%
&&\int_{0}^t \int_{\Omega_\delta^{w_\varepsilon}}\frac{1}{J_\varepsilon}\nabla^{w_\varepsilon^{-1}} J_\varepsilon\cdot \nabla^{w^{-1}} b(\rho_\varepsilon) = \int_{0}^t \int_{\Omega_\delta} \frac{1}{J_\varepsilon^2} \nabla J_\varepsilon\cdot \nabla b(r_\varepsilon) \\
&&= \underbrace{\int_{0}^t \int_{ \partial \Omega_\delta} \frac{1}{J_\varepsilon^2}\nabla J_\varepsilon\cdot \nu b(r_\varepsilon)}_{=0} - \int_{0}^t \int_{\Omega_\delta} \nabla \cdot \Big(\frac{\nabla J_\varepsilon}{J_\varepsilon^2}\Big) b(r_\varepsilon)=- \int_{0}^t \int_{\Omega_\delta} \Big[\frac{\Delta J_\varepsilon}{J_\varepsilon^2}-2\frac{|\nabla J_\varepsilon|^2}{J_\varepsilon^3} \Big] b(r_\varepsilon),
\end{eqnarray*}%%----------------------------%%
so by formally choosing\footnote{Here, one should choose $b(x) = x ln(x+h)$ for $h>0$ and then pass to the limit $h \to 0^+$.} $b(x) =x \text{ln} x$ and $\varphi= \chi_{[0,t]}$ in the renormalized continuity inequality $\eqref{qq22RCEepsmov}$, we obtain
\begin{eqnarray}%%%%%%%%%%%%%
&&-\varepsilon \int_{0}^t \int_{\Omega_\delta} \Big[\frac{\Delta J_\varepsilon}{J_\varepsilon^2}-2\frac{|\nabla J_\varepsilon|^2}{J_\varepsilon^3} \Big] r_\varepsilon \ln (r_\varepsilon)+ \int_0^t \int_{\Omega_\delta^{w_\varepsilon}} (\nabla \cdot \bb{u}_\varepsilon) \rho_\varepsilon \nonumber\\
&&\leq \int_{\Omega_\delta^w} \rho_\varepsilon(0) \text{ln}(\rho_\varepsilon(0)) - \int_{\Omega_\delta^w} \rho_\varepsilon(t) \text{ln}(\rho_\varepsilon(t)).\label{qq22c2}
\end{eqnarray}%%----------------------------%%
On the other hand, from $\eqref{qq22RCEeps2}$ we can similarly have
\begin{eqnarray}\label{qq22c3}%%%%%%%%%%%%%
\int_0^t \int_{\Omega_\delta^w} (\nabla \cdot \bb{u}) \rho = \int_{\Omega_\delta^w} \rho(0) \text{ln}(\rho(0)) - \int_{\Omega_\delta^w} \rho(t) \text{ln}(\rho(t)).
\end{eqnarray}%%----------------------------%%
Now, since $x \ln x \leq x^2$, by the uniform estimates given in Lemma $\ref{qq22111}$, one has
\begin{eqnarray*}%%%%%%%%%%%%%
\varepsilon \int_{0}^t \int_{\Omega_\delta} \Big[\frac{\Delta J_\varepsilon}{J_\varepsilon^2}-2\frac{|\nabla J_\varepsilon|^2}{J_\varepsilon^3} \Big] r_\varepsilon \ln (r_\varepsilon) \to 0, \quad \text{as } \varepsilon \to 0,
\end{eqnarray*}%%----------------------------%%
which by $\eqref{qq22c1}$, $\eqref{qq22c2}$ and $\eqref{qq22c3}$ implies
\begin{eqnarray*}%%%%%%%%%%%%%
\limsup\limits_{\varepsilon \to 0} \int_{\Omega_\delta^{w_\varepsilon}}\rho_\varepsilon(t) \text{ln}(\rho_\varepsilon(t)) \leq \int_{\Omega_\delta^w} \rho(t) \text{ln}(\rho(t)),
\end{eqnarray*}%%----------------------------%%
so by the convexity of the function $f(x)= x ln x$, we obtain that $\rho_\varepsilon \to \rho$ in $L^1((0,T)\times \mathbb{R}^3)$ (see \cite[Corollary 3.33]{novotnystraskraba}). By Corollary $\ref{qq22somecor}$, we obtain that $\overline{p} = \rho^\gamma + \delta \rho^a$ a.e. in $(0,T)\times \mathbb{R}^3$, so the proof of Theorem $\ref{qq22artprth}$ is finished.

\sectionmark{Artificial pressure, fixed domain and structure regularization limit} 
\section{The vanishing artificial pressure, fixed reference domain collapse and the structure regularization limits}\label{qq22section6}
\sectionmark{Artificial pressure, fixed domain and structure regularization limit} 
In this section, we will prove the third main result given in Theorem $\ref{qq22mainth}$. The desired solutions will be obtained as a limit of the solutions $(r_\delta, \bb{U}_\delta, w_\delta, \theta_\delta)$ constructed in Theorem $\ref{qq22artprth}$ by letting $\delta \to 0$.
\begin{thm}\label{qq22RCEdelta}
The solution constructed in Theorem $\ref{qq22mainth}$ satisfies the following renormalized continuity equation
\begin{align}%%%%%%%%%%%%%
\int_0^T \frac{d}{dt} \int_{\mathbb{R}^3} b(\rho) \varphi-\int_0^T \int_{\mathbb{R}^3} \big( b(\rho) \partial_t \varphi + b(\rho)\bb{u}\cdot \nabla \varphi \big) = -\int_0^T \int_{\mathbb{R}^3} \big( \rho b'(\rho) -b(\rho) \big)(\nabla \cdot \mathcal{E}^w[\bb{u}])\varphi, \nonumber\\ \label{qq22RCEdeltaeq}
\end{align}%%----------------------------%%
for any $b \in C^1(\mathbb{R})$ such that $b'(x)=0$, for all $x\geq M_b$, where $M_b$ is a constant.
\end{thm}

\subsection{Convergence on the fixed reference domain $\Omega$}
For the approximate solutions $(r_\delta,\bb{U}_\delta, w_\delta, \theta_\delta)$ constructed in Theorem $\ref{qq22artprth}$, one has:
\begin{lem}\label{qq22weakconvdel}
The following convergences hold as $\delta \to 0$,
\begin{enumerate}%%%%%%%%%%%%%
\item[(i)] $\bb{U}_{\delta} \rightharpoonup \bb{U}, \text{ weakly in } L^2(0,T; W^{1, 2^-}(\Omega))$;
\item[(ii)] ${r}_{\delta} \rightharpoonup {r}, \text{ weakly* in } L^\infty(0,T; L^{\gamma} (\Omega))$;
\item[(iiia)] $w_{\delta} \rightharpoonup w$, weakly* $L^\infty(0,T; H_0^2(\Gamma))$ and $W^{1,\infty}(0,T; L^2(\Gamma))$;
\item[(iiib)] $w_{\delta} \to w$, in $C^{0,\alpha}([0,T]; H^{2\alpha}(\Gamma))$, for $0<\alpha<1$;
\item[(iiic)] $J_{\delta} \to J$ and $1/J_{\delta} \to 1/J$, in $C^{0,\alpha}([0,T]; C^{0,1-2\alpha}(\Gamma))$ for $0<\alpha<1/2$;
\item[(iiid)] $\theta_{\delta} \to \theta$ weakly* in $L^\infty(0,T; L^2(\Gamma))$ and weakly in $L^2(0,T; H^1(\Gamma))$; 
\item[(iiie)] $\mathcal{F}(w_{\delta}) \to \mathcal{F}(w_{\delta})$ in $C([0,T]; H^{-2}(\Gamma))$;
\item[(iiif)] $\delta \nabla^3 w_{\delta} \rightharpoonup 0$, weakly in $L^\infty(0,T; L^2(\Gamma))$;
\item[(iv)] $J_\delta {r}_\delta \to J{r}$, in $C_w(0,T; L^a(\Omega))$;
\item[(v)] $J_\delta {r}_\delta \bb{U}_\delta \rightharpoonup J{r} \bb{U}$, weakly in $L^2(0,T; L^{(\frac{6a}{a+6})^-}(\Omega))$ and weakly* in $L^\infty(0,T; L^{\frac{2a}{a+1}}(\Omega))$; 
\item[(vi)] $J_\delta {r}_\delta \bb{U}_\delta \otimes \bb{U}_\delta \rightharpoonup J {r} \bb{U} \otimes \bb{U}$, in $L^1(Q_{T})$;
\item[(vii)] $J_\delta {r}_\delta \bb{U}_\delta \otimes \bb{w}_\delta \rightharpoonup J {r} \bb{U} \otimes \bb{w}$, in $L^1(Q_{T})$.
\end{enumerate}%%----------------------------%%
\end{lem}
\begin{proof}
The proof can be carried out in the same way as in Lemma $\ref{qq22weakconveps}$. The only difference is that the domain transformation $A_w$ is now only in $C^{0,\alpha}(0,T; C^{0,1-2\alpha}(\Omega))$ for any $0<\alpha<1/2$, so some of the convergences (in particular in $(i)$ and $(v)$) are slightly weaker than those given in Lemma $\ref{qq22weakconveps}$.
\end{proof}

\noindent
\subsection{Convergence of the pressure on the physical domain $\Omega_\delta^{w_\delta}$}
The proof of the convergence of the pressure can, once again, be divided into the following steps: 
\begin{enumerate}
\item Weak convergence of pressure;
\item Convergence of effective viscous flux;
\item Strong convergence of density.
\end{enumerate}

\subsubsection{Weak convergence of pressure}
We use the same idea as in the previous section. First, we have:
\begin{lem}\label{qq22lemmaQproof3}
Let $\gamma>\frac{3}{2}$ and $0<\theta < \frac{2}{3} \gamma-1$. Then, for any parabolic cube $Q = I \times B \Subset (0,T) \times \Omega^w(t)$  where $B$ has a regular boundary, the following holds
\begin{eqnarray*}%%%%%%%%%%%%%
\int_Q (\rho_\delta^{\gamma+\theta} + \delta \rho_\delta^{a+\theta}) \leq C(Q),
\end{eqnarray*}%%----------------------------%%
where the constant $C(Q)$ is independent of $\delta$.
\end{lem}
\begin{proof}
The proof is quite similar to the one given in Lemma $\ref{qq22lemmaQproof2}$. Formally\footnote{Here, when we choose this test function, the term $\psi\partial_t \nabla \Delta^{-1}[\rho_\delta^\theta]$ is not necessarily integrable. Therefore, this Lemma should rigorously be proved by using the renormalized continuity equation $\eqref{qq22RCEeps2}$ with functions $b_k$ which are cut-off functions of $b(x) = x^\theta$. Then, we pass to the limit $k\to+\infty$ (see \cite[Section 7.95]{novotnystraskraba} for more details).}, by testing the equation $\eqref{qq22movingdomainsystem}$ by $(\psi\nabla \Delta^{-1}[\rho_\delta^\theta],0)$, we get the identity similar to the one given in $\eqref{qq22thatsimilar}$. To obtain the bound, we study the "worst" term, which is the convective term
\begin{eqnarray*}%%%%%%%%%%%%%
&&|\int_{Q}\big(\rho_\delta \bb{u}_\delta \cdot \big[\psi \nabla^2\Delta^{-1}[\rho_\delta^\theta] \big)\cdot \bb{u}_\delta | \\
&&\leq C||\rho_\delta ||_{L^\infty(0,T; L^\gamma(\Omega_\delta^{w_\delta}(t)))} ||\bb{u}_\delta||_{L^2(0,T; L^{6^-}(\Omega_\delta^{w_\delta}(t)))}^2||\rho^\theta||_{L^\infty (0,T; L^{ \frac{\gamma}{\theta}}(\Omega_\delta^{w_\delta}(t)))}||1||_{L^2(I; L^q(B))} \\
&&\leq C ||\rho_\delta ||_{L^\infty(0,T; L^\gamma(\Omega_\delta^{w_\delta}(t)))}^2||\nabla\bb{u}_\delta||_{L^2(0,T; L^{2^-}(\Omega_\delta^{w_\delta}(t)))}^2 |I|^{1/2}\mathcal{M}(B)^{1/q}= C(Q),
\end{eqnarray*}%%----------------------------%%
where $q\in (1,\infty)$ is such that $\frac{1}{\gamma}+\frac{2}{6}+ \frac{\theta}{\gamma}+\frac{1}{q}<1$. Other terms can be estimated in a similar way, since they have better regularity, so the proof is complete.
\end{proof}
\begin{lem}
Let $\gamma>\frac{12}{7}$. Then, for any $\kappa>0$ there is a set $A_\kappa \Subset (0,T) \times \Omega^w$ such that
\begin{eqnarray*}%%%%%%%%%%%%%
\int_{(0,T)\times \Omega^w \setminus A_\kappa} (\rho_\delta^{\gamma} + \delta \rho_\delta^{a}) \leq \kappa.
\end{eqnarray*}%%----------------------------%%
\end{lem}
\begin{proof}
This proof is the same as the proof of Lemma $\ref{qq22lemkappa}$ given in Appendix B, where the condition $\gamma>12/7$ is used to bound the term $I_3$ in $\eqref{qq22ajaja}$. Notice that here the domain $\Omega_\delta^{w_\delta}$ collapses to $\Omega^w$, so we only need to prove the inequality $\eqref{qq22set1}$ with $\rho_\varepsilon$ being replaced by $\rho_\delta$, but we do not need to do the same for the inequality $\eqref{qq22set2}$.
\end{proof}
Combining the previous two lemmas, we obtain:
\begin{cor}\label{qq22somecor2}
We have as $\delta \to 0$
\begin{eqnarray*}%%%%%%%%%%%%%
\rho_\delta^{\gamma} + \delta \rho_\delta^{a} &\rightharpoonup& \overline{p}, \quad \text{in } L^1((0,T)\times \Omega^w), \\
\delta \rho_\delta^{a} &\rightharpoonup& 0,  \quad  \text{in } L^1((0,T)\times \Omega^w).
\end{eqnarray*}%%----------------------------%%
Moreover, for every $\kappa>0$ there is a set $A_\kappa\Subset (0,T)\times \Omega^w(t)$ such that 
\begin{eqnarray*}%%%%%%%%%%%%%
\int_{(0,T)\times \Omega^w(t) \setminus A_\kappa} \overline{p} \leq \kappa,
\end{eqnarray*}%%----------------------------%%
and $\overline{p}\rho_\varepsilon^\theta \in L^1(A_\kappa)$, for any $0<\theta < \frac{2}{3} \gamma-1$.
\end{cor}

\noindent

Without proof, we state a simple result which we need for convergence of integral terms when the fixed reference domain $\Omega_\delta$ collapses to $\Omega$:
\begin{lem}\label{qq22lemsetconv}
If $\{f_\delta \}_{\delta>0}$ is a bounded family of functions in $L^p(\Omega_\delta)$ for some $p>1$,  then there is subsequence, which we still denote it as $\{f_\delta \}_{\delta>0}$, such that $f_\delta \rightharpoonup f$, weakly in $L^p(\Omega)$, as $\delta \to 0$.
\end{lem}

Now, by using Lemma $\ref{qq22weakconvdel}$, Corollary $\ref{qq22somecor2}$ and Lemma $\ref{qq22lemsetconv}$, we obtain that the limiting functions $(\rho,\bb{u},w,\theta)$ and $\overline{p}$ obtained in Lemma $\ref{qq22weakconvdel}$ and Corollary $\ref{qq22somecor2}$, respectively, satisfy the heat and continuity equations in the sense of Definition $\ref{qq22weaksolution}$ and the following coupled momentum equation:
\begin{align*}%%%%%%%%%%%%%%%
&\int_{Q_T^{w}}\rho \bb{u}\cdot \partial_t \bb{q} +\int_{Q_T^{w}}(\rho \bb{u} \otimes \bb{u}):\nabla \bb{q}-\mu \int_{\Omega^{w}} \nabla \bb{u}:\nabla\bb{q} -(\mu +\lambda)\int_{\Omega^{w}} (\nabla\cdot \bb{u})(\nabla\cdot \bb{q} ) \nonumber \\[2mm]%%
&+\int_{Q_T^{w}} \overline{p} (\nabla \cdot \bb{q}) +\bint_{\Gamma_T} \Big[ \partial_t w \partial_t \psi -\Delta w \Delta \psi
-\mathcal{F}(w) \psi+\nabla \theta \cdot \nabla \psi\Big] \\[2mm]%%
&=\int_0^T \frac{d}{dt}\int_{\Omega} J \rho \bb{u} \cdot \bb{q}+ \int_0^T \frac{d}{dt}\int_{\Gamma} \partial_t w \psi.
\end{align*}%%%%%%%%
It remains to identify the limit $\overline{p}$.

\subsubsection{Strong convergence of density}
The following proofs are merely a localized versions of the standard approach given by Feireisl \cite{feireisl1} (see also \cite{compressible} for more details), so we only present the main steps here:
\begin{enumerate}
\item Introduce a $L^\infty$ truncation function
\begin{eqnarray*}%%%%%%%%%%%%%
T_k(x) = k T\big(\frac{x}{k}\big), \quad x\in \mathbb{R}, ~k \in \mathbb{N},
\end{eqnarray*}%%----------------------------%%
where $T(x)$ is a smooth concave scalar function such that $T(x) = x$ for $x\leq 1$ and $T(x) = 2$ for $x\geq 3$, and similarly as in the previous section, we have the convergence
\begin{eqnarray*}%%%%%%%%%%%%%
\int_{(0,T)\times \Omega^{w_\delta}} (\rho_\delta^\gamma + \delta \rho_\delta^a - (\lambda + 2\mu)\nabla\cdot \bb{u}_\delta) T_k(\rho_\delta) \to \int_{(0,T)\times \Omega^{w}} (\overline{p}- (\lambda + 2\mu)\nabla\cdot \bb{u}) T^{1,k}.
\end{eqnarray*}%%----------------------------%%
where $T^{1,k}$ is the weak limit of $T_k(\rho_\delta)$ as $\delta \to 0$.
\item In the renormalized continuity equation $\eqref{qq22RCEeps2}$, we choose $b = T_k$ and pass to the limit $\delta \to 0$. Denoting $T^{2,k}$ as the weak limit of $(T_k'(\rho_\delta)\rho_\delta-T_k(\rho_\delta))\nabla \cdot \bb{u}_\delta$, we can obtain the following identity
\begin{eqnarray*}%%%%%%%%%%%%%
\partial_t T^{1,k}+ \nabla \cdot (T^{1,k} \bb{u}) + T^{2,k} = 0
\end{eqnarray*}%%----------------------------%%
and then by a standard smoothing procedure obtain that
\begin{eqnarray*}%%%%%%%%%%%%%
\partial_t b(T^{1,k})+ \nabla\cdot (b(T^{1,k}) \bb{u})+(b'(T^{1,k})T^{1,k}- b(T^{1,k})) \nabla \cdot \bb{u}+ b'(T^{1,k})T^{2,k} = 0,
\end{eqnarray*}%%----------------------------%%
holds for $b$ such that $b'(x) = 0$ for $x\geq M_b$ where $M_b$ is a positive constant. Then, we obtain Theorem $\ref{qq22RCEdelta}$, by proving that $T^{1,k} \to \rho$ and $T^{2,k}\to 0$, in $L^{\gamma^-}((0,T)\times \mathbb{R}^3)$ and $L^1((0,T)\times \mathbb{R}^3)$, respectively, as $k \to +\infty$. Here the key point is to control the amplitude of oscillations $\lim\limits_{\delta\to 0}\sup\int_{(0,T)\times \mathbb{R}^3}|T_k(\rho_\delta) - T_k(\rho)|^{\gamma+1}$ by a constant independent of $k$.
\item Finally, we define the function
\begin{eqnarray*}%%%%%%%%%%%%%
L_k(x) = \begin{cases}
x \ln x, & 0\leq x<k, \\
x \ln k + x \bint_k^x T_k(s)s^{-2} ds, &z \geq k,
\end{cases}
\end{eqnarray*}%%----------------------------%%
which is a suitable function for the renormalized continuity equation and also approximates $x\ln x$. Then, we take the difference of the renormalized continuity equations $\eqref{qq22RCEdeltaeq}$ satisfied by $(\rho_\delta,\bb{u}_\delta)$ and $\eqref{qq22RCEeps2}$ satisfied by $(\rho,\bb{u})$, we choose $b = L_k$ and pass to the limit $\delta \to 0$ and $k \to +\infty$ to obtain
\begin{eqnarray*}%%%%%%%%%%%%%
\lim\limits_{\delta \to 0} \int_{(0,T)\times \mathbb{R}^3} \rho_\delta \ln \rho_\delta \leq \int_{(0,T)\times \mathbb{R}^3} \rho \ln \rho,
\end{eqnarray*}%%----------------------------%%
so by the convexity of the function $x \mapsto x\ln x$, it follows that $\rho_\delta \to \rho$ in $L^1((0,T)\times \mathbb{R}^3)$. Thus, by Corollary $\ref{qq22somecor2}$, we obtain $\overline{p} = \rho^\gamma$ a.e. in $Q_T^w$.
\end{enumerate}

\subsubsection{The lifespan of the solution}
It is already known that the energy inequality $\eqref{qq22energymain}$ is satisfied by the limiting functions $(\rho,\bb{u},w,\theta)$ due to Lemma $\ref{qq22111}$. Thus, to finish the proof of Theorem $\ref{qq22mainth}$, it remains to prove that the time interval of the solution can be prolonged either to $+\infty$ or to any time $T<T^*$, where $T^*$ is the moment when the colision of the elastic structure $\Gamma^w$ and the bottom of the cavity $\Gamma \times \{-1\}$ occurs. We follow the approach given in \cite[pp. 397-398]{time} (see also \cite[Theorem 7.1]{BorSun}) to study this issue. 

Let $[0, T_1]$ be the time interval of the solution we have constructed in the previous section. First, from Lemma $\ref{qq22111}(iii)$, we know that $c_1: =1+\min\limits_{X\in\Gamma} w(T_1,X) \geq \frac{c_0}{2} \geq 0$, where $c_0= \min\limits_{x\in\Gamma} w_0(X) + 1$. We can now again construct a solution on the time interval $[T_1,T_2]$ such that $T_2-T_1 = \big(\frac{c_1}{2C(E(0))}\big)^4$, which ensures
\begin{eqnarray*}%%%%%%%%%%%%%
&& c_2:=\min\limits_{t \in [0,T_2], X\in\Gamma}w(t,X)+1 \geq\min\limits_{t \in [0,T_1], X\in\Gamma} w(t,X) +1-C(E(0))(T_n-T_{n-1})^{\frac{1}{4}} \\
&&\geq c_1 -C(E(0))(T_n-T_{n-1})^{\frac{1}{4}}\geq \frac{c_1}{2}.
\end{eqnarray*}%%----------------------------%%
Repeating this process any number of times $n\in\mathbb{N}$ while always choosing $T_{n}-T_{n-1} =  \big(\frac{c_{n-1}}{2C(E_0)}\big)^4$, one obtains the solution on the time interval $[0,T_n]$. Denote by 
\begin{eqnarray*}%%%%%%%%%%%%%
c^*: = \lim\limits_{n\to\infty} \min\limits_{t\in[0,T_n], x\in\Gamma} w(t,X) + 1
\end{eqnarray*}%%----------------------------%%
and $T^* := \lim\limits_{n\to\infty}T_n$. If $c^* = 0$, then $T^*$ is the moment when the structure reaches the bottom of the cavity $\{z=-1\}$ so the proof is finished. Otherwise if $c^*>0$, then by construction
\begin{eqnarray*}%%%%%%%%%%%%%
T_{n}-T_{n-1} =  \Big(\frac{c_n}{2C(E_0)}\Big)^4 \geq \Big(\frac{c^*}{2C(E_0)}\Big)^4>0,
\end{eqnarray*}%%----------------------------%%
so $T^*=\infty$. Thus, the proof is finished.

\subsection{Conclusions and discussions}\label{qq22concl}
In this paper, we proved the existence of a weak solution for an interaction problem between a compressible viscous fluid and a nonlinear thermoelastic plate by constructing a novel decoupling approximation scheme. This way, we have filled a gap in theory, as to our knowledge, no any result was available for such a problem in which the structure is governed by \textbf{nonlinear equation(s)} and in which a \textbf{decoupling scheme} was constructed. It is easy to see that the same result holds for the corresponding two-dimensional problem with $\gamma>1$, and when the structure nonlinearity $\mathcal{F}$ satisfies the assumptions (A1) and (A2) given in section $\ref{qq22sec1.2}$, then the same result holds if the structure is described by an elasticity equation (without the heat conduction). However, when the structure nonlinearity is of the form $\mathcal{F} = \Delta (\Delta w)^3$, then the heat equation for the structure is used in Lemma $\ref{qq226.7}$ to obtain the strong convergence of $\partial_t w_\delta$ in $L^2(\Gamma_T)$ (see Appendix A), so the same result cannot be attained in the same way if the structure doesn't conduct heat.

As one can notice, the scheme constructed in this paper is not fully discrete both in time and space, so it cannot be directly used for numerical purposes. It is known that numerical schemes for compressible viscous fluids converge, at least rigorously, under assumption $\gamma>3$ (see \cite{pokorny}), so one could not expect anything better for $\gamma$ in the context of fluid-structure interaction problems. The scheme constructed in this paper covers a wider range for $\gamma$ (greater than $12/7$) and thus more physically relevant cases. Decreasing this lower bound for $\gamma$ from $12/7$ to $3/2$, as it is in the standard theory for compressible viscous fluids, and constructing a decoupling numerical scheme for this interaction problem are interesting problems. We refer to a recent result \cite{schshe} for a monolithic (non-decoupling) numerical scheme. 

In contrast with the problem studied in \cite{compressible}, where the elastic boundary of the fluid domain is a surface deforming in its normal direction, the geometry of the model studied in this paper seems to be more restrictive. However, the proofs presented here should work equally well with the other geometry, with only essential difference due to geometry being the proof that the mass of the approximate pressure doesn't concentrate near the boundary, see the details given in Appendix B and \cite[Lemma 6.4]{compressible}, respectively. Moreover, compared to \cite{compressible}, the proof given in Appendix B seems to be more difficult as the domain has corners. We believe that the approach presented in this proof could be generalized to a larger class of domains\footnote{Such domains would of course require that locally in time the elastic structure can deform without intersecting the rigid part of the boundary.} for which the rigid part of the boundary is of $W^{1,p}$-regularity in the sense of \cite{kukucka}, for $p>\max \{2, 3\gamma/(2\gamma-3)\}$, by using the same ideas from \cite{kukucka} to construct the test function $\bb{q}_K^1$ given in Appendix B.

\setcounter{equation}{0}

\section*{Appendix A: Compressible fluid interacting with a quasilinear thermoelastic plate}\label{secasdasd}
\addcontentsline{toc}{section}{\numberline{}{Appendix A: Compressible fluid interacting with a quasilinear thermoelastic plate}}
  
\sectionmark{Compressible fluid interacting with a quasilinear thermoelastic plate}
\setcounter{equation}{0}
In this appendix, we sketch the main idea to obtain the same result as in Theorem $\ref{qq22mainth}$ for the case when $\mathcal{F}(w) = \Delta(\Delta w)^3$, with the potential being $\Pi(w) = \frac{1}{4}|| \Delta w ||_{L^4(\Gamma)}^4$. First, it is easy to have that 
\begin{eqnarray*}%%%%%%%%%%%%%
&&\big|\int_\Gamma( (\Delta w_1)^3 - (\Delta w_2)^3) \Delta \psi \big| = \big| \int_\Gamma ((\Delta w_1)^2 +\Delta w_1 \Delta w_2+ (\Delta w_2)^2)(\Delta w_1 - \Delta w_2)\Delta \psi \big| \\
&&\leq \frac{3}{2}(||\Delta w_1||_{L^{2,4}(\Gamma)}^2+||\Delta w_2||_{L^{4}(\Gamma)}^2 ) || \Delta w_1 - \Delta w_2 ||_{L^4(\Gamma)} || \Delta \psi||_{L^{4}(\Gamma)},
\end{eqnarray*}%%----------------------------%%
for any $w_1, w_2,\psi \in W^{2,4}(\Gamma)$, so $\mathcal{F}$ is locally Lipschitz continuous from $W^{2,4}(\Gamma)$ to $[W^{2,4}(\Gamma)]'$. This means that one can solve the (SSP) in the same way as in Lemma $\ref{qq22sspestimateeps}$. The proof of the convergence for $\Delta t\to 0$, $k\to \infty$ and $\varepsilon\to 0$ can be carried out in the same way as in the sections 4.3, 4.4 and 4.5, since we have enough spatial and time regularity to pass the convergence in the term $\mathcal{F}$ by the Aubin-Lions lemma due to the bound that comes from the regularizing term $\delta||\nabla^3 w||_{L^\infty(0,T; L^2(\Gamma))}$. Unlike the semilinear case which was studied in Theorem $\ref{qq22mainth}$, passing the limit in this term when $\delta \to 0$ requires more effort. Let $G$ denote the weak limit of $\mathcal{F}(w_\delta)$ in $[W_0^{2,4}(\Gamma)]'$. First, in the same way as in Theorem $\ref{qq22mainth}$, one can conclude that the following equation is satisfied by the limiting functions
\begin{align}%%%%%%%%%%%%%%%
& \int_{Q_{T}}\Big[J{r} \bb{U} \cdot \partial_t \bb{q} +J\big[({r} \bb{U}-{r} \bb{w}) \cdot \nabla^{w}\bb{q}\big]\cdot \mathbf{U}-\mu J\nabla^{w} \mathbf{U}:\nabla^{w}\bb{q} \nonumber \\
&-(\mu +\lambda) J (\nabla^{w}\cdot \mathbf{U})(\nabla^{w}\cdot \bb{q})
+ J {r}^\gamma (\nabla^{w}\cdot \bb{q})\Big] \nonumber \\
&+\bint_{\Gamma_T} \Big[ \partial_t w \partial_t \psi-\Delta w \Delta \psi
-G\psi + \nabla \theta \cdot \nabla \psi\Big] =\int_0^T \frac{d}{dt} \int_{\Omega}J {r} \bb{U}\cdot\bb{q} +\int_0^T \frac{d}{dt} \int_{\Gamma} \partial_t w \psi, \quad\quad\quad\label{qq22almostthere}
\end{align}
for all $\bb{q} \in C_0^\infty(Q_{\Gamma,T})$ and $\psi \in \Gamma_T$ such that $\bb{q}_{\Gamma\times\{0\}} = \psi\bb{e}_3$.

\begin{lem}\label{qq226.7}
The following hold:
\begin{enumerate}
\item[(i)] $\partial_t w_\delta \to \partial_t w$ in $L^2(\Gamma_T)$, as $\delta \to 0$;
\item[(ii)] $G = \Delta(\Delta w)^3$ in $[W^{2,4}(\Gamma)]'$ for almost all $t\in[0,T]$.
\end{enumerate}
\end{lem}
\begin{proof}
We define the extension operator $R :L^2(\Gamma)\to L^2(\Omega)$ as
\begin{eqnarray*}%%%%%%%%%%%%%
R[f] := (1+z) f \bb{e}_3,
\end{eqnarray*}%%----------------------------%%
and $R_\delta: L^2(\Gamma)\to L^2(\Omega_\delta)$ as
\begin{eqnarray*}%%%%%%%%%%%%%
R_\delta[f]:= \begin{cases}
(1+z) f \bb{e}_3, \quad \text{in } \Omega, \\
0, \quad \text{in }\Omega_\delta\setminus\Omega.
\end{cases}
\end{eqnarray*}%%----------------------------%%
Denote by $f_\delta^\Delta:= f - f_\delta$ the difference between a limiting function and the function itself, for example $\theta_\delta^\Delta :=\theta - \theta_\delta$ or $(J_\delta r_\delta \bb{U}_\delta )_\delta^\Delta : = Jr\bb{U} - J_\delta r_\delta \bb{U}_\delta$.
Taking the difference of the heat equation $\eqref{qq22thermalweak}$ and $\eqref{qq22heat3}$ satisfied by $\theta$ and $\theta_\delta$, respectively, and applying $\Delta^{-1}$, one obtains
\begin{eqnarray*}%%%%%%%%%%%%%
\partial_t \Delta^{-1}\theta_\delta^\Delta =\theta_\delta^\Delta + \partial_t w_\delta^\Delta\in L^2(\Gamma_T).
\end{eqnarray*}%%----------------------------%%
We now take the difference of the equation $\eqref{qq22almostthere}$ with $(\bb{q},\psi) = (R[\Delta^{-1}\theta_\delta^\Delta], \Delta^{-1}\theta_\delta^\Delta)$ and the equation
$\eqref{qq22momentumeps}$ with $(\bb{q},\psi) =(R_\delta[\Delta^{-1}\theta_\delta^\Delta], \Delta^{-1}\theta_\delta^\Delta)$ to obtain
\begin{align}%%%%%%%%%%%%%%%
&||\partial_t w_\delta^\Delta||_{L^2(\Gamma_T)}^2  \nonumber\\
&= \int_{Q_{T}}(-J{r} \bb{U})_\delta^\Delta \cdot \big(R[\partial_t w_\delta^\Delta] +  \int_{Q_{T}}R[\theta_\delta^\Delta] \big) +\big[J({r} \bb{U}-{r} \bb{w}) \otimes \bb{U}\big]_\delta^\Delta: \nabla^{w}R[\Delta^{-1}\theta_\delta^\Delta]\nonumber \\
&\quad+ \int_{Q_{T}}\mu (J\nabla^{w} (\mathbf{u}))_\delta^\Delta :\nabla^{w}R[\Delta^{-1}\theta_\delta^\Delta] + \int_{Q_{T}}(\mu +\lambda) (J \nabla^{w}\cdot \mathbf{u})_\delta^\Delta (\nabla^{w}\cdot R[\Delta^{-1}\theta_\delta^\Delta]) \nonumber \\
&\quad-  \int_{Q_{T}}( (J{r}^\gamma)_\delta^\Delta + \delta {r}_\delta^a) (\nabla^{w}\cdot R[\Delta^{-1}\theta_\delta^\Delta])+\bint_{\Gamma_T} \Delta w_\delta^\Delta \theta_\delta^\Delta \nonumber \\
&\quad+\bint_{\Gamma_T}(G - \Delta (\Delta w_\delta)^3)\Delta^{-1}[\theta_\delta^\Delta] +\bint_{\Gamma_T} (\theta_\delta^\Delta)^2+\bint_{\Gamma_T}\delta \nabla^3 w_\delta : \nabla^3 (\Delta^{-1}\theta_\delta^\Delta ) \nonumber \\
&\quad-\int_0^T \frac{d}{dt} \int_{\Omega}(J {r} \bb{U}\cdot R[\Delta^{-1}\theta_\delta^\Delta])-\int_0^T \frac{d}{dt} \int_{\Gamma} \partial_t w \Delta^{-1}\theta_\delta^\Delta.\nonumber 
\end{align}
From the heat equation, one can infer that $\theta_\delta^\Delta \to 0$ in $L^2(0,T; L^{2}(\Gamma))$ and $\Delta^{-1}[\theta_\delta^\Delta] \to 0$ in $L^2(0,T; H^{3^-}(\Gamma))$ and in, say, $C([0,T]; L^{4}(\Gamma))$. Using the convergences from Lemma $\ref{qq22weakconvdel}$ one can obtain that almost all the terms on the right-hand side converge to zero (for the term including $G$ we can use the weak$^*$ convergence of $\Delta(\Delta w_\delta)^3$ to $G$ in \\ $L^\infty(0,T; [W^{2,4}(\Gamma)]')$). The only one that requires special attention is 
\begin{eqnarray*}%%%%%%%%%%%%%
I:= \int_{Q_T} (-J{r} \bb{U})_\delta^\Delta \cdot R[\partial_t w_\delta^\Delta].
\end{eqnarray*}%%----------------------------%%
Since $R[\partial_t w_\delta^\Delta]$ converges in the same space as $\bb{w}_\delta^\Delta$, one can use the same idea as the proof of Lemma $\ref{qq22weakconvthirdlevel}(viii)$ so $(i)$ follows.

Now, we aim to prove the claim $(ii)$. Following the ideas from \cite{lasiecka} (or \cite[Lemma 4.1]{trwa2} in the context of fluid-structure interaction), we will prove the following inverse type inequality
\begin{eqnarray*}%%%%%%%%%%%%%
\int_0^T\limsup\limits_{\delta \to 0}\langle G - \mathcal{F}(w_\delta), w - w_\delta \rangle \leq 0.
\end{eqnarray*}%%----------------------------%%
which will, by maximal monotonicity property of $\mathcal{F}$ and \cite[Proposition 1.2.6]{karmanplates}, give us the desired result. For this reason, by taking the difference of the equation $\eqref{qq22almostthere}$ with $(\bb{q},\psi) = (R[w_\delta^\Delta], w_\delta^\Delta)$ and the equation
$\eqref{qq22momentumeps}$ with $(\bb{q},\psi) =(R_\delta[w_\delta^\Delta], w_\delta^\Delta)$, one obtains
\begin{eqnarray*}%%%%%%%%%%%%%%%
&&||\Delta w_\delta^\Delta||_{L^2(\Gamma_T)}^2+ \delta || \nabla^3 w_\delta^\Delta||_{L^2(\Gamma_T)}^2 +\langle G - \mathcal{F}(w_\delta), w - w_\delta \rangle \\
&&=- \int_{Q_{T}}\Big[ (J{r} \bb{U})_\delta^\Delta \cdot R[\partial_t w_\delta^\Delta] -\int_{Q_{T}}\big[J({r} \bb{U}-{r} \bb{w}) \otimes \bb{U}\big]_\delta^\Delta :\nabla^{w}R[ w_\delta^{\Delta}]\nonumber \\
&&\quad-\mu \int_{Q_{T}}(J\nabla^{w} (\mathbf{u}))_\delta^\Delta :\nabla^{w}R[ w_\delta^{\Delta}] -(\mu +\lambda)\int_{Q_{T}} (J \nabla^{w}\cdot \mathbf{u})_\delta^\Delta \nabla^{w}\cdot R[ w_\delta^{\Delta}] \nonumber \\
&&\quad+\int_{Q_{T}} ( (J{r}^\gamma)_\delta^\Delta + \delta {r}_\delta^a) \nabla^{w}\cdot R[ w_\delta^{\Delta}])-\bint_{\Gamma_T}(\partial_t w_\delta^\Delta)^2 - \bint_{\Gamma_T} \nabla w_\delta^{\Delta}\cdot \nabla \theta_\delta^{\Delta} \nonumber \\
&&\quad-\int_0^T \frac{d}{dt} \int_{\Omega}(J {r} \bb{U} \cdot R[ w_\delta^{\Delta}])-\int_0^T \frac{d}{dt} \int_{\Gamma} \partial_t w w_\delta^{\Delta}.\nonumber 
\end{eqnarray*}
Now, since $w_\delta^\Delta \to 0$ in $L^2(0,T; W^{2^-,4}(\Gamma))$ and $\partial_t w_\delta^\Delta \to 0$ in $L^2(\Gamma_T)$ by $(i)$ and the uniform estimates given in Lemma $\ref{qq22111}$, the right-hand converges to zero as $\delta \to 0$, so the proof is finished.
\end{proof}

\section*{Appendix B: Proof of Lemma $\ref{qq22lemkappa}$}
\sectionmark{Proof of Lemma $\ref{qq22lemkappa}$}
\addcontentsline{toc}{section}{\numberline{}{Appendix B: Proof of Lemma $\ref{qq22lemkappa}$}}
\setcounter{equation}{0}
For a $\kappa>0$, we will construct a set $\mathcal{A}_\kappa = \mathcal{A}_\kappa^1 \cup \mathcal{A}_\kappa^2$ where $\mathcal{A}_\kappa^1 \Subset (0,T)\times \Omega^{w_\varepsilon}(t)$ and $\mathcal{A}_\kappa^2 \Subset (0,T)\times (\Omega_\delta^{w_\varepsilon}(t) \setminus \Omega^{w_\varepsilon}(t))$, and prove
\begin{eqnarray}%%%%%%%%%%%%%
\int_{((0,T)\times \Omega^{w_{\varepsilon}}(t)) \setminus A_\kappa^1} (\rho_\varepsilon^{\gamma} + \delta \rho_\varepsilon^{a}) \leq \frac{\kappa}{2}, \label{qq22set1} \\
\int_{((0,T)\times (\Omega_\delta^{w_{\varepsilon}}(t)\setminus \Omega^w(t)))) \setminus A_\kappa^2} (\rho_\varepsilon^{\gamma} + \delta \rho_\varepsilon^{a}) \leq \frac{\kappa}{2}, \label{qq22set2}
\end{eqnarray}%%----------------------------%%
by constructing a test function that has an arbitrarily large positive divergence in a thin layer near the boundary, and a bounded $W^{1,\infty^-}$ spatial norm away from the boundary. \\

\noindent
\textit{Step 1: Proof of }$\eqref{qq22lemkappa}$. To follow the proof more easily, it is helpful to see Figure $\ref{qq22slicica}$. Now, since $\Gamma$ is a 2D Lipschitz domain, it is also $W^{1,\infty}$ (see \cite[Section 5.8, Theorem 4]{evans}) and can be represented as a union of star-shaped domain (see \cite[Proposition 2.5.4]{Lip}). Let then, for simplicity, $\Gamma$ be a star-shaped domain, and since all star-shaped domain are isomorphic to a disc, let then $\Gamma = \{ |X| < R \}$ for some $R>0$.\footnote{Even though we assumed that $\Gamma = \{ |X| < R \} $, we will only use the $W^{1,\infty}$ regularity of $\Gamma$ in the proof.}
Denote by $(r,\alpha)$ the polar coordinates in the $\{z=0\}$ plane, and let $m,M \in C_0^\infty(\Gamma)$ satisfying $m(X) = \tilde{m}(r)$ and $M(X)=\tilde{M}(r)$ for some $\tilde{m},\tilde{M}\in C^\infty[0,R]$ such that $\tilde{m}(R) = \tilde{M}(R)=0$ and $m(X)<w_\varepsilon(t,X)<M(X)$, for all $t \in [0,T],X\in\Gamma ,\varepsilon>0$ (existence of such functions is ensured by the uniform estimates of $w_\varepsilon$ in $L^\infty(0,T; H_0^2(\Gamma))\cap W^{1,\infty}(0,T; L^2(\Gamma))\hookrightarrow C^{0,\beta}(0,T; C^{0,1-2\beta}(\Gamma))$, for $0<\beta<1/2$). Also, for simplicity, we will assume that $m(X) \geq |X| -R$.

\begin{figure}[h!]

\centering\includegraphics[scale = 0.14]{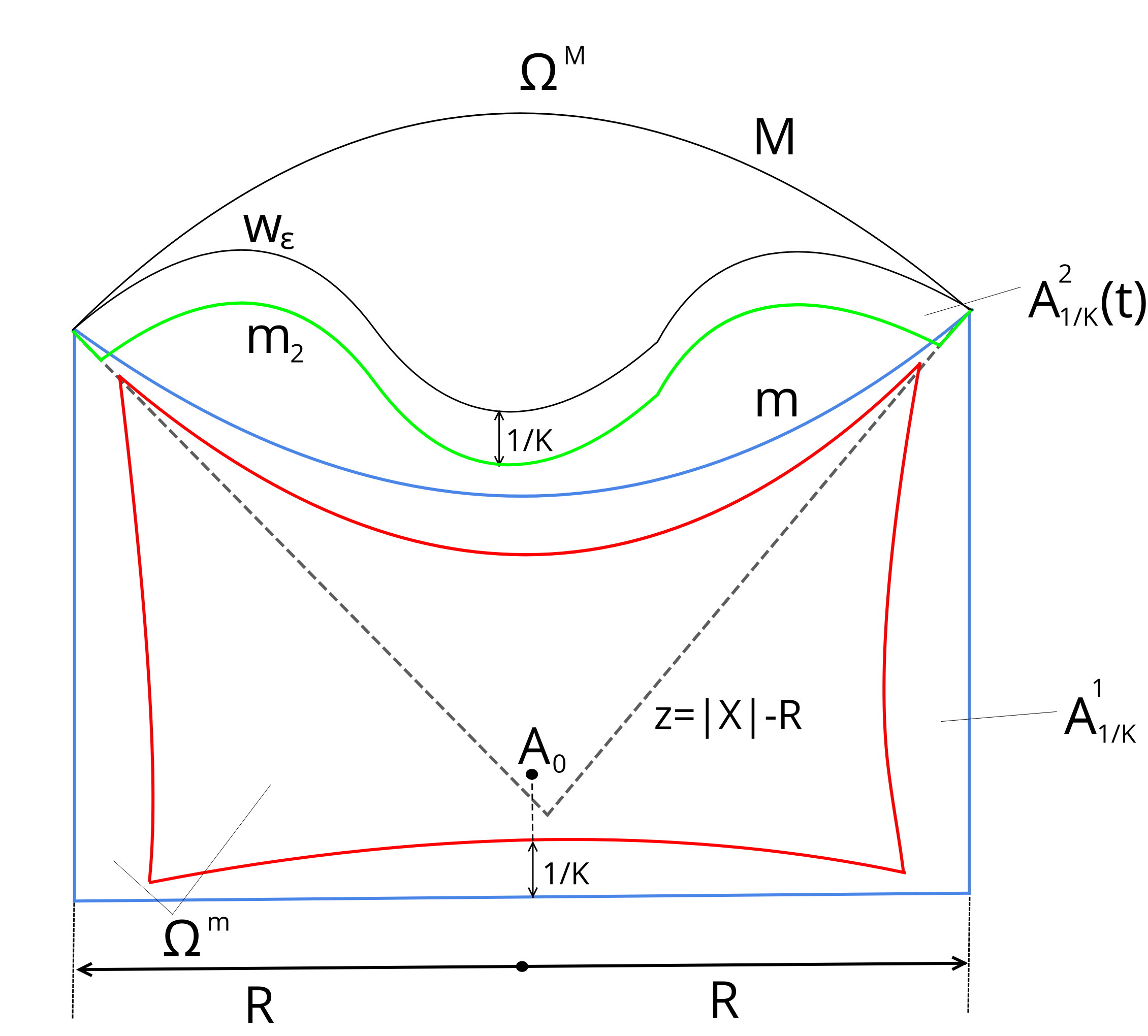}

\caption{The sets $\Omega^M, \Omega^m, A_{1/K}^1(t)$ and $A_{1/K}^2$ on a vertical section at time $t$.}
\label{qq22slicica}
\end{figure}

Denote by $\Omega^M:= \{(X,z): X\in \Gamma, -1<z<M(X)\}$ and $\Omega^m:= \{(X,z): X\in \Gamma, -1<z<m(X)\}$. Obviously, $\Omega^m$ is of $W^{1,\infty}$ and Lipschitz regularity, so we can w.l.o.g. assume that $\Omega^m$ is star-shaped around the point $A_0$ so that $\partial \Omega^m= \{ A_0 + g(\alpha,\beta)\bb{r}(\alpha,\beta),\text{ for } \alpha,\beta \in [0,\pi]\times[0,2\pi] \}$ for a $C^{0,1}\cap W^{1,\infty}$ function $g$, where $\alpha$ and $\beta$ correspond to the spherical coordinates (polar and azimuthal angles) and $\bb{r}(\alpha,\beta)$ is the radial outward unit vector in the direction $(\alpha,\beta)$. We will write $r_{A_0}(X,z)$, $\alpha_{A_0}(X,z)$ and $\beta_{A_0}(X,z)$ to denote the radial coordinate, the polar and azimuthal angles, respectively, of the point $(X,z)-A_0$. Let 
\begin{align*}%%%%%%%%%%%%%
A_{1/K}^1:=\big\{ (X,z)\in \Omega^m:& g(\alpha_{A_0}(X,z),\beta_{A_0}(X,z))-1/K < r_{A_0}(X,z)\\
&<g(\alpha_{A_0}(X,z),\beta_{A_0}(X,z))\big\}
\end{align*}%%----------------------------%%
be the $1/K$ layer set near the boundary $\partial\Omega^m$ and (see Figure $\ref{qq22slicica2}$)
\begin{eqnarray*}%%%%%%%%%%%%%
\bb{q}_K^1:=\begin{cases}
K\big[r_{A_0}-g(\alpha_{A_0},\beta_{A_0})\big] \bb{r}(\alpha_{A_0},\beta_{A_0}), \quad &\mbox{on } A_{1/K}^1, \\
-\bb{r}(\alpha_{A_0},\beta_{A_0}), \quad &\mbox{in } \Omega^m\setminus (A_{1/K}^1\cup \{A_0\}) ,\\
0, \quad &\mbox{on } (\Omega^{w_\varepsilon}\setminus \Omega^m) \cup \{A_0\}.
\end{cases} 
\end{eqnarray*}%%----------------------------%%
\begin{figure}[h!]

\centering\includegraphics[scale = 0.14]{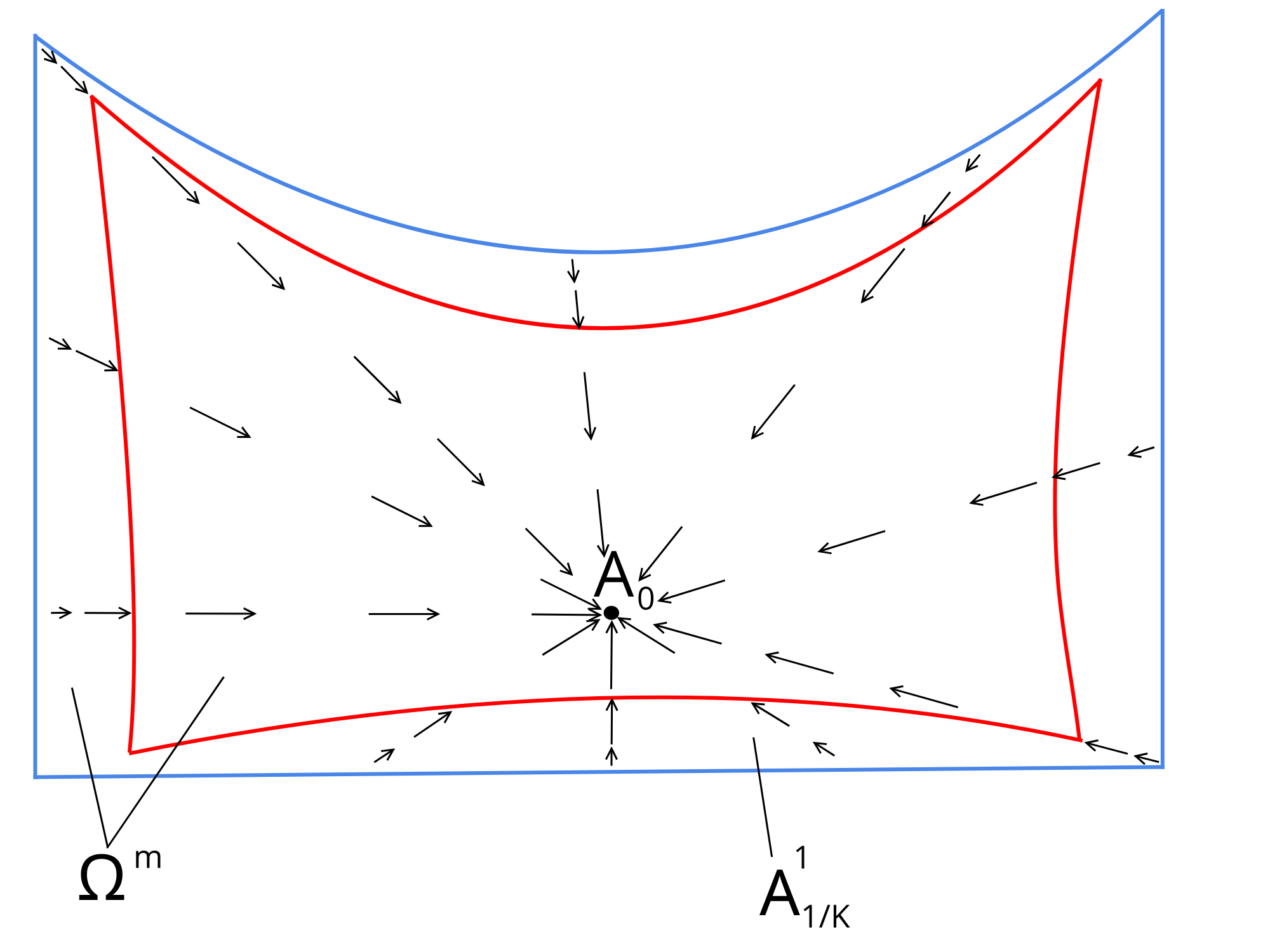}

\caption{The function $\bb{q}_K^1$ on a vertical section.}
\label{qq22slicica2}
\end{figure}

Next, we introduce 
\begin{eqnarray*}%%%%%%%%%%%%%
\bb{q}_K^2(t,X,z):=\begin{cases}
\bb{q}_K^{w_\varepsilon}(t,X,z), &\mbox{for } (X,z) \in A_{1/K}^2(t), \\
\bb{q}_K^{w_\varepsilon}(t,X, m_2(X)), &\mbox{elsewhere in } \Omega^{w_\varepsilon}(t),
\end{cases}
\end{eqnarray*}%%----------------------------%%
where
\begin{eqnarray*}%%%%%%%%%%%%%
\bb{q}_K^{w_\varepsilon}(t,X,z):= -K (w_{\varepsilon}(t,X)-z)\bb{e}_3, 
\end{eqnarray*}%%----------------------------%%
is defined on the set $A_{1/K}^2(t):= \{ (X,z): X\in \Gamma, m_2(t,X)< z<w_\varepsilon(t,X) \}$ for $t\in[0,T]$, with
\begin{eqnarray*}%%%%%%%%%%%%%
m_2(t,X):= \max\{ |X|-R, w_\varepsilon(t,X)-1/K \} .
\end{eqnarray*}%%----------------------------%
Now, let $\varphi_1,\varphi_2:\Omega^M \to [0,1]$ be smooth functions such that $|\nabla \varphi_1|,|\nabla \varphi_2| \leq C$ and:
\begin{enumerate}
\item $\varphi_1= 1$ on $A_{1/K}^1$ and $\varphi_1= 0$ in a small ball around $A_0$.
\item $\varphi_2=1$ on $\cup_{t \in [0,T]}A_{1/K}^2(t)$ and $\varphi_2 = 0$ for $-1\leq z\leq c$, for some \\$-1< c \leq \min\limits_{t \in [0,T], X\in\Gamma}m_2$;
\end{enumerate}
Finally, let
\begin{eqnarray*}%%%%%%%%%%%%%
\bb{q}_K:= \varphi_1\bb{q}_K^1 + \varphi_2 \bb{q}_K^2.
\end{eqnarray*}%%----------------------------%%
Now, to study the properties of $\bb{q}_K$, it will be useful to introduce the sets
\begin{eqnarray*}%%%%%%%%%%%%%
B_{1/K}^1(t):= \Omega^{w_\varepsilon}(t) \setminus A_{1/K}^1(t),& \quad B_{1/K}^2(t):= \Omega^{w_\varepsilon}(t) \setminus A_{1/K}^2(t)\\
A_{1/K}(t) := A_{1/K}^1 \cup A_{1/K}^2(t), & \quad B_{1/K}(t):= \Omega^{w_\varepsilon}(t) \setminus A_{1/K}(t).
\end{eqnarray*}%%----------------------------%%
The following hold:
\begin{enumerate}
\item We have that $\nabla \cdot \bb{q}_K^2 = K$ on $[0,T]\times A_{1/K}^2(t)$, and for $\bb{q}_K^1$, we can calculate on $A_{1/K}^1$
\begin{eqnarray*}%%%%%%%%%%%%%
\nabla \cdot \bb{q}_K^1 &=& \nabla \cdot \Big[ K\big[r_{A_0}-g(\alpha_{A_0},\beta_{A_0})\big] \bb{r}(\alpha_{A_0},\beta_{A_0})\Big] \\
&=&K \underbrace{\nabla r_{A_0} \cdot \bb{r}(\alpha_{A_0},\beta_{A_0})}_{=1} 
-K \underbrace{\nabla g(\alpha_{A_0},\beta_{A_0}) \cdot \bb{r}(\alpha_{A_0},\beta_{A_0})}_{=0}\\
&&+K \underbrace{\big[ r(\alpha_{A_0},\beta_{A_0}) -g(\alpha_{A_0},\beta_{A_0}) \big]}_{\geq -1/K \text{ and } \leq 0} \underbrace{\nabla \cdot \bb{r}(\alpha_{A_0},\beta_{A_0})}_{=1}\\
&\geq & K - 1,
\end{eqnarray*}%%----------------------------%%
since $\nabla r_{A_0} = \bb{r}(\alpha_{A_0},\beta_{A_0})$ and $\nabla g \perp \bb{r}(\alpha_{A_0},\beta_{A_0})$. Now, because $|\nabla \varphi_1|,|\nabla \varphi_2| \leq C$, its easy to see that $||\varphi_1 \bb{q}_K^1||_{L^\infty(0,T; W^{1,\infty}(B_{1/K}^1(t)))} \leq C$, since we have excluded the singularity at $A_0$ with $\varphi_1$, and $||\varphi_2 \bb{q}_K^2||_{L^\infty(0,T; W^{1,\infty^-}(B_{1/K}^2(t)))} \leq C$. Thus, we can conclude the following:
\begin{eqnarray}%%%%%%%%%%%%%
\nabla \cdot \bb{q}_K &\geq& K-c, \quad \text{on }[0,T]\times A_{1/K}(t), \label{qq22lowdivbound} \\
\nabla \cdot \bb{q}_K &\leq& C, \quad \text{on }[0,T]\times B_{1/K}(t), \label{qq22updivbound} \\
|| \bb{q}_K||_{L^\infty(0,T; W^{1,\infty^-}(A_{1/K}(t)))} &\leq& C(K+1), \label{qq22nablabound1}\\
|| \bb{q}_K||_{L^\infty(0,T; W^{1,\infty^-}(B_{1/K}(t)))} &\leq& C. \label{qq22nablabound2}
\end{eqnarray}%%----------------------------%%
\item The time derivative has the following form
\begin{eqnarray*}%%%%%%%%%%%%%
\partial_t \bb{q}_K = \begin{cases}
\varphi_2 \partial_t \bb{q}_K^2 , \quad \text{ on } [0,T]\times A_{1/K}^2(t) \\
0, \quad \text{ elsewhere in } \Omega^w(t),
\end{cases}
\end{eqnarray*}%%----------------------------%%
which easily implies that
\begin{eqnarray}\label{qq22sometimebound}%%%%%%%%%%%%%
||\partial_t \bb{q}_K||_{L^\infty(0,T; L^{4^-}( A_{1/K}(t) )} \leq CK ||\partial_t w_\varepsilon||_{L^\infty(0,T; L^{4^-}( \Gamma))} \leq C K
\end{eqnarray}%%----------------------------%%
by Lemma $\ref{qq22111}(iv)$ and imbedding of Sobolev spaces.
\end{enumerate}

\noindent
Now, by choosing $(\bb{q},\psi)=(\bb{q}_K,0)$ in the equation $\eqref{qq22movingdomainsystem}_1$ (by the density argument), we have
\begin{eqnarray}%%%%%%%%%%%%%
&&\int_{(0,T)\times A_{1/K}(t) } (\rho_\varepsilon^{\gamma} + \delta \rho_\varepsilon^{a})( \nabla \cdot \bb{q}_K) \nonumber \\
&&\leq -\int_{(0,T)\times B_{1/K}(t) } (\rho_\varepsilon^{\gamma} + \delta \rho_\varepsilon^{a})( \nabla \cdot \bb{q}_K)+ \int_{(0,T)}\frac{d}{dt} \int_{\Omega^{w_\varepsilon}} \rho_\varepsilon \bb{u}_\varepsilon \cdot\bb{q}_K 
 -\int_{Q_{T}^w} \rho_\varepsilon \bb{u}_\varepsilon \cdot \partial_t \bb{q}_K \nonumber\\
 &&\quad+\int_{Q_{T}^w}\big[\rho \bb{u} \otimes \bb{u}\big]:\nabla\bb{q}_K+\mu \int_{Q_{T}^w} \nabla \bb{u}:\nabla\bb{q}_K+(\mu +\lambda)\int_{Q_{T}^w} (\nabla\cdot \bb{u})(\nabla\cdot \bb{q}_K) 
\nonumber \\[2mm]
&&\quad +\varepsilon \bint_{Q_{T}^w} \nabla^{w^{-1}} \rho \cdot ( \bb{q}_K\cdot\nabla^{w^{-1}} \bb{u}
+ \bb{u} \cdot\nabla^{w^{-1}} \bb{q}_K)=: I_1+...+I_7. \label{qq22ajaja}
\end{eqnarray}%%----------------------------%%
Let us first estimate the ``worst term''
\begin{eqnarray*}%%%%%%%%%%%%%
I_4 &\leq& \max\limits_{t\in[0,T]}\mathcal{M}(A_{1/K}(t))^{\frac{1}{p}}||\rho_\varepsilon \bb{u}_\varepsilon \otimes \bb{u}_\varepsilon||_{L^1(I_K; L^{\frac{3\gamma}{\gamma+3}}(A_{1/K}(t)))} || \nabla \bb{q}_K||_{L^\infty(0,T; L^{q}(A_{1/K}(t)))}\\
&&+C||\rho_\varepsilon \bb{u}_\varepsilon \otimes \bb{u}_\varepsilon||_{L^1(I_K; L^{\frac{6\gamma}{4\gamma+3}}(B_{1/K}(t)))} || \nabla \bb{q}_K||_{L^\infty(0,T; L^{q}(B_{1/K}(t)))} \\
&\leq& CK^{-1/{p}} (K+1)+C,
\end{eqnarray*}%%----------------------------%%
for $p,q \in (1,\infty)$ such that $\frac{1}{p} + \frac{\gamma+3}{3\gamma} +\frac{1}{q} = 1$, by $\eqref{qq22nablabound1},\eqref{qq22nablabound2}$ and the uniform energy estimates, where, we have used the fact that $\max\limits_{t\in[0,T]}\mathcal{M}(A_{1/K}(t)) \leq C/K$.  Now, the terms $I_1,I_2,I_4,I_5$ can be estimated in a similar way based on $\eqref{qq22nablabound1}$ and $\eqref{qq22nablabound2}$, where for $I_1$ we use $\eqref{qq22updivbound}$ and for $I_3$ we use $\eqref{qq22sometimebound}$. The $\varepsilon$ term can be estimated as
\begin{align*}%%%%%%%%%%%%%
I_6 &\leq\max\limits_{t\in[0,T]}\mathcal{M}(A_{1/K}(t))^{\frac{1}{18}} \varepsilon||\nabla r_\varepsilon||_{L^3(0,T; L^{\frac{9}{4}}(A_{1/K}(t)))}\times \\
&\quad\quad\Big[|| \bb{U} ||_{L^2(0,T; L^{2}(A_{1/K}(t)))} || \nabla \bb{q}_K||_{L^\infty(0,T; L^{q}(A_{1/K}(t)))}\\
&\quad\quad+ ||\nabla \bb{U} ||_{L^2(0,T; L^{2}(A_{1/K}(t)))} || \bb{q}_K||_{L^\infty(0,T; L^{q}(A_{1/K}(t)))} \Big]\\
&\quad\quad+C\varepsilon||\nabla r_\varepsilon||_{L^3(0,T; L^{\frac{9}{4}}(B_{1/K}(t))))}\Big[|| \bb{U} ||_{L^2(0,T; L^{2}(B_{1/K}(t)))} || \nabla \bb{q}_K||_{L^\infty(0,T; L^{q}(B_{1/K}(t)))}\\
&\quad\quad+ ||\nabla \bb{U} ||_{L^2(0,T; L^{2}(B_{1/K}(t)))} || \bb{q}_K||_{L^\infty(0,T; L^{q}(B_{1/K}(t)))} \Big]\\
& \leq CK^{-\frac{1}{18}}(K+1)+C,
\end{align*}%%----------------------------%% 
by $\eqref{qq22nablabound1}$ and $\eqref{qq22nablabound2}$. Combining the previous estimates and $\eqref{qq22lowdivbound}$, from $\eqref{qq22ajaja}$ we have
\begin{eqnarray}%%%%%%%%%%%%%
\int_{(0,T)\times (A_{1/K}(t) )} (\rho_\varepsilon^{\gamma} + \delta \rho_\varepsilon^{a}) &<& \frac{1}{K-c}\int_{(0,T)\times A_{1/K}(t) } (\rho_\varepsilon^{\gamma} + \delta \rho_\varepsilon^{a})\chi_{\Omega^{w_\varepsilon}(t)} (\nabla \cdot \bb{q}_K) \nonumber\\
&\leq&\frac{C}{K-c} \frac{K+1}{(CK)^{1/q}}+\frac{C}{K-c} \leq \frac{\kappa}{4}, \label{qq22tadaa}
\end{eqnarray}%%----------------------------%%
for some (large) $q\in(1,\infty)$ and $K>0$ large enough such that 
\begin{eqnarray*}%%%%%%%%%%%%%
\frac{C}{K-c} \frac{K+1}{(CK)^{1/q}}+\frac{C}{K-c} \leq \frac{\kappa}{4}.
\end{eqnarray*}%%----------------------------%%
Since $|| \rho_\varepsilon^{\gamma} + \delta \rho_\varepsilon^{a}||_{L^\infty(0,T;L^1(\Omega^{w_\varepsilon}(t)))} \leq C$, we have
\begin{eqnarray*}%%%%%%%%%%%%%
|| \rho_\varepsilon^{\gamma} + \delta \rho_\varepsilon^{a}||_{L^1( (0,s) \cup (T-s, T);L^1(\Omega^{w_\varepsilon}(t)) )} \leq 2Cs,
\end{eqnarray*}%%----------------------------%%
so for $s<\kappa /(8C)$, the inequality $\eqref{qq22set1}$ holds for $\mathcal{A}_\kappa^1 = (s,T-s)\times B_{1/K}(t)$, by $\eqref{qq22tadaa}$.\\

\textit{Step 2: Proof of }$\eqref{qq22set2}$. For this part, we can decompose the set $\Omega_\delta^{w_\varepsilon}\setminus \Omega^{w_\varepsilon}$ into a union of intersecting star-shaped domains and then construct a test function on each of those sub-domains in the same way as that of $\bb{q}_K^1$. Then, we can sum up these functions by means of partition of unity and use the same ideas as in Step 1 to obtain the inequality $\eqref{qq22set2}$.

\vspace{.1in}
\noindent{\bf Acknowledgments:} This research was partially supported by National Natural Science Foundation of China (NNSFC) under Grant No. 11631008. The author S.T. would like to express his graditude to Prof. \v{S}\'{a}rka Ne\v{c}asov\'{a} for valuable comments and discussions during his visit at the Institute of Mathematics of the Czech Academy of Sciences, which improved the quality and clarity of this paper.

\end{document}